\documentclass[9pt,a4paper]{article}
\usepackage[T1]{fontenc}
\usepackage[francais,english]{babel}
\usepackage[applemac]{inputenc}
\usepackage{pb-diagram}
\usepackage{enumitem}
\usepackage[vmargin=3.5cm]{geometry}%
\usepackage{color}%
\usepackage{mdwlist}%
\usepackage{graphicx}%
\usepackage{float}%
\usepackage{makeidx}%
\usepackage{amsmath}%
\usepackage{amssymb}%
\usepackage{amsthm}%
\usepackage{amsfonts}%
\usepackage{mathrsfs}%
\usepackage{bbm}%
\usepackage[all,cmtip]{xy}%
\usepackage{rotate}%
\usepackage{yfonts}%
\usepackage{array}%
\newtheorem{Prop}{Proposition}[section]%
\newtheorem{Conj}[Prop]{Conjecture}%
\newtheorem{TheoEnglish}[Prop]{Theorem}%
\newtheorem{DefEnglish}[Prop]{Definition}

\newtheorem{CorEnglish}[Prop]{Corollary}
\newtheorem{HypEnglish}[Prop]{Assumption}%
\newcommand{\eqdef}{\overset{\operatorname{def}}{=}}
\newcommand{\A}{\mathbb A}%
\newcommand{\C}{\mathbb C}%
\newcommand{\Fp}{\mathbb F}%
\newcommand{\G}{\mathbf G}%

\newcommand{\kbold}{\mathbf k}
\newcommand{\Q}{\mathbb Q}%
\newcommand{\qp}{\mathbb Q_{p}}%
\newcommand{\R}{\mathbb R}%
\newcommand{\Z}{\mathbb Z}%
\newcommand{\zp}{\mathbb Z_{p}}%
\newcommand{\Acal}{\mathcal A}%
\newcommand{\Bcal}{\mathcal B}
\newcommand{\Ccal}{\mathcal C}%
\newcommand{\Fcal}{\mathcal F}%
\newcommand{\Lcal}{\mathcal L}%
\newcommand{\Ncal}{\mathcal N}%
\newcommand{\Ocal}{\mathcal O}%
\newcommand{\Kcal}{\mathcal K}%
\newcommand{\Rcal}{\mathcal R}%
\newcommand{\Wcal}{\mathcal W}%
\newcommand{\Zcal}{\mathcal Z}%
\newcommand{\Fcali}{\mathscr F}%

\newcommand{\Lcali}{\mathscr L}%
\newcommand{\Xcali}{\mathscr X}%
\newcommand{\Ycali}{\mathscr Y}%
\newcommand{\aid}{\mathfrak a}%
\newcommand{\bid}{\mathfrak b}%
\newcommand{\cid}{\mathfrak c}%
\newcommand{\qid}{\mathfrak q}%
\newcommand{\pid}{\mathfrak p}%
\newcommand{\mgot}{\mathfrak m}%
\newcommand{\zid}{\mathfrak z}%
\newcommand{\Lambdaf}{\mathbf{\Lambda}}
\newcommand{\Tbar}{\overline{T}}%
\newcommand{\GL}{\operatorname{GL}}%

\newcommand{\GSp}{\operatorname{GSp}}%

\newcommand{\et}{\operatorname{et}}

\newcommand{\produit}[2]{\underset{#1}{\overset{#2}\prod}}%
\newcommand{\produittenseur}[2]{\underset{#1}{\overset{#2}\bigotimes}}%
\newcommand{\sommedirecte}[2]{\underset{#1}{\overset{#2}\bigoplus}}%
\newcommand{\suiteexacte}[5]{0\fleche#3\overset{#1}{\fleche}#4\overset{#2}{\fleche}#5\fleche0}
\newcommand{\cardinal}[1]{|#1|}

\newcommand{\limproj}[1]{\underset{\underset{#1}\longleftarrow}\lim}
\newcommand{\liminj}[1]{\underset{\underset{#1}\longrightarrow}\lim}

\newcommand{\Hom}{\operatorname{Hom}}%

\newcommand{\cor}{\operatorname{Cor}}
\newcommand{\isom}{\overset{\sim}{\longrightarrow}}

\newcommand{\plonge}{\hookrightarrow}
\newcommand{\matrice}[4]{\begin{pmatrix}#1&#2\\ #3&#4\end{pmatrix}}

\newcommand{\rank}{\operatorname{rank}}%
\newcommand{\ord}{\operatorname{ord}}

\newcommand{\tenseur}{\otimes}
\newcommand{\Ltenseur}{\overset{\operatorname{L}}{\tenseur}}
\newcommand{\modulo}{\operatorname{ mod }}
\newcommand{\Spec}{\operatorname{Spec}}
\newcommand{\Id}{\operatorname{Id}}%
\newcommand{\Aut}{\operatorname{Aut}}%
\newcommand{\End}{\operatorname{End}}%
\newcommand{\Frac}{\operatorname{Frac}}%
\newcommand{\Tate}{\operatorname{Ta}}%
\newcommand{\Ext}{\operatorname{Ext}}%
\newcommand{\nInd}{\operatorname{n-Ind}}%
\newcommand{\fleche}{\longrightarrow}%
\newcommand{\croix}{^{\times}}%
\newcommand{\surjection}{\twoheadrightarrow}%
\newcommand{\rhobar}{\bar{\rho}}%
\newcommand{\dR}{\operatorname{dR}}%
\newcommand{\Fil}{\operatorname{Fil}}
\newcommand{\Gr}{\operatorname{Gr}}
\newcommand{\Deltatilde}{\tilde{\Delta}}

\newcommand{\Hun}{H^{1}}

\newcommand{\Htilde}{\tilde{H}}
\newcommand{\Htildetilde}{\tilde{\tilde{H}}}

\newcommand{\RGamma}{\operatorname{R}\Gamma}%
\newcommand{\Det}{\operatorname{{D}et}}%
\newcommand{\tr}{\operatorname{tr}}
\newcommand{\Fr}{\operatorname{Fr}}%
\newcommand{\Gal}{\operatorname{Gal}}

\newcommand{\Qbar}{\bar{\Q}}%
\newcommand{\Fbar}{\bar{F}}%
\newcommand{\Kbar}{\bar{K}}%
\newcommand{\Dbar}{\bar{D}}
\newcommand{\kbar}{\bar{k}}%
\newcommand{\s}{\sigma}%
\newcommand{\hgot}{\mathfrak h}%
\newcommand{\Hecke}{\mathbf{T}}%
\newcommand{\Hs}{\Hecke^{\Sigma}}%
\newcommand{\Hsm}{\Hs_{\mgot}}
\newcommand{\Hsmo}{\Hecke^{\Sigma_{0}}_{\mgot}}
\newcommand{\Hsmord}{\Hecke^{\Sigma,\ord}_{\mgot}}
\newcommand{\Raid}{R(\aid)}

\newcommand{\Eul}{\operatorname{Eul}}
\newcommand{\Ssplit}{\Sigma^{\operatorname{split}}}

\DeclareFontEncoding{OT2}{}{} 
\newcommand{\Nekovar}{Nekov\'a\v{r}}%
\newcommand{\cl}{\operatorname{cl}}
\newcommand{\mot}{\operatorname{mot}}
\newcommand{\aut}{\operatorname{aut}}
\numberwithin{equation}{subsection}%
\date{}
\usepackage{titling}
\newcommand{\subtitle}[1]{%
  \posttitle{%
    \par\end{center}
    \begin{center}\large#1\end{center}
    \vskip0.5em}%
}

\begin{document}%
\title{$p$-adic properties of motivic fundamental lines}
\subtitle{Kato's general conjecture is (probably) false for (not so) trivial reasons}
\author{Olivier Fouquet}%
\maketitle
\begin{abstract}
We prove the conjectured compatibility of $p$-adic fundamental lines with specializations at motivic points for a wide class of $p$-adic families of $p$-adic Galois representations (for instance, the families which arise from $p$-adic families of automorphic representations of the unit group of a quaternion algebra or of a totally definite unitary groups) and deduce the compatibility of the Equivariant Tamagawa Number Conjectures for them. However, we also show that fundamental lines are not compatible with arbitrary characteristic zero specializations with values in a domain in general. This points to the need to modify the conjectures of \cite{KatoViaBdR} using completed cohomology.
\end{abstract}
%
\selectlanguage{english}%

\newcommand{\hord}{\mathfrak h^{\ord}}%
\newcommand{\hdual}{\mathfrak h^{dual}}%
\newcommand{\matricetype}{\begin{pmatrix}\ a&b\\ c&d\end{pmatrix}}%
\newcommand{\Iw}{\operatorname{Iw}}%
\newcommand{\Hi}{\operatorname{Hi}}
\newcommand{\cyc}{\operatorname{cyc}}
\newcommand{\ab}{\operatorname{ab}}
\newcommand{\can}{\operatorname{can}}%
\newcommand{\Fitt}{\operatorname{Fitt}}%
\newcommand{\Tiwa}{\mathcal T_{\operatorname{Iw}}}%
\newcommand{\Af}{\operatorname{A}}%
\newcommand{\Dunzero}{D_{1,0}}%
\newcommand{\Uun}{U_{1}}%
\newcommand{\Uzero}{U_{0}}%
\newcommand{\Uundual}{U^{1}}%
\newcommand{\Uunun}{U^{1}_{1}}%
\newcommand{\Wdual}{\Wcal^{dual}}
\newcommand{\JunNps}{J_{1,0}(\Ncal, P^{s})}%
\newcommand{\Tatepord}{\Tate_{\pid}^{ord}}%
\newcommand{\Kum}{\operatorname{Kum}}%
\newcommand{\zcid}{z(\cid)}%
\newcommand{\kgtilde}{\tilde{\kappa}}%
\newcommand{\kiwa}{\varkappa}%
\newcommand{\kiwatilde}{\tilde{\varkappa}}%
\newcommand{\Hbar}{\bar{H}}%
\newcommand{\Tred}{T/\mgot T}%
\newcommand{\Riwa}{R_{\operatorname{Iw}}}%
\newcommand{\Kiwa}{\Kcal_{\operatorname{Iw}}}%
\newcommand{\Sp}{\operatorname{Sp}}%
\newcommand{\St}{\operatorname{St}}%
\newcommand{\Aiwa}{\mathcal A_{\operatorname{Iw}}}%
\newcommand{\Viwa}{\mathcal V_{\operatorname{Iw}}}%
\newcommand{\pseudiso}{\overset{\centerdot}{\isom}}%
\newcommand{\pseudisom}{\overset{\approx}{\fleche}}%
\newcommand{\carac}{\operatorname{char}}%
\newcommand{\length}{\operatorname{length}}
\newcommand{\eord}{e^{\ord}}%
\newcommand{\eordm}{e^{\ord}_{\mgot}}%
\newcommand{\hordinfini}{\hord_{\infty}}%
\newcommand{\Mordinfini}{M^{\ord}_{\infty}}%
\newcommand{\hordm}{\hord_{\mgot}}%
\newcommand{\hminm}{\hgot^{min}_{\mgot}}%
\newcommand{\Mordm}{M^{\ord}_{\mgot}}%
\newcommand{\Mtwist}{M^{\tw}_{\mgot}}
\newcommand{\Xun}{X_{1}}%
\newcommand{\Xundual}{X^{1}}%
\newcommand{\Xunun}{X^{1}_{1}}%
\newcommand{\Xtw}{X^{\tw}}
\newcommand{\Inert}{\mathfrak{In}}%
\newcommand{\Tsp}{T_{\Sp}}%
\newcommand{\Asp}{A_{\Sp}}%
\newcommand{\Vsp}{V_{\Sp}}%
\newcommand{\SK}{\mathscr{S}}%
\newcommand{\Rord}{R^{\ord}}%
\newcommand{\per}{\operatorname{per}}
\newcommand{\z}{\mathbf{z}}
\newcommand{\zs}{\tilde{\mathbf{z}}}
\newcommand{\Ebarbar}{\bar{\bar{E}}}
\newcommand{\Grsym}{\mathfrak S}
\newcommand{\epsi}{\varepsilon}
\newcommand{\Fun}[2]{F^{{\mathbf{#1}}}_{#2}}
\newcommand{\tw}{\operatorname{tw}}
\newcommand{\ctf}{\operatorname{ctf}}
\newtheorem*{TheoA}{Theorem A}%
\newtheorem*{TheoB}{Theorem B}%
\newcommand{\isocan}{\overset{\can}{\simeq}}
\newcommand{\gen}{\operatorname{gen}}
  
   \tableofcontents

\section{Introduction}
Fix once and for all an odd prime $p$. The purpose of this manuscript is to draw attention towards two $p$-adic properties of the motivic fundamental lines which are believed to encode the special values of $L$-functions of motives as they are interpolated in $p$-adic families. More precisely, we prove first the general form of Mazur's Control Theorem (\cite{MazurRational,OchiaiControl}) conjectured in \cite[Conjecture 3.2.2 (i)]{KatoViaBdR}: the motivic fundamental line $\Delta_{\Sigma}(T/F)$ (see definitions \ref{DefFundamental} and \ref{DefFundamentalBis} for the notation) of a $p$-adic family of $\Gal(\Fbar/F)$-representations parametrized by the spectrum of a characteristic zero domain $\Lambda$ commutes with specialization at motivic points of $\Spec\Lambda[1/p]$ and, in fact, with the larger and fairly general class of specializations at pure monodromic point. However, we further show that as defined by K.Kato, $p$-adic families of motivic fundamental lines need not commute with arbitrary specializations nor even with arbitrary specializations with values in characteristic zero domains, perhaps contrary to what is suggested in \cite{KatoViaBdR}. When the $p$-adic family $T$ arises from a $p$-adic family of automorphic representations of a reductive group $\G$ over $F$ and when $\Lambda$ is a Hecke algebra, we give evidence suggesting that defining fundamental lines using the Weight-Monodromy filtration and the completed cohomology of \cite{EmertonCompleted} yields better behaved object. We propose a revised conjecture along these lines, show that it is compatible with arbitrary characteristic zero specialization and that it holds for $\GL_{2}$ under mild hypotheses.
\subsection{Motivic fundamental lines}
Motivic fundamental lines were originally introduced by A.Grothendieck and P.Deligne in \cite[Exposé III]{DixExposes} and \cite{SGA41/2} to explain the properties of the $L$-function of a $p$-adic étale sheaf $\Fcali$ on a separated scheme of finite type over $\Fp_{\ell}$ ($\ell\neq p$), in which case the fundamental line $\Delta(X,\Fcali)$ is the determinant of the étale cohomology of $X$ with coefficients in $\Fcali$. The counterparts over number fields of $X$ and $\Fcali$ are respectively $\Spec\Ocal_{F}[1/\Sigma]$ and the $p$-adic étale realization of a motive $M/F$ seen as $\Gal(\Fbar/F)$-representation where for concreteness the motive $M$ may be chosen to be the universal cohomology $h^{i}(Y)(j)$ attached to a smooth, proper scheme $Y/F$. After two decades of efforts the most important milestones of which are due to J-P.Serre, P.Deligne, A.Beilinson, S.Bloch, K.Kato, J-M.Fontaine and B.Perrin-Riou (see \cite{SerreFacteurs,DeligneFonctionsL,BlochCyclesLfunctionsI,BeilinsonConjecture,BlochCyclesLfunctionsII,BlochKato,FontainePerrinRiou} and references therein), a prediction of the $p$-adic valuation of the special value at zero of the $L$-function $L(M,s)$ of such a motive (say with coefficients in $\Q$) was achieved, albeit of course only conjecturally, in terms of a $p$-adic fundamental line
\begin{equation}\nonumber
\Delta(M_{\et,p})=\Det^{-1}_{\qp}\RGamma_{\et}(\Spec\Ocal_{F}[1/p],M_{\et,p})\tenseur_{\qp}\produittenseur{\s:F\plonge\C}{}\Det^{-1}_{\qp}M_{\et,p}(-1)^{+}
\end{equation}
and of a basis $\z(M_{\et,p})$,  called the $p$-adic zeta element of $M$, of a $\zp$-submodule of $\Delta(M_{\et,p})$.

Following the work of K.Iwasawa on $p$-adic interpolation of special values of $L$-functions of characters of $\Gal(\Qbar/\Q)$  (\cite{IwasawaOnPadic}) and that of B.Mazur and Y.Manin (\cite{MazurRational,ManinHecke,MazurValues}) on $p$-adic interpolation of the special values of $L$-functions attached to eigencuspforms and related algebraic invariants, it became clear that the formulation of the conjectures on special values of $L$-functions of motives should be supplemented by a description of the compatibilities of these values with congruences between motives and more generally of the variation of these values as the $p$-adic étale realization of the motive ranges over a $p$-adic analytic family of Galois representations (see \cite{MazurTheme} for general considerations). In terms of fundamental lines and zeta elements, this task amounts to the following questions. Suppose that $\Lambda$ is a complete, local, noetherian ring with residual characteristic $p$, that $T$ is a $\Lambda$-adic $\Gal(\Fbar/F)$-representation, that $\Ocal$ is the unit ball of a finite extension of $\qp$ and that $\psi:\Lambda\fleche\Ocal$ is a motivic specialization, that is to say a local $\zp$-algebras morphism such that $T_{\psi}\eqdef T\tenseur_{\Lambda,\psi}\Ocal$ is an $\Ocal$-lattice inside the $p$-adic étale realization of a motive $M$. Does there then exist a $\Lambda$-adic zeta element $\z_{\Lambda}(T)$ inside a $\Lambda$-adic fundamental line $\Delta_{\Lambda}(T)$ such that the commutative diagram of specializations 
\begin{equation}\nonumber
\xymatrix{
\Lambda\ar[rr]^{\psi}\ar[rd]_{\modulo\mgot_{\Lambda}}&&\Ocal\ar[ld]^{\modulo\mgot_{\Ocal}}\\
&k&
}
\end{equation}
yields a natural commutative diagram 
\begin{equation}\label{EqDiagComIntro}
\xymatrix{
&\Delta_{\Lambda}(T)\tenseur_{\Lambda,\psi}\Ocal\ar[r]^(0.6){\simeq}\ar[dd]&\Delta(T_{\psi})\ar[dd]^{}\\
\Delta_{\Lambda}(T)\ar[ru]\ar[rd]\\
&\Delta_{\Lambda}(T)\tenseur_{\Lambda}k\ar[r]^{\simeq}&\Delta(T\tenseur_{\Lambda}k).
}
\end{equation}
sending zeta elements to zeta elements? (This formal statement expresses the very concrete requirement that the $\Lambda$-adic fundamental line should $p$-adically interpolate the fundamental lines at motivic points.)

The first step towards this goal, namely the case of a motive $M/F$ endowed with a supplementary action of the Galois group $G$ of a finite abelian extension of $F$, was completed by K.Kato who conjectured in \cite{KatoHodgeIwasawa} that, in this setting, the motivic fundamental line of $M$ is endowed with an action of $G$ and that the period maps allowing the computation of the special values of $L(M,s)\in\C[G]^{\C}$ are $G$-equivariant. Taking the system of conjectures attached to the finite sub-extensions of a $\zp^{d}$-extension of $F$ recovers and generalizes the Iwasawa theory for motives with ordinary reduction of R.Greenberg (\cite{GreenbergIwasawaRepresentation,GreenbergIwasawaMotives}).

Soon thereafter, K.Kato proposed in \cite{KatoViaBdR} a much bolder and deeper conjecture, therein named the generalized Iwasawa main conjecture but nowadays more commonly called the Equivariant Tamagawa Number Conjecture, which defines fundamental lines and zeta elements of arbitrary $p$-adic families of Galois representations and states their properties. More precisely, let $\Lambda$ be a $p$-adic ring; for instance a complete, local, noetherian ring of mixed characteristic $(0,p)$. To any number field $F$, any finite subset of finite places $\Sigma\supset\{v|p\}$ of $F$ and any perfect complex of étale sheaves of $\Lambda$-modules $\Fcali$ on $\Spec\Ocal_{F}[1/\Sigma]$, the conjectures of \cite[Section 3.2]{KatoViaBdR} attach a projective $\Lambda$-module $\Delta(\Spec\Ocal_{F}[1/\Sigma],\Fcali)$ of rank 1 with a distinguished basis $\z(\Spec\Ocal_{F}[1/\Sigma],\Fcali)$ satisfying a number of properties; among which that $\z(\Spec\Ocal_{F}[1/\Sigma],\Fcali)$ computes the special value $L(M^{*}(1),0)$ if $\Fcali$ is the étale sheaf of $\Qbar_{p}$-vector spaces attached to the $p$-adic étale realization of a motive $M$ (see conjectures \ref{ConjTNC} and \ref{ConjETNC} below for a complete statement).
\subsection{Statement of results}
\subsubsection{Motivic points}One of the crucial predictions of \cite[Conjecture 3.2.2]{KatoViaBdR} is that if the fibre $\Fcali_{x}\eqdef\Fcali\tenseur_{\Lambda}\kbold(x)$ at $x\in\Spec\Lambda$ of $\Fcali$ is the $p$-adic étale realization of a motive $M$ (here $\kbold(x)$ is the residual field of $\Lambda_{x}$ and is a finite extension of $\qp$ or $\Qbar_{p}$), then the zeta element $\z(\Spec\Ocal_{F}[1/\Sigma],\Fcali)$ attached to the family should specialize to the zeta element $\z(\Spec\Ocal_{F}[1/\Sigma],\Fcali_{x})$ attached by the conjectures of \cite{BlochKato,FontainePerrinRiou} to the motive $M$ under the natural map
\begin{equation}\label{EqIsoIntro}
\Delta(\Spec\Ocal_{F}[1/\Sigma],\Fcali)\tenseur_{\Lambda}\kbold(x)\fleche\Delta(\Spec\Ocal_{F}[1/\Sigma],\Fcali_{x}),
\end{equation}
which is thus an isomorphism. As above, this simply means that the upper horizontal arrow of diagram \eqref{EqDiagComIntro} should indeed be an isomorphism and that $\z(\Spec\Ocal_{F}[1/\Sigma],\Fcali)$ should interpolate the zeta elements attached to the motivic points of $\Fcali$. It should be noted that this requirement does not obviously hold nor does it immediately follows from general arguments, even for some of the simplest non-trivial cases. The reader familiar with the definition of the fundamental lines in terms of Selmer groups or of the \Nekovar-Selmer complexes of \cite{SelmerComplexes} can convince himself of this fact by checking that \eqref{EqIsoIntro} implies an error-free version of Mazur's Control Theorem when $\Fcali$ is attached to the families of cyclotomic twists of the Tate module of an abelian variety (we return to this point below). In fact, even if one defines $\z(\Spec\Ocal_{F}[1/\Sigma],\Fcali)$ and $\Delta(\Spec\Ocal_{F}[1/\Sigma],\Fcali)$ by the property that they interpolate their counterparts in a Zariski-dense subset of motivic points, the statement remains non-tautological. Indeed, assume for the moment that $\Spec\Lambda$ admits a Zariski-dense subset $X$ of motivic points such that $\z(\Spec\Ocal_{F}[1/\Sigma],\Fcali_{x})$ has been constructed and shown to satisfy the Tamagawa Number Conjectures for $x\in X$. Then, if \eqref{EqIsoIntro} is an isomorphism as predicted for $x\in X$, this data suffices to specify $\z(\Spec\Ocal_{F}[1/\Sigma],\Fcali)$ if it exists. The question then arises as whether the interpolation property for zeta elements hold for $x\notin X$, or in other words whether \cite[Conjecture 3.2.2]{KatoViaBdR} for $\Fcali$ is compatible with the conjectures of \cite{BlochKato,KatoHodgeIwasawa} for $\Fcali_{x}$ at \textit{all} motivic points.

Our first main result establishes that this is the case for a large class of $p$-adic families (see theorem \ref{TheoPositif} for the precise statement, which is significantly stronger).
\begin{TheoEnglish}\label{TheoIntro}
Let $(T,\rho,\Lambda)$ be a $\Gal(\Fbar/F)$-representation with coefficients in a characteristic zero domain $\Lambda$. Assume that $(T,\rho,\Lambda)$ satisfies the hypotheses \ref{HypMain}. Then for all specialization $\psi:\Lambda\fleche\Lambda'$ with values in a characteristic zero domain and all finite sets $\Sigma\supset\{v|p\}$ of finite places of $F$, the $\Lambda'$-adic fundamental line $\Delta_{\Sigma}(T\tenseur_{\Lambda,\psi}\Lambda'/F)$ is defined. If $\Ocal$ is the ring of integers of a finite extension of $\qp$ and $\psi:\Lambda\fleche\Ocal$ is motivic (that is to say when $T_{\psi}\eqdef T\tenseur_{\Lambda,\psi}\Ocal$ is a $\Gal(\Fbar/F)$-stable $\Ocal$-lattice in the $p$-adic étale realization of a pure motive $M$), then $\Delta_{\Sigma}(T_{\psi}/F)$ coincides with the $p$-adic fundamental line of \cite{FontainePerrinRiou} and the natural map 
\begin{equation}\nonumber
\Delta_{\Sigma}(T/F)\tenseur_{\Lambda,\psi}\Ocal\fleche\Delta_{\Sigma}(T_{\psi}/F)
\end{equation}
is an isomorphism.
\end{TheoEnglish}
Here, we mention that we take as a defining property of the $p$-adic étale realization $M_{\et,\pid}$ of a pure motive $M/F$ the fact that, for all $v\nmid p$, the Weil-Deligne representation of $\Gal(\Fbar_{v}/F_{v})$ attached to $M_{\et,\pid}$ is pure in the sense of the Weight-Monodromy Conjecture (see conjecture \ref{ConjWMC} below for details). We also note that already the unconditional definition of $\Delta_{\Sigma}(T\tenseur_{\Lambda,\psi}\Lambda'/F)$ appears to be new, even in the ordinary case (\cite{KatoViaBdR} requires the étale sheaf on $\Spec\Z[1/\Sigma]$ attached to $T$ to belong to the category $D_{\ctf}(\Spec\Ocal_{F}[1/\Sigma],\Lambda)$ of perfect complexes of constructible étale sheaves of $\Lambda$-modules on $\Spec\Ocal_{F}[1/\Sigma]$, \cite{FukayaKato} requires $\Sigma$ to contain all places of ramification of $T$ and \cite{GreenbergIwasawaMotives} requires $T$ to be ordinary and $\Lambda$ to be a normal ring).

In addition to the case of families arising from classical Iwasawa theory, which is easy and well-known, the hypotheses of theorem \ref{TheoIntro} are known to hold for many $p$-adic families of automorphic Galois representations, for instance those attached to $p$-adic families of Hilbert eigencuspforms or those parametrized by the spherical Hecke algebra of a reductive group $\G$ compact at infinity. More precisely, let $\G$ be a reductive group which is either the unit group of a quaternion algebra over a totally real field $F$ or a totally definite unitary group in $n$ variables over a totally real field $F^{+}$ split by a CM extension $F/F^{+}$. Let $\mgot$ be a non-Eisenstein ideal of the reduced, spherical Hecke algebra $\Hsm(U^{p})$ acting on the completed cohomology in middle degree of the Shimura variety attached to $\G$ (here $\Sigma$ is a finite set of primes of $\Ocal_{F}$; we refer to section \ref{SubAutomorphic} below for further details about the notations). Let $T$ be the $\Gal(\Fbar/F)$-representation with coefficients in $\Hsm(U^{p})$ interpolating the representations attached to automorphic systems of Hecke eigenvalues occurring in completed cohomology ($T$ is characterized by the property that $\tr(\Fr(w))=T_{1}(w)$ for all $w$ in a dense subset of places of $F$). Applying theorem \ref{TheoIntro} in this setting yields the following corollary, which strengthens and vastly generalizes the control theorems of \cite{MazurRational,OchiaiControl,FouquetOchiai}. 
\begin{CorEnglish}\label{CorIntro}
Let $\psi:\Hsm(U^{p})\fleche\Ocal_{\pi}\subset\Qbar_{p}$ be the morphism sending a Hecke operator to its eigenvalue in its action on a cohomological automorphic representation $\pi$. Let $K$ be a finite abelian extension of $F$ unramified outside $\Sigma$. Denote by $\Fcali$ the étale sheaf attached to $T$ and by $\Fcali_{\psi}$ the étale sheaf attached to $T_{\psi}=T\tenseur_{\Hsm,\psi}\Ocal_{\pi}$. If $\z(\Fcali,\Spec\Ocal_{K}[1/\Sigma])$ is a basis of $\Delta_{\Sigma}(T/K)$, then $\z(\Fcali,\Spec\Ocal_{K}[1/\Sigma])\tenseur1\in\Delta_{\Sigma}(T/K)\tenseur_{\Hsm,\psi}\Ocal_{\pi}\isocan\Delta_{\Sigma}(T_{\psi}/K)$ is a basis of $\Delta_{\Sigma}(T_{\psi}/K)$.
\end{CorEnglish}
Informally, this corollary states that if the ETNC is true for $T$, then it is true for any motivic point of $T$. Maybe because the generalized Iwasawa main conjecture \cite[Conjecture 3.2.2]{KatoViaBdR} is hopelessly out of reach already in very simple cases, it seems that the possibility that statements like theorem \ref{TheoIntro} and corollary \ref{CorIntro} could hold has been neglected so that, for instance, this author believes they are new already for Hida families of quaternionic eigencuspforms or for the $p$-adic families of modular forms occurring in the completed cohomology of the tower of modular curves. Note also the important fact that we make no assumption on the local properties of $T$ at $v|p$ (most of the literature assumes that $\rho|G_{F_{v}}$ is ordinary, flat or of finite slope at $v|p$).
\subsubsection{Failure at non-motivic points}
While theorem \ref{TheoIntro} settles the question of motivic specialization (and in fact the more general class of specializations with values in characteristic zero domains and for which the weight and monodromy filtrations are defined and coincide), it does not apply to all specializations. This is the object of our second main result, which explains the subtitle of this note. This subtitle, like the title of \cite{GrothendieckTrivial} which it parodies, is misleading as readers might be naturally induced to believe that our results cast doubts on the conjectures of \cite{KatoViaBdR} referred to above. In fact, they point towards an interesting connection between the $p$-adic properties of special values of $L$-function, the $p$-adic properties of the Local Langlands Correspondence (\cite{EmertonCoates,EmertonHelm,EmertonPatching}) and Ihara's lemma and its generalizations (\cite{IharaLemma,DiamondTaylor,ClozelHarrisTaylor}). This author believes that such a surprising connection, insofar as it impacts the Equivariant Tamagawa Number Conjectures, strengthens rather than weakens their plausibility.

In order to understand what may require a modification in \cite[Conjecture 3.2.2]{KatoViaBdR}, it is helpful to come back to the definition of the $p$-adic étale realization of the fundamental line of a pure, strictly critical motive $M$ (see section \ref{SubMotives} for the meaning of this property) over a number field $F$ with coefficients in $\Q$. As we saw above, in the normalization of \cite{KatoViaBdR}, this is by definition the one-dimensional $\qp$-vector space
\begin{equation}\nonumber
\Delta(M_{\et,p})=\Det^{-1}_{\qp}\RGamma_{\et}(\Ocal_{F}[1/\Sigma],M_{\et,p})\tenseur_{\qp}\produittenseur{\s:F\plonge\C}{}\Det^{-1}_{\qp}M_{\et,p}(-1)^{+}.
\end{equation}
As recalled in section \ref{SubFun} below, in the computation of $L(M_{\et,p}^{*}(1),0)$ in terms of $\Delta(M_{\et,p})$, the étale cohomology complex recovers through $p$-adic comparison theorems a rational structure coming from the de Rham realization of $M$. The role of the term $\produittenseur{\s:F\plonge\C}{}\Det^{-1}_{\qp}M_{\et,p}(-1)^{+}$ is to yield a second rational structure coming from the Betti realization of $M$ which, when compared with the first one, computes the special value. This, and the fact that the tensor product has to be taken over all complex embeddings of $\C$, points out to the fact that even though it is the $p$-adic étale realization of $M$ which appears in this second term, this term is in essence a Betti cohomology group (and indeed, it is believed to come by tensor product with $\qp$ from the Betti term in a putative motivic fundamental line). For a single motive, the canonical isomorphism $M_{B}\tenseur_{\Q}\qp\isocan M_{\et,p}$ ensures that identifying Betti and étale cohomology groups can cause no harm but this discussion raises the question of the correct analogue of $M_{\et,p}$ in the case of $p$-adic families of Galois representations.

The implicit choice made in \cite{KatoViaBdR} is that étale cohomology remains the correct choice, so that $\produittenseur{\s:F\plonge\C}{}\Det^{-1}_{\qp}M_{\et,p}(-1)^{+}$ is replaced by $\produittenseur{\s:F\plonge\C}{}\Det^{-1}_{\Lambda}T(-1)^{+}$ in the definition of the $\Lambda$-adic fundamental line
\begin{equation}\nonumber
\Delta(T)=\Det^{-1}_{\Lambda}\RGamma_{\et}(\Ocal_{F}[1/\Sigma],T)\tenseur_{\Lambda}\produittenseur{\s:F\plonge\C}{}\Det^{-1}_{\Lambda}T(-1)^{+}
\end{equation}
of a family of Galois representation $T$ parametrized by $\Spec\Lambda$. However, this choice runs into difficulty when fundamental lines are specialized at non-motivic points. Indeed, suppose that $T_{x}$ is a $\zp[\Gal(\Qbar/\Q)]$-module arising as specialization of a $\Lambda$-adic family $T$ which is ramified at some prime $\ell$ different from $p$ and such that $T_{x}$ is unramified at $\ell$. In that case, as we hinted above, the natural map 
\begin{equation}\label{EqErreur}
(\Det_{\Lambda}\RGamma_{\et}(\Z[1/\Sigma],T))\tenseur_{\Lambda}\zp\fleche\Det_{\zp}\RGamma_{\et}(\Z[1/\Sigma],T_{x})
\end{equation}
induced by $x:\Lambda\fleche\zp$ may fail to be an isomorphism in general when $\Sigma$ does not contain $\ell$ whereas $T^{+}\tenseur_{\Lambda,x}\zp$ is by construction isomorphic to $T^{+}_{x}$. When this occurs, nothing in $\Delta(T)\tenseur_{\Lambda,x}\zp$ can compensate the error term coming from the failure of \eqref{EqErreur} to be an isomorphism so the natural map $\Delta(T)\tenseur_{\Lambda,x}\zp\fleche\Delta(T_{x})$ fails to be an isomorphism, perhaps in contradiction with \cite[Conjecture 3.2.2 (i)]{KatoViaBdR}.

Our second main result is that that phenomenon is pervasive as soon as $p$-adic families of motives of rank greater than 1 which do not arise from classical Iwasawa theory are considered.
\begin{TheoEnglish}\label{TheoDeux}
The specialization property at non-motivic point of \cite[Conjecture 3.2.2 (i)]{KatoViaBdR} probably fails already for generic Hida families of ordinary eigencuspforms for $\GL_{2}(\Q)$ if fundamental lines are defined as in \emph{loc.cit.}.
\end{TheoEnglish}
The reason for the cautious adverb in the statement of this theorem is that no actual example of characteristic zero representations satisfying the requirement put on $T_{x}$ have appeared in the literature as far as this author knows (to recall, $T_{x}$ should be unramified at $\ell$ but should belong to a $p$-adic family of Galois representation which is generically ramified at $\ell$). The consensus among experts is that they exist and that they are in some sense the generic case (a consideration of generic monodromy suggests, for instance, that the universal deformation space $\Spec R_{\Sigma}(\rhobar)$ parametrizing deformations unramified outside a set $\Sigma$ of an absolutely irreducible residual automorphic Galois representation $\rhobar$ of rank at least 2 admit such points for an infinite number of finite sets $\Sigma$).

In view of theorem \ref{TheoIntro}, a somewhat tautological corrected version of the Equivariant Tamagawa Number Conjecture can be formulated simply by requiring that fundamental lines be compatible only with motivic specializations. However, this author does not consider such a weakening of the conjecture satisfactory. One of the main interest of $p$-adic families of motives (and the reason they were considered in the first place) is the light they show on congruences between motives. In order to answer, or simply make sense of, the question of whether congruent motives have congruent special values of $L$-function (the question asked for eigencuspforms in \cite{MazurValues}), it is necessary that $p$-adic fundamental lines and zeta elements be defined for Galois representations with coefficients in artinian rings (typically, but not exclusively, $\Fp_{q}$ or $\Z/p^{s}\Z$). Specializations with values in such rings can never be motivic so banning them precludes any discussion of congruences between special values of $L$-function. Besides, recent years have shown the utmost arithmetic significance of not necessarily classical $p$-adic automorphic representations and Galois representations. It would seem inadequate to decree that Iwasawa theory has nothing to say about them. 

If, unsatisfied with the exclusion of non-motivic specializations, we seek to modify the Equivariant Tamagawa Number Conjecture to take them into account, we encounter the following question. As étale cohomology is not the appropriate choice for the definition of $p$-adic fundamental lines, what is? In the general case, I have no answer to offer to this question, but in the case of $p$-adic families of Galois representations arising from $p$-adic families of automorphic representations of a reductive group $\G$ over $F$ admitting Shimura varieties of dimension $d$, it is very tempting to speculate that the correct $p$-adic analogue of Betti realization is the completed cohomology
\begin{equation}\nonumber
\Htilde^{d}_{c}(U^{p},\Ocal)\eqdef\limproj{s}\left(\liminj{U_{p}}\ H^{d}_{c}(Y(U^{p}U_{p})(\C),\Ocal/\varpi^{s})\right)
\end{equation}
of \cite{EmertonCompleted}. Here, $U^{p}$ is a compact open subgroup of $\G(\A_{F}^{(p\infty)})$, $U_{p}$ is a compact open subgroup of $\G(F\tenseur_{\Q}\qp)$ and $Y(U^{p}U_{p})$ is the double quotient $\G(F)\backslash\G(\A_{F})/U^{\circ}_{\infty}U^{p}U_{p}$ for $U_{\infty}^{\circ}$ the connected component of the identity of a maximal compact subgroup of $\G(F\tenseur_{\Q}\R)$. Indeed, the local-global conjecture of \cite{EmertonCoates,EmertonPatching,EmertonICM} predicts that the $\rho$-isotypic part of completed cohomology involves the Local Langlands Correspondence normalized as in \cite{BreuilSchneider}. If $\rho|G_{\Q_{\ell}}$ does not satisfy the Weight-Monodromy Conjecture, this normalization involves an extra Euler factor which exactly compensates the error term appearing in the specialization of the étale cohomology complex. Under this choice of fundamental line then, neither the first nor the second term of the fundamental line commute with specializations but the error terms appearing in each factors cancel each other. 

\paragraph{Relation with \cite{FukayaKato}}
There is a point the importance of which cannot be overestimated. Because the phenomena at play in theorem \ref{TheoDeux} are essentially local at $\ell\neq p$ a place of ramification of $T$, they can play a role only in the formulation of the Equivariant Tamagawa Number Conjectures as given in \cite{KatoViaBdR} and not in the formulation of \cite{FukayaKato}. To be more concrete, it is a consequence of theorem \ref{TheoIntro} (or more precisely of its proof) that a specialization $x$ can fail to induce a canonical isomorphism between the fundamental lines $\Delta(\Spec\Ocal_{F}[1/\Sigma],\Fcali)$ and $\Delta(\Spec\Ocal_{F}[1/\Sigma],\Fcali_{x})$ \textit{only if} there exists $v\notin\Sigma$ such that $\Fcali_{x}$ seen as $G_{F_{v}}$-representation is less ramified than $\Fcali$. As \cite{FukayaKato} defines fundamental lines only for $\Sigma$ containing all the places of ramification of $\Fcali$, these fundamental lines are by construction compatible with arbitrary specializations. Correlatively, they contain less precise information than the one of \cite{KatoViaBdR} (or the one constructed from completed cohomology in this manuscript) as they do not $p$-adically interpolate Euler factors at places of bad reduction. Because claims to the contrary have circulated even in publications (much to the dismay of this author), we restate this point: theorem \ref{TheoDeux} does not affect the framework of \cite{FukayaKato}. 
\paragraph{Notations}
If $F$ is a field, denote by $G_{F}$ the Galois group of the separable closure of $F$. If $F$ is a number field, $\Sigma$ is a finite set of finite places of $F$ and $K$ is an extension of $F$, denote by $K_{\Sigma}$ the maximal Galois extension of $K$ which is ramified only at for places $w$ above the places of $F$ in $\Sigma\cup\{v|\infty\}$. Denote by $G_{K,\Sigma}$ the Galois group $\Gal(K_{\Sigma}/K)$. The ring of integers of $F$ is denoted by $\Ocal_{F}$. If $\ell$ is a place of $\Q$ (including $\infty$), denote by $F_{\ell}$ the tensor product $F\tenseur_{\Q}\Q_{\ell}$ and by $I_{F,\ell}$ the set of field embeddings of $F$ inside $F_{v}$ for $v|\ell$.
\section{The Equivariant Tamagawa Number Conjectures for motives}
We recall the formulation of the Equivariant Tamagawa Number Conjectures (ETNC for short) for a pure motive $M$ over a number field $F$ with coefficients in another number field $E$ and with a supplementary action of the Galois group of a finite abelian extension $K$ of $F$. Despite the fact that several expositions of these conjectures already exist (\cite{BlochKato,FontaineValeursSpeciales,KatoHodgeIwasawa,FontainePerrinRiou,BurnsFlach,FukayaKato}), we felt the need to include this section for two reasons. First, not all the references above normalize in the same way the objects considered nor do they define in the same way the maps relating motivic fundamental lines and $p$-adic fundamental lines (this point, though crucial for a precise statement, is sometimes neglected in the literature on the ETNC) and, non coincidentally, a common point of discrepancy is in the treatment of the local Euler factors whose study lies at the core of this manuscript. Second, as we intend to study $p$-adic families of motives, including the ones coming from classical Iwasawa theory, it is necessary for us that $K/F$ may be non-trivial. The main class of $p$-adic families satisfying theorem \ref{TheoPositif} below requires $F$ to be a CM extension, hence a non trivial one. Our main guiding examples all arise from $p$-adic families of automorphic forms and it is expected that the field of coefficients $E$ of motives in automorphic families grow without bound (see \cite{HidaHecke} for some known cases). Hence it is important to us that all the extensions $F/\Q$, $E/\Q$ and $K/F$ may be non-trivial. None of the references above (and none known to this author) allow  this generality. On the other hand, the contribution to the supplementary action of $\Gal(K/F)$ to the problems studied in this note is easy and well understood (it is recalled in section \ref{SubPunctual}) and the fields of coefficients of automorphic motives are naturally subfields of $\C$ (as this is the case of the reflex field of a Shimura variety) so the reader who wishes to take $K=F$ and to fix an embedding of $E$ in $\C$ may do so without fear of neglecting any essential difficulty.
\subsection{Motives and motivic Galois representations}
\subsubsection{Realizations}\label{SubMotives}
Let $E$ and $F$ be two number fields. Denote by $X$ a proper, smooth scheme over $F$ of dimension $d$ with good reduction outside a finite set $\Sigma_{0}$ of finite places of $F$. For $Y$ a proper, smooth scheme, denote by $CH^{i}(Y)$ the abelian group of closed, irreducible sub-schemes of $Y$ of codimension $i$ modulo homological equivalence. A pure motivic data $M$ over $F$ with coefficients in $E$ is a quadruple $(X,\epsi,i,j)$ where $\epsi$ is a projector of $CH^{\dim X}(X\times_{F} X)\tenseur_{\Z}E$ and $(i,j)\in\Z^{2}$ are integers. It is conjectured that to $(X,\epsi,i,j)$ is attached a motive $h^{i}(X)(\epsi)(j)$ over $F$ and with coefficients in $E$ which we also denote by $M$. Independently of the existence of this conjectural framework, pure motivic data admit cohomological realizations satisfying the following properties.
\paragraph{Betti realizations}For all $\s:F\plonge\C$, the Betti realization of $M$ is the $E$-vector space 
\begin{equation}\nonumber
M_{B,\s}=H^{i}((X\times_{F,\s}\C)(\C),\Q(j))(\epsi).
\end{equation}
When $\s$ is real, $M_{B,\s}$ is endowed with an action of $\Gal(\C/\R)$ given by the natural action on $(X\times_{F,\s}\C)(\C)$ and on $\Q(j)\eqdef(2\pi i)^{j}\Q$. When $\s$ is complex, we consider $M_{B,\s}$ to be endowed with the trivial action of $\Gal(\C/\R)$. In both cases, we denote by $M_{B,\s}^{\pm}$ the $E$-vector space on which the non-trivial element $c$ of $\Gal(\C/\R)$ acts with the eigenvalue $\pm1$.

For each choice of embedding $\s:F\plonge\C$, the $E\tenseur\C[\Gal(\C/\R)]$-module $M_{B,\s}\tenseur_{E}\C$ decomposes as
\begin{equation}\nonumber
M_{B,\s}\tenseur_{E}\C=\sommedirecte{\tau:E\plonge\C}{}\left(M_{B,\s}\tenseur_{E,\tau}\C\right).
\end{equation}
For each choice of embeddings $\s,\tau$ as above, the Hodge decomposition
\begin{equation}\nonumber
(M_{B,\s}\tenseur_{E}E_{\tau})\tenseur_{\R}\C=\sommedirecte{p+q=w}{}H^{p,q}_{\s,\tau}
\end{equation}
with action of $\Gal(\C/\R)$ on both sides of the tensor product endows $V_{\s,\tau}\eqdef M_{B,\s}\tenseur_{E}E_{\tau}$ with the structure of an $E_{\tau}$-Hodge structure over $F_{\s}$ in the sense of \cite{DeligneHodgeII} pure of weight $w=i-2j$ (by \cite{Hodge}).

More concretely, there is a finite, increasing filtration $W^{\bullet}$ indexed by $\Z$ on $V_{\s,\tau}$ and there are two finite, decreasing, $w$-opposite Hodge filtrations $\Fil^{\bullet}$ and $\overline{\Fil}^{\bullet}$ indexed by $\Z$ on $V_{\s,\tau}\tenseur_{\R}\C$ such that the filtrations $W^{\bullet},\Fil^{\bullet}$ and $\overline{\Fil}^{\bullet}$ on $V_{\s,\tau}\tenseur_{\R}\C$ are opposite in the sense of \cite[Définition 1.2.7]{DeligneHodgeII}: for all $r\in\Z$, the natural map induces an isomorphism
\begin{equation}\nonumber
\Fil^{r}\Gr_{n}^{W}(V_{\s,\tau}\tenseur_{\R}\C)\oplus\overline{\Fil}^{n+1-r}\Gr_{n}^{W}(V_{\s,\tau}\tenseur_{\R}\C)\simeq\Gr_{n}^{W}(V_{\s,\tau}\tenseur_{\R}\C)=\begin{cases}
V_{\s,\tau}\tenseur_{\R}\C&\textrm{if $n=w$,}\\
0&\textrm{else.}
\end{cases}
\end{equation}
If $\s$ is a real embedding, then $V_{\s,\tau}$ is moreover endowed with an action of $\Gal(\C/\R)$ compatible with the weight filtration and such that the Hodge filtration comes through tensor product over $\R$ to $\C$ from a filtration on $(V_{\s,\tau}\tenseur_{\R}\C)^{\Gal(\C/\R)}$. If $\tau$ is a complex embedding, then $\C$ embeds in $\End_{\R}(V_{\s,\tau})$. We recall that even when $\s$ and $\tau$ are complex embeddings, the $V_{\s,\tau}$ are always regarded as $\R$-vector spaces.

If $V_{\s,\tau}$ is a Hodge structure over $\R$ (equivalently, if $\s$ is real), denote by $V_{\s,\tau}^{\pm}$ the subspace $V_{\s,\tau}^{c=\pm1}$ of $V_{\s,\tau}$ and by $V_{\s,\tau,\dR}$ the subspace $(V_{\s,\tau}\tenseur_{\R}\C)^{\Gal(\C/\R)}=V_{\s,\tau}^{+}\oplus iV_{\s,\tau}^{-}$. Otherwise, put $V_{\s,\tau}^{+}=V_{\s,\tau}$ and $V_{\s,\tau,\dR}=V_{\s,\tau}\tenseur_{\R}\C$. In both cases, the quotient of $V_{\s,\tau,\dR}$ by $V_{\s,\tau}^{+}$ is $V_{\s,\tau}(-1)^{+}$ and there is a map of $E_{\tau}$-vector spaces
\begin{equation}\nonumber
\psi_{\s,\tau}:\Fil^{0}(V_{\s,\tau}\tenseur_{\R}\C)\subset V_{\s,\tau,\dR}\surjection V_{\s,\tau}(-1)^{+}.
\end{equation}
\paragraph{De Rham realization}The de Rham realization (\cite{GrothendieckDeRham}) of $M$ is a finite dimensional $F\tenseur_{\Q}E$-module
\begin{equation}\nonumber
M_{\dR}=H^{i}_{\dR}(X,\Omega^{\bullet}_{X/F})(\epsi)(j)
\end{equation}
with a finite filtration $\Fil^{i}M_{\dR}$ indexed by $\Z$ such that $\Fil^{k}M_{\dR}$ is the $F\tenseur_{\Q}E$-module image of $H^{i}_{\dR}(X,\Omega^{\geq k+j}_{X/F})(\epsi)$ inside $M_{\dR}$. The tangent space $t_{\dR}$ of $M$ is the $F\tenseur_{\Q} E$-module $M_{\dR}/\Fil^{0}M_{\dR}$.

For each choice of embeddings $\s:F\plonge\C$ and $\tau:E\plonge\C$, the resolution of the constant sheaf $\C$ by the de Rham complex yields a canonical comparison isomorphism of $\C[\Gal(\C/\R)]$-modules
\begin{equation}\label{EqIsoBettiDeRham}
I_{\s,\infty}:(M_{B,\s}\tenseur_{E}E_{\tau})\tenseur_{\R}\C\isocan M_{\dR}\tenseur_{F\tenseur E}\C
\end{equation}
where $F$ and $E$ are embedded in $\C$ through $\s$ and $\tau$ on the right-hand side and in which the action of $\Gal(\C/\R)$ is as in the previous paragraph on the left-hand side and on $\C$ on the right-hand side. This isomorphism sends $\sommedirecte{p\geq k}{}H^{p,q}_{\s,\tau}$ to $(\Fil^{k}M_{\dR})\tenseur_{F\tenseur E}\C$.

Fix an embedding $\tau:E\plonge\C$. The composition of the inverse of $I_{\s,\infty}$ with the map $\psi_{\s,\tau}$ for all $\s:F\plonge\C$ defines a map
\begin{equation}\label{EqCritical}
\psi_{\tau}:\left((\Fil^{0}M_{\dR})\tenseur_{F\tenseur E}E_{\tau}\right)\tenseur_{\R}\C\fleche\sommedirecte{\s:F\plonge\C}{}V_{\s,\tau}(-1)^{+}.
\end{equation}
The motivic data $M$ (or the system of realizations $\{\{M_{B,\s}\}_{\s},M_{\dR}\}$) is said to be critical in the sense of \cite{DeligneFonctionsL} if there exists a $\tau$ such that the map $\psi_{\tau}$ is an isomorphism (in which case the map $\psi_{\tau}$ is an isomorphism for all $\tau$).
\paragraph{Étale realizations}For all finite place $\pid$ of $E$, the $\pid$-adic étale realization of $M$ is the $E_{\pid}$-vector spaces\begin{equation}\nonumber
M_{\et,\pid}=H^{i}_{\et}(X\times_{F}\Fbar,\Q_{p}(j))(\epsi)
\end{equation}
with its natural action of $G_{F}$. For any finite set $\Sigma$ containing $\{v|p\}$, $M_{\et,\pid}$ defines a constructible étale sheaf on $\Spec\Ocal_{F}[1/\Sigma]$. The natural map $\nu_{v}:\Spec\Ocal_{F_{v}}\fleche\Spec\Ocal_{F}[1/\Sigma]$ if $v\notin\Sigma$ or $\nu_{v}:\Spec F_{v}\fleche\Spec\Ocal_{F}[1/\Sigma]$ for all finite $v$ allow us to consider $M_{\et,\pid}$ (more precisely $\nu_{v}^{*}M_{\et,\pid}$) as a constructible étale sheaf over $\Spec\Ocal_{F_{v}}$ or $\Spec F_{v}$ respectively.

For $v\notin\Sigma_{0}\cup\{v|p\}$ a finite place of $F$, the smooth and proper base change theorem implies that the $G_{F_{v}}$-representation $M_{\et,\pid}$ is unramified. By \cite{DeligneWeil}, the Euler factor 
\begin{equation}\label{EqDefEuler}
\Eul_{v}(M_{\et,\pid},X)\eqdef\det\left(1-\Fr(v)X|M_{\et,\pid}\right)\in E_{\pid}[X]
\end{equation}
has furthermore coefficients in $E$ and is independent of the choice of $\pid$ (here $X$ is an indeterminate unrelated to the scheme $X$ above).

Let $v|p$ and $\pid|p$ be finite places of $F$ and $E$ respectively. Let $B_{\dR,v}$ be the ring of $p$-adic periods relative to the extension $\Fbar_{v}/F_{v}$ (\cite{FontaineDeRham,FontaineCorps}) and denote by $D_{\dR,v}$ the functor 
\begin{equation}\nonumber
D_{\dR,v}:V\mapsto H^{0}(G_{F_{v}},V\tenseur_{\qp}B_{\dR,v})
\end{equation}
from the category of finite-dimensional $E_{\pid}$-vector spaces with a continuous action of $G_{F_{v}}$ to the category of $F_{v}\tenseur_{\qp}E_{\pid}$-module with a finite, decreasing filtration indexed by $\Z$. A finite-dimensional $E$-vector space $V$ equipped with a continuous action of $G_{F_{v}}$ is de Rham (\cite{FontaineSemistable}) if the dimension of $D_{\dR,v}(V)$ as $F_{v}$-vector space is equal to $[E_{\pid}:\qp]\dim_{E_{\pid}}V$. When $V$ is a de Rham representation, the cup-product with $\log(\chi_{\cyc})\in\Hun(G_{F_{v}},\qp)$ is an isomorphism
\begin{equation}\label{EqExpDual}
D_{\dR,v}^{i}(V)\eqdef H^{0}(G_{F_{v}},V\tenseur_{\qp}B_{\dR,v}^{i})\isom\Hun(G_{F_{v}},V\tenseur_{\qp}B_{\dR,v}^{i})
\end{equation}
for all $i\in\Z$. The proof of the semi-stable conjecture of J-M.Fontaine and U.Jannsen, which was completed in \cite{TsujiSemistable}, shows that the $G_{F_{v}}$-representation $M_{\et,\pid}$ is de Rham for all $v|p$.

For each choice of embedding $\bar{\s}:\Fbar\plonge\C$ extending $\s$, Artin comparison theorem (\cite[Exposé XI]{SGA4}) gives a canonical isomorphism
\begin{equation}\nonumber
I_{\bar{\s},\pid}:M_{B,\s}\tenseur_{E}E_{\pid}\isocan M_{\et,\pid}
\end{equation}
which is compatible with the action of $\Gal(\C/\R)$ if $\s$ is a real embedding. We fix a $\bar{\s}$ for each $\s:F\plonge\C$ and denote by $I_{\pid}$ the isomorphism 
\begin{equation}\nonumber
I_{\pid}=\sommedirecte{\s:F\plonge\C}{}I_{\bar{\s},\pid}:\sommedirecte{\s:F\plonge\C}{}M_{B,\s}\tenseur_{E}E_{\pid}\simeq\sommedirecte{\s:F\plonge\C}{} M_{\et,\pid}.
\end{equation}

By \cite{TsujiSemistable} (or \cite{FaltingsCrys,FaltingsAlmost}), there is a comparison isomorphism 
\begin{equation}\nonumber
I_{\dR,v,\pid}:M_{\dR}\tenseur_{F}B_{\dR,v}\simeq M_{\et,\pid}\tenseur_{\qp}B_{\dR,v}
\end{equation}
which is canonical by \cite{NiziolUniqueness} and which induces an isomorphism of filtered $E_{\pid}$-modules
\begin{equation}\nonumber
M_{\dR}\tenseur_{F\tenseur E}F_{v}\tenseur E_{\pid}\simeq D_{\dR,v}(M_{\et,\pid}).
\end{equation}
As the rank of $D_{\dR}(M_{\et,\pid})$ as $F_{v}\tenseur E_{\pid}$-module is equal to the rank of $M_{\dR}$ as $F\tenseur E$-module, it does not depend on $v|p$.

The composition of the direct sum of the localization maps
\begin{equation}\nonumber
\Hun_{\et}(\Ocal_{F}[1/\Sigma],M_{\et,\pid})\fleche\sommedirecte{v|p}{}\Hun_{\et}(\Spec F_{v},M_{\et,\pid})\simeq\sommedirecte{v|p}{}\Hun(G_{F_{v}},M_{\et,\pid})
\end{equation}
with the extension of scalar from $\Hun(G_{F_{v}},M_{\et,\pid})$ to $\Hun(G_{F_{v}},M_{\et,\pid}\tenseur_{\qp}B_{\dR,v}^{+})$ and the direct sum of the inverse of the maps \eqref{EqExpDual} defines a map
\begin{equation}\label{EqExp}
\sommedirecte{v|p}{}\exp_{v}^{*}:\Hun_{\et}(\Ocal_{F}[1/\Sigma],M_{\et,\pid})\fleche \sommedirecte{v|p}{}D_{\dR,v}^{0}(M_{\et,\pid})\simeq (\Fil^{0} M_{\dR})\tenseur_{E}E_{\pid}
\end{equation}
called the dual exponential map (\cite[Section 1.2]{KatoViaBdR}). The motivic data $M$ (or the system of realizations $\{M_{\et,\pid},M_{\dR}\}$) is said to be critical at $p$ if $\exp^{*}$ is an isomorphism.%
\paragraph{Rank}The rank of $M$ over $F$ is the common dimension of the $E$-vector spaces $M_{B,\s}$. This is also the rank of $M_{\dR}$ as $E\tenseur_{\Q}F$-module and, for all finite place $\pid$ of $E$, the dimension of $M_{\et,\pid}$ as $E_{\pid}$-vector space.
\paragraph{Dual}The dual motivic data of $(X,\epsi,i,j)$ is the motivic data $(X,\epsi,2d-i,d+1-j)$. The dual motive $M^{*}(1)$ of $M$ is the motive conjecturally attached to the dual motivic data. Unconditionally, the realizations of the dual motivic data are dual to the realizations of the motivic data $(X,\epsi,i,j)$. If the motivic data $(X,\epsi,i,j)$ is pure of weight $w$, then the dual motivic data is pure of weight $-w-2$. In particular, either $M$ or $M^{*}(1)$ is pure of weight $w\geq-1$.
\paragraph{Coefficients}It is conjectured that to a pair of motives $M,N$ over $F$ and with coefficients in $E$, there exists a tensor product motive $M\tenseur N$. The only instance of this construction we have cause to consider is the tensor product of a motive $M$ with the motive $h^{0}(\Spec K)$ attached to the pure motivic data $(\Spec K,1,0,0)$ where $K/F$ is a finite, abelian extension. Let $(X,\epsi,i,j)$ be a pure motivic data over $F$ with coefficients in $E$. There exists a pure motivic data $(X\times_{F}K,\epsi,i,j)$ over $F$ (and hence conjecturally a motive which we denote by $M\times_{F}K$ or $M_{K}$ for brevity) whose realizations are
\begin{equation}\nonumber
\begin{cases}
(M\times_{F}K)_{B,\s}=M_{B,\s}\tenseur_{E}E[G],\\
(M\times_{F}K)_{\dR}=M_{\dR}\tenseur_{F}K,\\
(M\times_{F}K)_{\et,p}=M_{\et,p}\tenseur_{E_{\pid}}E_{\pid}[G].
\end{cases}
\end{equation}
We refer to $M_{K}$ as the motive $M$ over $F$ with coefficients in $E[\Gal(K/F)]$.
\paragraph{Strictly critical motive}A motive $M$ is called strictly critical (with respect to $\tau$ and $\pid$ a finite place of $E$) if it satisfies the following three properties.
\begin{enumerate}\nonumber
\item The map $\psi_{\tau}:\left((\Fil^{0}M_{\dR})\tenseur_{F\tenseur E}E_{\tau}\right)\tenseur_{\R}\C\fleche\sommedirecte{\s:F\plonge\C}{}V_{\s,\tau}(-1)^{+}$ of \eqref{EqCritical}  is an isomorphism.
\item The dual exponential map $\sommedirecte{v|p}{}\exp_{v}^{*}:\Hun_{\et}(\Ocal_{F}[1/\Sigma],M_{\et,\pid})\fleche(\Fil^{0} M_{\dR})\tenseur_{E}E_{\pid}$ of \eqref{EqExp} is an isomorphism.
\item The cohomology of the complex $\RGamma_{\et}(\Ocal_{F}[1/\Sigma],M_{\et,\pid})$ vanishes outside of degree 1.
\end{enumerate}
Combining the weak Leopoldt's conjecture of \cite{PerrinRiouLpadique} and the conjecture of Bloch-Kato on the order of vanishing yields the following conjecture.
\begin{Conj}\label{ConjCritique}
Let $M$ be a critical motive over $F$ and let $F_{\infty}/F$ be the cyclotomic $\zp$-extension of $F$. Then $M_{K}$ is strictly critical for all finite sub-extensions $F\subset K\subset F_{\infty}$ except possibly finitely many. 
\end{Conj}
\subsubsection{Weight-monodromy and independence of $\ell$}
\paragraph{Independence of $\ell$}The following conjecture (\cite[Conjecture C5]{SerreFacteurs}, see also \cite[Appendix]{SerreTate}) strengthens the independence property noted after \eqref{EqDefEuler}.
\begin{Conj}[Independendance of $\ell$]\label{ConjEll}
Let $\mathfrak q$ and $\mathfrak p$ be two rational primes and let $v\nmid\pid\qid$ be a finite place of $F$. Then the local Euler polynomials 
\begin{equation}\nonumber
\Eul_{v}(M_{\et,\pid},X)=\det\left(1-\Fr(v)X|M_{\et,\pid}^{I_{v}}\right),\ \Eul_{v}(M_{\et,\qid},X)=\det\left(1-\Fr(v)X|M_{\et,\qid}^{I_{v}}\right) 
\end{equation}
have coefficients in $E[X]$ and are equal.
\end{Conj}
This conjecture is known to hold for any $i$ if the dimension of $X$ is at most 2 (\cite{SaitoWeightIndependence}). It also holds for any $i$ if $X$ is an abelian scheme (\cite[Exposé IX]{SGA7}) and more generally if $M$ is a strict 1-motive (\cite{RaynaudMonodromie}). Hence, it holds for general $X$ if $i\leq 1$.
\paragraph{Weight-Monodromy} In \cite[Principe 2.1, 8.1]{DeligneHodge}, \cite[Page 41]{RapoportZink} and especially \cite[Conjecture 3.9]{IllusieMonodromie}, the following property of $M$ is also conjectured.
\begin{Conj}[Weight-Monodromy conjecture]\label{ConjWMC}
Let $v\nmid p$ be a finite place and let $\s(v)$ be a lift of $\Fr(v)$ to $G_{F_{v}}$. Then the semisimple Weil-Deligne representation attached to the $G_{F_{v}}$-action on $M_{\et,\pid}$ for $\pid|p$ is pure of weight $w=i-2j$ in the sense that the eigenvalues of $\s(v)$ on the $r$-th graded piece of the monodromy filtration of $M_{\et,p}$ are Weil numbers of weight $w+r$.
\end{Conj}
In \cite[Question 3.13.2]{IllusieMonodromie}, it is asked whether conjectures \ref{ConjEll} and \ref{ConjWMC} admit the following common generalization.
\begin{Conj}[Weight-Monodromy and independence of $\ell$]\label{ConjIllusie}
Let $v\nmid p$ be a finite place and let $\s(v)$ be a lift of $\Fr(v)$ to $G_{F_{v}}$. Then $\det(1-\s(v)X|\Gr_{r}^{M}M_{\et,\pid})$ has coefficients in $E$ independent of the choice of $\pid$ and all its roots are Weil numbers of weight $w+r$.
\end{Conj}
\subsubsection{Compatible systems of Galois representations}
For $G$ a topological group, a $G$-representation $(T,\rho,R)$ with coefficients in $R$ is an $R$-module $T$ free of finite rank with a continuous morphism
\begin{equation}\nonumber
\rho:G\fleche\Aut_{R}(T).
\end{equation}
Let $L_{\pid}/\qp$ be a finite extension and let $(V,\rho,L_{\pid})$ be a $G_{F}$-representation. Then $V$ is said to be geometric if it is unramified outside a finite set of prime and if $\rho|G_{F_{v}}$ is de Rham for all $v|p$. It is said to be motivic if it is geometric and if it satisfies conjecture \ref{ConjWMC}. It is said to be rational over $E$ if there exists a number field $E$ such that $\det(1-\rho(\Fr(v))X|V^{I_{v}})$ belongs to $E[X]$ for all $v\nmid p$ and it is said to be rational if it is rational over $E$ for some number field $E$. A representation $V$ which is rational over $E$ is said to be part of a compatible system of representations if there exists a finite set $\Sigma$ such that  for all finite place $\qid$ of $E$, there exists an $E$-rational $G_{F}$-representation $(V_{\qid},\rho_{\qid},E_{\qid})$ unramified outside $\Sigma\cup\{\qid\}$ such that $V_{\pid}$ is equal to $V$ and such that the polynomials $\det(1-\rho(\Fr(v))X|V)$ and $\det(1-\rho_{\qid}(\Fr(v))X|V_{\qid})$ are equal for all $v\notin\Sigma\cup\{\pid,\qid\}$ except possibly finitely many. The representation $V$ is said to come from algebraic geometry if it arises as the $p$-adic étale realization of a motivic data $(X,\epsi,i,j)$ over $F$ with coefficients in $E$ and it is said to come from a motive if it comes from algebraic geometry and if the motivic data $(X,\epsi,i,j)$ satisfies conjecture \ref{ConjIllusie}.

These definitions and the properties of realizations of motives recalled above imply the following statements. A representation which comes from a motive comes from algebraic geometry, is rational and belongs to a compatible system of representations. A representation which comes from algebraic geometry is geometric. Conjecture \ref{ConjEll} states that a representation which comes from algebraic geometry is rational and belongs to a compatible system of representations and conjecture \ref{ConjWMC} states that a representation that comes from algebraic geometry is motivic.

In the reverse direction, we record the following two conjectures.
\begin{Conj}[Fontaine-Mazur]\label{ConjFM}
An irreducible, geometric representation comes from algebraic geometry.
\end{Conj}
This is \cite[Conjecture 1]{FontaineMazur}.
\begin{Conj}\label{ConjTaylor}
A compatible system of irreducible representations comes from algebraic geometry.
\end{Conj}
Conjecture \ref{ConjTaylor} seems to be widely believed, see \cite[Conjecture 1.3]{TaylorGalois} for a reference. 
\subsection{$L$-functions of motives and their special values}
\subsubsection{$L$-functions of motives}
Let $K/F$ be an abelian extension with Galois group $G$. Let $M$ be a rank $n$ motive over $F$ with coefficients in $E$ which is pure of weight $w$ and let $M_{K}$ be the motive $M$ with coefficients in $E[\Gal(K/F)]$. Let $\Sigma$ be a finite set of primes of $\Ocal_{F}$ containing $\{v|p\}$. Choose $\pid\nmid 2$ a finite place of $E$ and denote by $\Ocal$ the ring of integer of $E_{\pid}$. For simplicity, we sometimes write $M_{K,\pid}$ for $M_{K,\et,\pid}$. Finally, fix $\tau:E\plonge\C$ an embedding.

The $\Sigma$-partial $\Gal(K/F)$-equivariant $L$-function $L_{\Sigma}(M_{K},s)$ is the formal partial Euler product
\begin{equation}\nonumber
L_{\Sigma}(M_{K},s)=\produit{v\notin\Sigma}{}\frac{1}{\Eul_{v}(M_{K,\pid},(N_{F/\Q}v)^{-s})}=\produit{v\notin\Sigma}{}\frac{1}{\det(1-\Fr(v)X|M_{K,\pid}^{I_{v}})_{X=(N_{F/\Q}v)^{-s}}}
\end{equation}
in which the product is taken over the finite places of $F$.

Let $\tau_{\chi}\in\Hom(E[G],\C)$ be the morphism obtained by extending $\tau$ by linearity through a complex character $\chi\in\hat{G}$. Under our assumption that conjecture \ref{ConjEll} holds, each Euler factor of $L_{\Sigma}(M,s)$ is a polynomial in $(N_{F/\Q}v)^{-s}$ with coefficients in $E[G]$ and thus can be considered as a polynomial with coefficients in $\C$ through $\tau_{\chi}$. Hence $L_{\Sigma}(M,s)$ and the choice of $\chi$ define a complex Euler product which converges to a non-zero holomorphic complex function on the half-plane $\Re(s)>A$ for $A$ sufficiently large (more precisely, on the half-plane $\Re(s)>1+w/2$ under our assumption that conjecture \ref{ConjWMC} holds) and which is equal to an absolutely convergent Dirichlet $L$-series on this half-plane.

Removing the choice of $\chi$ allows us to consider $L_{\Sigma}(M,s)$ as a function on the half-plane $\Re(s)>1+w/2$ with values in $\C[G]$ and further removing the choice of $\tau$ defines a function with values in $E\tenseur\C[G]$. These functions are conjectured to admit a meromorphic continuation to $\C$ satisfying a function equation relating $L_{\Sigma}(M,s)$ and $L_{\Sigma}(M^{*}(1),-s)$ (see \cite{SerreFacteurs,FontaineValeursSpeciales,FontainePerrinRiou,BurnsFlachMotivic} for this conjecture in various degree of generality).

We denote by $L^{*}_{\Sigma}(M^{*}(1),0)\in E\tenseur\C[G]$ the first non-zero term in the Taylor expansion of $L_{\Sigma}(M^{*}(1),s)$ about zero and call it the analytic special value of the $L$-function of $M^{*}(1)$ with Euler factors for primes in $\Sigma$ removed. Because the knowledge of this special value is equivalent to the knowledge of this value after embedding $E$ in $\C$ through all possible embeddings, we henceforth consistently regard $E$ as a subfield of $\C$ though $\tau$.
\subsubsection{Motivic fundamental lines}\label{SubFun}
\paragraph{Complex fundamental line}A \emph{complex fundamental line} of $M_{K}$ is a projective $\C[\Gal(K/F)]$-module $\Delta_{\Sigma}(M_{K})\tenseur_{E,\tau}\C$ of rank 1 equipped with a canonical isomorphism 
\begin{equation}\nonumber
\per_{\C}:\Delta_{\Sigma}(M_{K})\tenseur_{E,\tau}\C\isocan\C[\Gal(K/F)]
\end{equation}
which is equivariant for the action of $\Gal(K/F)$ on the left and right-hand side. Following the fundamental ideas of \cite{DeligneFonctionsL,BeilinsonConjecture} (in the critical and general case respectively), conjectural definitions of complex fundamental lines have been proposed in increasing generality in \cite{BlochKato,FontainePerrinRiou,KatoHodgeIwasawa}.
\paragraph{$p$-adic fundamental line}Let $T\subset M_{K,\pid}$ be a $G_{F}$-stable $\Ocal$-lattice inside the $\pid$-adic étale realization of $M$. Because $E_{\pid}$ and $\Ocal$ are regular local rings and because Shapiro's lemma yields an isomorphism
\begin{equation}\nonumber
\RGamma_{\et}(\Ocal_{K}[1/\Sigma],T)\simeq\RGamma_{\et}(\Ocal_{F}[1/\Sigma],T\tenseur_{\Ocal}\Ocal[G])
\end{equation} of complexes of $\Ocal[G]$-modules, the complexes
\begin{equation}\nonumber
M_{K,\pid}(-1)^{+},\ \RGamma_{\et}(\Ocal_{F}[1/\Sigma],M_{K,\pid}),\ (T(-1)\tenseur_{\Ocal}\Ocal[G])^{+}, \RGamma_{\et}(\Ocal_{K}[1/\Sigma],T)
\end{equation}
are all perfect complexe (of $E_{\pid}[G]$-modules for the first two and $\Ocal[G]$-modules for the last two). Hence, there exists a projective $\Ocal[G]$-module of rank 1
\begin{equation}\nonumber
\Det^{-1}_{\Ocal[G]}\RGamma_{\et}(\Ocal_{K}[1/\Sigma],T)\tenseur_{\Ocal[G]}\produittenseur{\s:F\plonge\C}{}\Det^{-1}_{\Ocal[G]}(T(-1)\tenseur_{\Ocal}\Ocal[G])^{+}
\end{equation}
inside a projective $E_{\pid}[G]$-module of rank 1
\begin{equation}\nonumber
\Det^{-1}_{E_{\pid}[G]}\RGamma_{\et}(\Ocal_{F}[1/\Sigma],M_{K,\pid})\tenseur_{E_{\pid}[G]}\produittenseur{\s:F\plonge\C}{}\Det^{-1}_{E_{\pid}[G]}M_{K,\pid}(-1)^{+}.
\end{equation} 
The $\Ocal[G]$-lattice defined by the former in the latter is independent of the choice of $T$ by Tate's formula for $\Spec\Ocal_{F}[1/\Sigma]$-cohomology with $p$-torsion coefficients (\cite[Theorem 2.2]{TateBSD}).

A \emph{$p$-adic fundamental line} of $M_{K}$ is a projective $E_{\pid}[\Gal(K/F)]$-module $\Delta_{\Sigma}(M_{K})\tenseur_{E}E_{\pid}$ or rank 1 equipped with a canonical isomorphism
\begin{equation}\nonumber
\per_{\pid}:\Delta_{\Sigma}(M_{K})\tenseur_{E}E_{\pid}\isocan\Det^{-1}_{E_{\pid}[G]}\RGamma_{\et}(\Ocal_{F}[1/\Sigma],M_{K,\pid})\tenseur_{E_{\pid}[G]}\produittenseur{\s:F\plonge\C}{}\Det^{-1}_{E_{\pid}[G]}M_{K,\pid}(-1)^{+}
\end{equation}
which is equivariant for the action of $\Gal(K/F)$ on the left and right-hand side.

Note that if $(V,\rho,E_{\pid})$ is a $G_{F,\Sigma_{0}}$-representation for some finite set $\Sigma_{0}\supset\{v|p\}$, then 
\begin{equation}\nonumber
\Delta_{\Sigma}(V/K)\eqdef\Det^{-1}_{E_{\pid}[G]}\RGamma_{\et}(\Ocal_{F}[1/\Sigma],V\tenseur_{E_{\pid}}E_{\pid}[G])\tenseur_{E_{\pid}[G]}\produittenseur{\s:F\plonge\C}{}\Det^{-1}_{E_{\pid}[G]}(V\tenseur_{E_{\pid}}E_{\pid}[G])(-1)^{+}
\end{equation}
is a $p$-adic fundamental line for the identity period map.
\paragraph{Motivic fundamental line}\label{SubMotivicFundamental}Finally, a \emph{motivic fundamental line} of $M_{K}$ is a triplet $$(\Delta_{\Sigma}(M_{K}),\per_{\C},\per_{\pid})$$
where $\Delta_{\Sigma}(M_{K})$ is a projective $E[\Gal(K/F)]$-module $\Delta_{\Sigma}(M_{K})$ of rank 1 such that $\Delta_{\Sigma}(M_{K})\tenseur_{E}\C$ equipped with $\per_{\C}$ is a complex fundamental line and such that $\Delta_{\Sigma}(M_{K})\tenseur_{E}E_{\pid}$ equipped with $\per_{p}$ is a $p$-adic fundamental line. 

Let $M_{K}$ be a strictly critical motive. Then the motivic fundamental line together with the period maps can be explicitly and unconditionally defined as we now recall. Set 
\begin{equation}\nonumber
\Delta_{\Sigma}(M_{K})\eqdef\Det_{E[G]} \Fil^{0}M_{K,\dR}\tenseur_{E[G]}\produittenseur{\s:F\plonge\C}{}\Det^{-1}_{E}M_{K,B,\s}(-1)^{+}.
\end{equation}
Then $\Delta_{\Sigma}(M_{K})$ is endowed with a $G$-equivariant isomorphism
\begin{equation}\nonumber
\per_{\C}:\Delta_{\Sigma}(M_{K})\tenseur_{E[G]}\C[G]\simeq\C[G]
\end{equation}
induced by the isomorphism \eqref{EqCritical} and with a $G$-equivariant isomorphism
\begin{equation}\nonumber
\per_{\pid}:\Delta_{\Sigma}(M_{K})\tenseur_{E[G]}E_{\pid}[G]\simeq\Det^{-1}_{E_{\pid}[G]}\RGamma_{\et}(\Ocal_{F}[1/\Sigma],M_{K,\pid})\tenseur_{E_{\pid}[G]}\produittenseur{\s:F\plonge\C}{}\Det^{-1}_{E_{\pid}[G]}M_{K,\pid}(-1)^{+}
\end{equation}
induced by the isomorphism \eqref{EqExp}. Hence, the triplet $(\Delta_{\Sigma}(M_{K}),\per_{\C},\per_{\pid})$ is a motivic fundamental line. Note the important fact that the definition of the motivic fundamental line and the normalization of the period isomorphisms make no reference to the set $\Sigma$.
\paragraph{Remark:}The main references in the study of special values of $L$-functions of motives (\cite{DeligneFonctionsL,BeilinsonConjecture,BlochKato,FontainePerrinRiou,KatoHodgeIwasawa,BurnsFlach}) all define a motivic fundamental line or an object playing a similar role as well as period maps. However, the objects they consider and especially the period maps attached to them are not always (known to be) the same. It follows from conjectures of J.Tate, S.Bloch and A.Beilinson (\cite{TateConjecture,BeilinsonConjecture}) that the algebraic $K$-theory groups of the proper, smooth variety $X$ define a motivic fundamental line (see especially \cite[Conjecture 2.4.2.1]{BeilinsonConjecture} for the case $X$ admits a proper, regular model over $\Ocal_{F}$). In \cite{BlochCycles}, higher Chow groups are considered and in \cite{FontainePerrinRiou}, the fundamental line is constructed from extensions in the category of mixed motives. To avoid ambiguity, we always normalize the fundamental line and the period maps as above in the strictly critical case. As explained below, conjecture \ref{ConjCritique} and the behavior of weights in $p$-adic families of automorphic representations ensure that defining the fundamental line in the strictly critical case is often enough to define it in general.
\paragraph{$p$-adic fundamental line revisited}
Let $\Sigma\supset\{v|p\}$ be a finite set of finite places and let $\Lambda$ be a semi-local ring whose local factors are rings of mixed characteristic $(0,p)$. It is convenient to expand the definition of a $p$-adic fundamental line slightly in two related directions: first, to the objects in the category $D_{\ctf}(\Spec\Ocal_{F}[1/\Sigma],\Lambda)$ of perfect complexes of constructible étale sheaves of $\Lambda$-modules on $\Spec\Ocal_{F}[1/\Sigma]$  following \cite[Section 3.1]{KatoViaBdR}; second, to $G_{F,\Sigma_{0}}$-representations with coefficients in integral domains.
\begin{DefEnglish}\label{DefFundamental}
Let $\Fcali$ be an object of $D_{\ctf}(\Spec\Ocal_{F}[1/\Sigma],\Lambda)$. The $\Lambda$-adic fundamental line of $\Fcali$ over $K$ is the projective $\Lambda[G]$-module of rank 1
\begin{equation}\nonumber
\Delta_{\Sigma}(\Fcali/K)\eqdef\Det^{-1}_{\Lambda[G]}\RGamma_{\et}(\Ocal_{F}[1/\Sigma],\Fcali\Ltenseur_{\Lambda}\Lambda[G])\tenseur_{\Lambda[G]}\produittenseur{\s:F\plonge\C}{}\Det^{-1}_{\Lambda[G]}(\Fcali\Ltenseur_{\Lambda}\Lambda[G])(-1)^{+}.
\end{equation}
\end{DefEnglish}
Let $(T,\rho,\Lambda)$ be a $G_{F,\Sigma_{0}}$-representation with coefficients in a characteristic zero, semi-local ring $\Lambda$. For all finite set $\Sigma_{0}\supset\Sigma\supset\{v|p\}$ of finite places, $T$ gives rise to an étale sheaf $\Fcali$ on $\Spec\Ocal_{F}[1/\Sigma]$. By construction, $\Fcali$ is then an object of  $D_{\ctf}(\Spec\Ocal_{F}[1/\Sigma_{0}],\Lambda)$. The excision sequence in étale cohomology (\cite[p. 535]{MazurEtale}) reduces the claim that it belongs to $D_{\ctf}(\Spec\Ocal_{F}[1/\Sigma],\Lambda)$, and hence the claim that it is a suitable input for definition \ref{DefFundamental}, to the claim that it belongs to $D_{\ctf}(\Spec\Ocal_{F_{v}},\Lambda)$ for all $v\in\Sigma_{0}\backslash\Sigma$. This holds in particular if $\Lambda$ has finite global homological dimension by the theorem of Auslander-Buchsbaum and Serre.

Suppose now that $\Lambda$ is an integral domain with field of fraction $\Lcal$ and that $T^{I_{v}}$ is of finite projective dimension as $\Lambda$-module for all $v\in\Sigma_{0}\backslash\Sigma$. Denote by $V$ the $G_{F,\Sigma_{0}}$-representation $T\tenseur_{\Lambda}\Lcal$. Assume furthermore for simplicity that for all $v\in\Sigma_{0}\backslash\Sigma$, the eigenvalues of $\Fr(v)$ in its action on $T^{I_{v}}$ are distinct from 1 and belong to $\Lambda$ (as discussed in \ref{SubTwo} below, the latter condition is expected to be true when $T$ comes from a $\Lambda$-adic family of automorphic representations of a reductive group $\G$ split at $v\in\Sigma_{0}\backslash\Sigma$). 
\begin{Prop}\label{PropBienDef}
For all $v\in\Sigma_{0}\backslash\Sigma$, denote by $d_{v}$ the rank of $T^{I_{v}}$ and by $\Eul_{v}(T,1)\in\Lambda_{d_{v}}[X]$ the polynomial $\det(1-X\Fr(v)|V^{I_{v}})$. There is a canonical isomorphism 
\begin{equation}\label{EqExcision}
\Delta_{\Sigma}(T/F)\isocan\Delta_{\Sigma_{0}}(T/F)\tenseur\produittenseur{v\in\Sigma_{0}\backslash\Sigma}{}\left(\Det_{\Lambda}\RGamma_{\et}(F_{v},T)\tenseur_{\Lambda}\Det^{-1}_{\Lambda}\frac{\Lambda}{\Eul_{v}(T,1)\Lambda}\right)
\end{equation}
induced by the excision sequence and a lifting of $1-\Fr(v)$ to a projective resolution of $T^{I_{v}}$ and which is compatible with the natural isomorphism $\RGamma_{\et}(\Ocal_{F_{v}},T)\simeq \Lambda^{d_{v}}/(1-\Fr(v))\Lambda^{d_{v}}$ when $T^{I_{v}}$ is a free $\Lambda$-module on which $\Fr(v)$ acts without the eigenvalue 1.
\end{Prop}
\begin{proof}
Fix $v\in\Sigma_{0}\backslash\Sigma$ and denote by $P_{v}\in\Lambda_{d_{v}}[X]$ the polynomial $\det(1-X\Fr(v)|V^{I_{v}})$. As $\RGamma_{\et}(\Ocal_{F_{v}},T)$ is computed by the complex $[T^{I_{v}}\overset{1-\Fr(v)}{\fleche}T^{I_{v}}]$ placed in degree 0 and 1, it is enough to show that there is a canonical isomorphism
\begin{equation}\nonumber
\Det_{\Lambda}^{-1}[T^{I_{v}}\overset{1-\Fr(v)}{\fleche}T^{I_{v}}]\isocan\Det^{-1}_{\Lambda}[\Lambda\overset{P_{v}(1)}{\fleche}\Lambda].
\end{equation}
A choice of projective resolution of $T^{I_{v}}$, the lifting property of projective modules and the extension of the determinant functor from the category of projective $\Lambda$-modules to the category of perfect complexes of $\Lambda$-modules reduces this problem to the existence of a canonical isomorphism
\begin{equation}\label{EqCanIsoEtape}
C_{1}=\Det^{-1}_{\Lambda}[M\overset{f}{\fleche}M]\isocan\Det^{-1}_{\Lambda}[\Lambda\overset{\det f}{\fleche}\Lambda]=C_{2}
\end{equation}
where $M$ is a free $\Lambda$-module and $f$ acts without the eigenvalue 0 on $M\tenseur_{\Lambda}\Lcal$. As the complexes $C_{1}$ and $C_{2}$ both become acyclic after tensor product with $\Lcal$, the existence of the isomorphism \eqref{EqCanIsoEtape} follows from \cite[Lemma 1]{BurnsFlachTate} (but with opposite sign conventions).
\end{proof}
\begin{DefEnglish}
For all finite place $v\nmid p$, let $\Eul_{v}(T,1)$ be $\det(1-\Fr(v)|V^{I_{v}})$.
\end{DefEnglish}
If we now relax now the hypothesis that $T^{I_{v}}$ is a $\Lambda$-module of finite projective dimension for all $v\notin\Sigma$, $\Fcali$ might not belong to $D_{\ctf}(\Spec\Ocal[1/\Sigma],\Lambda)$. Note, however, that $\Fcali$ nevertheless belongs to $D_{\ctf}(\Spec\Ocal[1/\Sigma_{0}],\Lambda)$ so that the right-hand side of \eqref{EqExcision} remains meaningful. Hence,  it is tempting to take it as definition of the $p$-adic fundamental line.
\begin{DefEnglish}\label{DefFundamentalBis}
Let $\Lambda$ be a noetherian, integral domain. Let $(T,\rho,\Lambda)$ be a $G_{F,\Sigma_{0}}$-representation, let $\Sigma_{0}\supset\Sigma\supset\{v|p\}$ be a finite set finite places and let $K/F$ be an abelian extension with Galois group $G$ which is unramified outside $\Sigma$. For $v\notin\Sigma$, assume that $\Eul_{v}(T/K,1)$ belongs to $\Lambda[G]$. The $\Lambda$-adic fundamental line of $T$ is the rank 1, projective $\Lambda[G]$-module
\begin{equation}\nonumber
\Delta_{\Sigma}(T/K)\eqdef\Delta_{\Sigma_{0}}(T/K)\tenseur\produittenseur{v\in\Sigma_{0}\backslash\Sigma}{}\left(\Det_{\Lambda[G]}\RGamma(F_{v},T\tenseur_{\Lambda}\Lambda[G])\tenseur_{\Lambda[G]}\Det^{-1}_{\Lambda[G]}\frac{\Lambda[G]}{\Eul_{v}(T/K,1)\Lambda[G]}\right).
\end{equation}
\end{DefEnglish} 
Proposition \ref{PropBienDef} ensures that the $\Lambda$-adic fundamental line is well-defined when both definitions \ref{DefFundamental} and \ref{DefFundamentalBis} apply.%
\subsubsection{The Equivariant Tamagawa Number Conjecture}
The following conjecture, called the $p$-part of the Equivariant Tamagawa Number Conjecture, is due to S.Bloch and K.Kato (\cite{BlochKato}) when $K=F$ and to \cite{KatoHodgeIwasawa} in general. It has been further generalized to non-abelian extension $K$ in \cite{BurnsFlach,FukayaKato}. The presentation given here is inspired by \cite{FontaineValeursSpeciales,KatoHodgeIwasawa,FontainePerrinRiou,BurnsFlachMotivic}. \begin{Conj}[Equivariant Tamagawa Number Conjecture]\label{ConjTNC}
For all finite set $\Sigma\supset\{p\}$ of finite places of $\Ocal_{F}$ and all finite abelian extension $K$ of $F$ unramified outside $\Sigma$ with Galois group $G$, there exists a motivic zeta element $\z_{\Sigma}(M_{K})$ which is a basis of a motivic fundamental line $\Delta_{\Sigma}(M_{K})$ verifying the following properties.
\begin{enumerate}
\item\label{ItC} The image of $\z_{\Sigma}(M_{K})\tenseur1$ through $\per_{\C}$ is equal to $L^{*}_{\Sigma}(M_{K}^{*}(1),0)$.
\item\label{ItP} The $\Ocal[G]$-lattice $\Ocal[G]\cdot\per_{\pid}(\z_{\Sigma}(M_{K})\tenseur1)$ inside
\begin{equation}\nonumber
\Det^{-1}_{E_{\pid}[G]}\RGamma_{\et}(\Ocal_{F}[1/\Sigma],M_{K,\et,\pid})\tenseur_{E_{\pid}[G]}\produittenseur{\s:F\plonge\C}{}\Det^{-1}_{E_{\pid}[G]}M_{K,\et,\pid}(-1)^{+}
\end{equation}
is equal to 
\begin{equation}\nonumber
\Delta_{\Sigma}(T/K)=\Det^{-1}_{\Ocal[G]}\RGamma_{\et}(\Ocal_{F}[1/\Sigma],T\tenseur_{\Ocal}\Ocal[G])\tenseur_{\Ocal[G]}\produittenseur{\s:F\plonge\C}{}\Det^{-1}_{\Ocal[G]}(T\tenseur_{\Ocal}\Ocal[G])(-1)^{+}
\end{equation}
for any choice of $G_{F}$-stable $\Ocal$-lattice $T$ inside $M_{K,\et,\pid}$.
\item\label{ItCor} Let $L/F$ be a finite abelian unramified outside $\Sigma$ and containing $K$. Then the natural map
\begin{equation}\nonumber
\Delta_{\Sigma}(M_{L})\tenseur_{E[\Gal(L/F)]}E[\Gal(K/F)]\fleche\Delta_{\Sigma}(M_{K})
\end{equation} 
induced by corestriction is an isomorphism sending $\z_{\Sigma}(M_{L})\tenseur1$ to $\z_{\Sigma}(M_{K})$.
\end{enumerate}
\end{Conj}
Briefly put, the conjecture asserts that the two bases of the motivic fundamental line given respectively by the inverse image of $L^{*}_{\Sigma}(M_{K}^{*}(1),0)\in\C[\Gal(K/F)]$ and by the inverse image of the $p$-adic fundamental line $\Delta_{\Sigma}(T/K)$ coincide. Note also that if either assertion \ref{ItC} or \ref{ItP} is known to hold for $K$ and $L$, then assertion \ref{ItCor} admits an explicit formulation in terms of $L$-functions or $p$-adic fundamental lines.
\subsubsection{The generalized Iwasawa Main Conjecture}
The momentous generalization of conjecture \ref{ConjTNC} carried over in \cite{KatoViaBdR} applies to perfect complexes of constructible étale sheaves of $\Lambda$-modules on schemes $X$ where $\Lambda$ is a $p$-adic ring and $X$ is a finite scheme over $\Spec\Z[1/p]$. In this manuscript, we restrict ourselves for simplicity to the étale sheaves of $\Lambda$-modules on $\Spec\Ocal_{F}[1/\Sigma]$ which arise from irreducible $\Lambda$-adic $G_{F}$-representations.

Let $\Lambda$ be a semi-local, characteristic zero, noetherian ring all whose local factors have residual characteristic $p$. Let $\Sigma_{0}$ be a finite set of finite primes of $\Ocal_{F}$ containing $\{v|p\}$. Let $(T,\varrho,\Lambda)$ be a $G_{F,\Sigma_{0}}$-representation. For $\psi:\Lambda\fleche S$ a ring morphism, the specialization of $T$ at $\psi$ is the $G_{F,\Sigma_{0}}$-representation $T_{\psi}\eqdef T\tenseur_{\Lambda,\psi}S$. A specialization is said to be regular if $\psi$ has values in a characteristic zero, regular, semi-local ring. 

We assume that there exists a non-empty subset $\Spec^{\mot}\Lambda[1/p]\subset\Spec\Lambda[1/p]$, called the motivic locus, such that for all $x\in\Spec^{\mot}\Lambda[1/p]$ with residual field $\kbold(x)$, the fiber $V_{x}\eqdef T\tenseur_{\Lambda}\kbold(x)$ is an irreducible, motivic $G_{F}$-representation. Let $\Spec^{\cl}\Lambda[1/p]\subset\Spec\Lambda[1/p]$ be the subset of points $x$ with residual field $\kbold(x)$ such that the fiber $V_{x}$ comes from algebraic geometry. A specialization $\psi$ or $T_{\psi}$ is motivic (resp. classical) if $\psi$ is of the form $\psi_{x}:\Lambda\fleche\kbold(x)$ with $x\in\Spec^{\mot}\Lambda[1/p]$ (resp. with $x\in\Spec^{\cl}\Lambda[1/p]$), in which case $T_{\psi}$ is equal to $V_{x}$. Conjecture \ref{ConjFM} states that $\Spec^{\mot}\Lambda[1/p]$ is a subset of $\Spec^{\cl}\Lambda[1/p]$ and that a sufficient condition for the two sets to coincide is that all de Rham specializations $T_{\psi}$ of $T$ be irreducible. By construction, motivic and classical specializations are regular.

Fix a finite set $\Sigma\supset\{v|p\}$ of finite places. A specialization $\psi:\Lambda\fleche S$ with values in a characteristic zero ring is said to be $\Sigma$-brittle if either the étale sheaf $\Fcali_{\psi}$ attached to $T_{\psi}$ is an object of $D_{\ctf}(\Spec\Ocal_{F}[1/\Sigma],\Lambda)$ or if $S$ is an integral domain. Regular specializations are $\Sigma$-brittle for all $\Sigma$ and all specializations are $\Sigma$-brittle if $\Sigma\supset\Sigma_{0}$. Let $T$ be a $G_{F,\Sigma_{0}}$-specialization such that the identity is $\Sigma$-brittle. Using definition \ref{DefFundamentalBis} if necessary, the \textit{generalized Iwasawa Main Conjecture} below (\cite[Conjecture 3.2.2]{KatoViaBdR}) applies to $T$. 
\begin{Conj}[Kato's general Iwasawa Main Conjecture]\label{ConjETNC}
For all finite abelian extension $K/F$, all finite set $\Sigma\supset\{v|p\}$ of finite places of $F$ containing all the places of ramification of $K$ and all $\Sigma$-brittle specializations $\psi:\Lambda\fleche S$, there exist an $S[\Gal(K/F)]$-equivariant fundamental line $\Delta_{\Sigma}(T_{\psi}/K)$ and an $S[\Gal(K/F)]$-equivariant zeta element $\z_{\Sigma}(T_{\psi}/K)$ satisfying the following properties.
\begin{enumerate}[label=$\operatorname{(\roman{*})}$]
\item\label{ItDef} The fundamental line $\Delta_{\Sigma}(T_{\psi}/K)$ is the rank 1, projective $S[G]$-module
\begin{equation}\nonumber
\Delta_{\Sigma}(T_{\psi}/K)=\Det^{-1}_{S[G]}\RGamma_{\et}(\Ocal_{K}[1/\Sigma],T_{\psi})\tenseur_{S[G]}\produittenseur{\s:F\plonge\C}{}\Det^{-1}_{S[G]}T_{\psi}(-1)^{+}
\end{equation}
and $\z_{\Sigma}(T_{\psi}/K)$ is a basis of $\Delta_{\Sigma}(T_{\psi}/K)$.
\item\label{ItCores} Let $F\subset K\subset L$ be finite, abelian extensions of $F$. Let $\Sigma'$ be a finite set of finite places of $L$ as above and let $\Sigma$ be the set of finite places of $K$ below the places in $\Sigma'$. Then the natural isomorphism
\begin{equation}\nonumber
\Delta_{\Sigma'}(T_{\psi}/L)\tenseur_{S[\Gal(L/F)]}S[\Gal(K/F)]\isocan\Delta_{\Sigma}(T_{\psi}/K)
\end{equation}
induced by corestriction sends $\z_{\Sigma'}(T_{\psi}/L)\tenseur1$ to $\z_{\Sigma}(T_{\psi}/K)$.
\item\label{ItemSpec} For all brittle specialization $\psi:\Lambda\fleche S$, the natural map of $S[\Gal(K/F)]$-modules
\begin{equation}\nonumber
\Delta_{\Sigma}(T/K)\tenseur_{\Lambda,\psi}S\fleche\Delta_{\Sigma}(T_{\psi}/K)
\end{equation}
is an isomorphism sending $\z_{\Sigma}(T/K)\tenseur1$ to $\z_{\Sigma}(T_{\psi}/K)$.
\item Let $x\in\Spec^{\cl}\Lambda[1/p]$ with values in $E_{\pid}=\kbold(x)$ be a motivic point attached to the motivic data $M(x)$ with coefficients in a number field $E$. There is a canonical isomorphism
\begin{equation}\nonumber
\Delta_{\Sigma}(T_{\psi_{x}}/K)\isocan\Delta_{\Sigma}(M(x)_{K})\tenseur_{E}E_{\pid}
\end{equation}
sending $\z_{\Sigma}(T_{\psi_{x}}/K)\tenseur1$ to the $p$-adic zeta element $\z_{\Sigma}(M(x)_{K})\tenseur1$.
\item For all $x\in\Spec^{\cl}\Lambda[1/p]$, conjecture \ref{ConjTNC} holds for $M(x)_{K}$.
\end{enumerate}
\end{Conj}
\section{Examples}
Actually existing $\Lambda$-adic families of $G_{F,\Sigma}$-representation fall broadly in three categories. Firstly, one can consider a single motive $M$ over $F$ and consider the system of motives with coefficients $\{M\times_{F}K\}_{K\subset L}$ indexed by finite sub-extensions $K$ of a $\zp^{d}[\Delta]$-extension $L$ ($\Delta$ a finite commutative group). When $L$ is the cyclotomic extension of $F$, this construction recovers classical Iwasawa theory of the motive $M$. Secondly, one can consider $\Lambda$-adic families of automorphic representations of a reductive group $\G$ and the corresponding family of $G_{F,\Sigma}$-representation. In that case, $\Lambda$ is typically a local factor of a suitably defined Hecke algebra attached to $\G$. Finally, one can consider the more abstract (or algebraic) families arising as universal deformation of an absolutely irreducible residual $G_{F,\Sigma}$-representation $\rhobar$. 
\subsection{Punctual family}\label{SubPunctual}
The simplest exemple of $\Lambda$-adic family to which conjecture \ref{ConjETNC} applies is the punctual family $(T,\rho,\Ocal)$ given by the choice of a $G_{F}$-stable $\Ocal$-lattice inside the $p$-adic étale realization $M_{\et,\pid}$ of a pure motivic data. In that case, conjecture \ref{ConjETNC} is simply a restatement of conjecture \ref{ConjTNC} but the language it introduces proves useful.

Let $F_{\infty}$ be an abelian extension of $F$ unramified outside $\Sigma$ whose Galois group $\Gamma_{\infty}=\Gamma\times\Delta$ is the product of a free $p$-group $\Gamma\simeq\zp^{d}$ of rank $d\geq1$ with a finite group $\Delta=\Delta_{0}\times\Delta_{p}$ with $\Delta_{p}$ a finite, abelian $p$-group and $\Delta_{0}$ a finite abelian group of order prime to $p$. Let $\Lambda$ be the completed group $\Ocal$-algebra
\begin{equation}\label{EqLambdaAlg}
\Lambda= \Ocal[[\Gamma_{\infty}]]=\sommedirecte{\chi\in\hat{\Delta}_{0}}{}\Ocal_{\chi}[\Delta_{p}][[\Gamma]]\simeq\sommedirecte{\chi\in\hat{\Delta}_{0}}{}\Ocal_{\chi}[\Delta_{p}][[X_{1},\cdots,X_{d}]].
\end{equation}
In \eqref{EqLambdaAlg}, $\hat{\Delta}_{0}$ denotes the group of character of $\Delta_{0}$ and $\Ocal_{\chi}$ denotes the ring generated over $\Ocal$ by the values of $\chi$. The ring $\Lambda$ is semi-local and its local factors of $\Lambda$ are of the form $\Ocal_{\chi}[\Delta_{p}][[\Gamma]]$. Hence, they are $\Ocal_{\chi}$-algebras of Krull dimension $1+d$ which are regular if $\Delta_{p}$ is trivial and complete intersection in general. Let $(F_{\alpha},\cor_{\alpha,\beta})$ be the projective system of finite sub-extensions of $F_{\infty}$ with corestriction as transition maps. Then the inverse limit
\begin{equation}\nonumber
\Delta_{\Sigma}(T/F_{\infty})\eqdef\limproj{F_{\alpha}}\ \Delta_{\Sigma}(T/F_{\alpha})
\end{equation}
coincides by construction with the $\Lambda$-adic fundamental line
\begin{equation}\nonumber
\Delta_{\Sigma}(T\tenseur_{\Ocal}\Lambda/F)=\Det^{-1}_{\Lambda}\RGamma_{\et}(\Ocal_{F}[1/\Sigma],T\tenseur_{\Ocal}\Lambda)\tenseur_{\Lambda}\produittenseur{\s:F\plonge\C}{}\Det^{-1}_{\Lambda}T\tenseur_{\Ocal}\Lambda(-1)^{+}
\end{equation}
of $T\tenseur_{\Ocal}\Lambda$ as defined in statement \ref{ItDef} of conjecture \ref{ConjETNC} (since identity is a $\Sigma$-brittle specialization of $\Lambda$). The fundamental line $\Delta_{\Sigma}(T/F_{\infty})$ admits according to statement \ref{ItCores} of conjecture \ref{ConjETNC} a basis
\begin{equation}\nonumber
\z_{\Sigma}(T/F_{\infty})\eqdef\limproj{F_{\alpha}}\ \z_{\Sigma}(T/F_{\alpha}).
\end{equation}
By construction and Shapiro's lemma, the pair $\z_{\Sigma}(T/F_{\infty})\in\Delta_{\Sigma}(T/F_{\infty})$ is equal to the pair $\z_{\Sigma}(T\tenseur_{\Ocal}\Lambda/F)\in\Delta_{\Sigma}(T\tenseur_{\Ocal}\Lambda/F)$. 

Because $\Lambda$ is unramified outside of $\Sigma$ as $G_{F}$-representation, there is an equality 
\begin{equation}\nonumber
H^{0}(I_{v},T\tenseur_{\Ocal}\Lambda)=H^{0}(I_{v},T)\tenseur_{\Ocal}\Lambda
\end{equation} 
for all $v\notin\Sigma$. Hence, there is a canonical isomorphism
\begin{equation}\nonumber
\RGamma_{\et}(\Ocal[1/\Sigma],T\tenseur_{\Ocal}\Lambda)\Ltenseur_{\Lambda}\Ocal[\Delta]\isocan\RGamma_{\et}(\Ocal[1/\Sigma],T\tenseur_{\Ocal}\Ocal[\Delta])
\end{equation}
given by the quotient by the augmentation ideal and more generally a canonical isomorphism
\begin{equation}\nonumber
\RGamma_{\et}(\Ocal[1/\Sigma],T\tenseur_{\Ocal}\Lambda)\Ltenseur_{\Lambda}S[\Delta]\isocan\RGamma_{\et}(\Ocal[1/\Sigma],T\tenseur_{\Ocal}S[\Delta])
\end{equation}
for all specialization $\Ocal[[\Gamma]]\fleche S$ with values in a characteristic zero domain. Hence, the natural map
\begin{equation}\nonumber
\Delta_{\Sigma}(T/K)\tenseur_{\Lambda,\psi}S\fleche\Delta_{\Sigma}(T_{\psi}/K)
\end{equation}
is an isomorphism for such $\psi$ and statement \ref{ItemSpec} thus holds for all $\psi$ restricting to such specializations.

An important particular case is given by the cyclotomic deformation of a motive (\cite{GreenbergIwasawaMotives}), in which $F_{\infty}$ is the cyclotomic $\zp$-extension of $F$ and $\Lambda$ is the classical Iwasawa-algebra $\Lambda_{\Iw}$. Because $\zp(n)$ is a specialization of $\Lambda_{\Iw}$ for all $n\in\Z$, conjecture \ref{ConjCritique} implies that if there exists $n\in\Z$ such that $M(n)$ is critical (for instance if $M$ is pure of odd weight), then the set of classical specializations $\psi$ of $\Lambda$ such that $M_{\psi}$ is a strictly critical motive is a Zariski-dense subset of $\Spec\Lambda_{\Iw}$. Hence, as long as $M$ admits a critical Tate twist, the normalization of the period maps in the strictly critical case given is enough to specify uniquely $\z_{\Sigma}(T\tenseur_{\Ocal}\Lambda_{\Iw}/F)$.
\subsection{$p$-adic families of automorphic representations}\label{SubAutomorphic}
We now turn to the more interesting second class of $p$-adic families. Under a condition on the support of completed cohomology, which in the case of $\GL_{2}$ amounts to localization at a non-Eisenstein maximal ideal of the reduced Hecke algebra, we construct many $p$-adic families of $G_{F,\Sigma}$-representations parametrized by $p$-adic Hecke algebras (and quotients thereof) which admit a Zariski-dense subset of motivic specializations. 
\subsubsection{Completed cohomology of tower of Shimura varieties}
In this subsection, we consider the $p$-adic families parametrized by $p$-adic Hecke algebras arising from $p$-adic families of automorphic representations of a reductive group $\G$ defined over $F^{+}$. 

Fix $F^{+}$ a number field and recall that $I_{F^{+},\infty}$ and $I_{F^{+},p}$ are the sets of embeddings of $F^{+}$ inside $F^{+}_{v}$ for $v|\infty$ and $v|p$ respectively. Fix an embedding of $F^{+}_{v}$ inside $\Fbar^{+}_{v}$ for each place $v$, an embedding of $\qp$ inside $\Qbar_{p}$ as well as an identification of $\C$ with $\Qbar_{p}$ extending this embedding. These choices identify the sets $I_{F^{+},\infty}$ and $I_{F^{+},p}$. Let $\G$ be a reductive algebraic group defined over $F^{+}$. Denote by $\operatorname{Ram}(\G)$ the finite set of finite places of $F^{+}$ such that $\G(F^{+}_{v})$ is ramified, that is to say the smallest set such that for all $v\notin\operatorname{Ram}(G)$, the group $\G(F^{+}_{v})$ is quasi-split and split over an unramified extension.

Let $U^{\circ}_{\infty}$ be the connected component of the identity of a maximal compact subgroup $U_{\infty}$ of the real points $\G(F^{+}\tenseur_{\Q}\R)$ of $\G$. To $U\subset\G(\A_{F^{+}}^{(\infty)})$ a compact open subgroup of the form $U_{p}U^{p}$ with $U_{p}\subset\G(F^{+}_{p})$ and $U^{p}\subset\G(\A_{F^{+}}^{(p\infty)})$ we attach the double quotient 
\begin{equation}\nonumber
Y(U)=\G(F^{+})\backslash\G(\A_{F^{+}})/U^{\circ}_{\infty}U
\end{equation}
and consistently assume that $U$ is sufficiently small so that $Y(U)$ is endowed with the structure of a Shimura variety of dimension $d$ (we allow $d$ equal to zero). Denote by $\Acal(\G(F^{+})\backslash\G(\A_{F^{+}}))$ and $\Acal(\G(F^{+})\backslash\G(\A_{F^{+}})/U)$ respectively the space of $U_{\infty}$-finite automorphic forms on $\G(F^{+})\backslash\G(\A_{F^{+}})$ and the space of $U_{\infty}$-finite automorphic forms on $\G(F^{+})\backslash\G(\A_{F^{+}})$ of level $U$. 

Consider $U$ as in the previous paragraph. If $M$ is a representation of $U_{p}\subset\G(F^{+}_{p})$, then $U$ acts on the right on $M$ by $m\cdot u=(u_{v}^{-1}m)_{v|p}$ and so $M$ defines a local system
\begin{equation}\nonumber
\Fcal_{M}\eqdef(W\times(\G(F^{+})\backslash\G(\A_{F^{+}})/U_{\infty}^{\circ}))/U
\end{equation}
on $Y(U)$. In particular, let $W$ be a finite dimensional algebraic complex representation of $\G(F^{+}_{\infty})$ and fix $E$ a finite extension of $\qp$ with ring of integers $\Ocal$ and over which $\G$ splits. By our fixed choice of isomorphisms, we may and do view $W$ as having coefficients in $\Qbar_{p}$. As $\G$ splits over $E$, we may further view it as having coefficients in $E$. Restricting $W$ to $U_{p}\subset\G(F^{+}_{p})\simeq\G(F^{+}_{\infty})$ thus defines a local system $\Fcal_{W}$ of $E$-vector spaces over $Y(U)$. Likewise, if $W_{0}$ is an $\Ocal$-lattice inside $W$ and if $U_{p}$ is sufficiently small, then $W_{0}$ is stable by $U_{p}$ (and so is an $\Ocal[U_{p}]$-modules). The local systems $\Fcal_{W_{0}}$ and $\Fcal_{W_{0}/\varpi^{s}}$ of $\Ocal$ and $\Ocal/\varpi^{s}$-modules respectively on $Y(U)$ are thus well-defined. We consider the cohomology group with compact support $H^{d}_{c}(Y(U)(\C),\Fcal_{W})$ as well as the cohomology groups $H^{d}_{c}(Y(U)(\C),\Fcal_{W_{0}})$ and $H^{d}_{c}(Y(U)(\C),\Fcal_{W_{0}/\varpi^{s}})$ when $U_{p}$ is sufficiently small.

Fix $U$ a compact open subgroup as above. Let $\Sigma(U^{p})\supset\operatorname{Ram}{\G}$ be the finite set of finite places such that $U^{p}_{v}\cap\G(F^{+}_{v})$ is not a maximal hyperspecial compact open subgroup of $\G(F^{+}_{v})$. Denote by $\Sigma$ a finite set of finite places containing $\Sigma(U^{p})\cup\{v|p\}$. The abstract local spherical Hecke algebra at $v\notin\Sigma$ is the $\Ocal$-algebra 
\begin{equation}\nonumber
\Hecke_{v}\eqdef\Ocal[U_{v}\backslash\G(F^{+}_{v})/U_{v}]
\end{equation}
and the abstract spherical Hecke algebra is the restricted tensor product 
\begin{equation}\nonumber
\Hs=\underset{v\notin\Sigma}{\bigotimes}{'}\Hecke_{v}.
\end{equation}
The Satake isomorphism shows that $\Hecke_{v}\tenseur_{\Ocal} E$ is a commutative algebra. So is therefore $\Hs\tenseur_{\Ocal}E$. For simplicity, we henceforth assume that there exists a finite Galois extension $F/F^{+}$ such that for all $v\notin\Sigma$, there exists a finite place $w|v$ of $F$ such that $\G(F^{+}_{v})$ splits over $K_{w}$. If $v\notin\Sigma$ moreover splits completely in $K$, then we fix an isomorphism $\iota_{w}:\G(F^{+}_{v})\simeq\GL_{n}(K_{w})$ such that $U_{v}$ is the inverse image of $\GL_{n}(\Ocal_{F_{w}})$. For
\begin{equation}\nonumber
g_{w,i}=\operatorname{diag}(\underset{i}{\underbrace{\varpi_{w},\cdots,\varpi_{w}}},1,\cdots,1),
\end{equation}
denote by $T_{i}(w)$ the element $[U_{v}\backslash \iota^{-1}_{w}( g_{w,i})/U_{v}]$ of $\Hecke_{v}$. Then the Satake isomorphism is an isomorphism between $\Hecke_{v}\tenseur_{\Ocal} E$ and $E[T_{i}^{\pm1};1\leq i\leq n]$ sending $T_{i}(w)$ to $T_{i}$. In that case, the reciprocal Satake polynomial at $w|v$ is defined to be the polynomial
\begin{equation}\nonumber
1+\cdots+(-1)^{j}(N_{F/\Q}w)^{j(j-1)/2}T_{j}(w)X^{j}+\cdots+(-1)^{n}(N_{F/\Q}w)^{n(n-1)/2}T_{n}(w)X^{n}.
\end{equation}
Denote by $\Ssplit$ the set of finite places places of $K$ above the finite places $v\notin\Sigma$ of $F^{+}$ which splits completely in $F$.

The abstract spherical Hecke algebra acts on $H^{d}_{c}(Y(U)(\C),\Ocal)$ and $\Acal(\G(F^{+})\backslash\G(\A_{F^{+}}))$ (and the latter action preserves $\Acal(\G(F^{+})\backslash\G(\A_{F^{+}})/U)$). Denote by $\Hs(U)$ the image of $\Hs$ inside the endomorphism of $H^{d}_{c}(Y(U)(\C),\Ocal)$.



An automorphic representation $\pi$ of $\G(\A^{(\infty)}_{F})$ is an irreducible $\G(\A^{(\infty)}_{F})$-representation arising as a quotient of a sub-representation of $\Acal(\G(F)\backslash\G(\A_{F}))$. An automorphic representation $\pi$ is said to contribute to cohomology if there exists a complex representation $W$ of $\G$ as above such that $\pi$ is a sub-quotient of $\liminj{U}\!\ H^{d}_{c}(Y(U)(\C),\Fcal_{W})$.

If $\pi=\otimes{'}\pi_{v}$ is an automorphic representation of $\G(\A_{F}^{(\infty)})$ and if $v\notin\Sigma$, then $\pi_{v}^{U_{v}}$ is a $\Qbar_{p}$-vector space of dimension 1 on which $\Hecke_{v}$ acts. Hence, to $\pi$ is attached a morphism $\lambda_{\pi}:\Hs\fleche\Qbar_{p}$. The image of $\lambda_{\pi}$ is contained in the ring of integers $\Ocal_{\pi}$ of a finite extension of $\Qbar_{p}$. For $L_{\pid}$ a finite extension of $\qp$ containing $\Ocal_{\pi}$, we say that a $G_{F,\Sigma}$-representation $(V,\rho,L_{\pid})$ is adapted to $\pi$ if for all finite place $w\in\Ssplit$, the Euler factor $\Eul_{v}(V/K,X)$ of $V$ is equal to the reciprocal Satake polynomial at $v$ or, more concretely, if 
\begin{equation}\label{EqAdapted}
\det(\Id-\rho(\Fr(w))X)=1-\lambda_{\pi}(T_{1}(w))X+\cdots+(-1)^{n}(N_{F/\Q}w)^{n(n-1)/2}\lambda_{\pi}(T_{n}(w))X^{n}
\end{equation}
for all $w\in\Ssplit$. More generally, if $\lambda:\Hsm\fleche S$ is a local $\Ocal$-algebras morphism from a local factor of $\Hs$ to a ring $S$, we say that a $G_{F,\Sigma}$-representation $(V,\rho,S)$ is adapted to $\lambda$ if
\begin{equation}\label{EqAdaptedLambda}
\det(\Id-\rho(\Fr(w))X)=1-\lambda(T_{1}(w))X+\cdots+(-1)^{n}(N_{F/\Q}w)^{n(n-1)/2}\lambda(T_{n}(w))X^{n}
\end{equation}
for all finite place $w\in\Ssplit$. As $\Ssplit$ is of Dirichlet density 1 in the set of finite places of $F$, if an absolutely irreducible representation $V$ is adapted to $\pi$, then $V$ is uniquely determined up to isomorphism. By \cite{CarayolRepresentationsGaloisiennes}, $V$ is characterized by the property that $\tr(\rho(\Fr(w)))=\lambda(T_{1}(w))$ for all $w\in\Ssplit$. If $(V,\rho,L_{\pid})$ is adapted to $\pi$, there exists a $G_{F,\Sigma}$-representation $(T,\rho,\Ocal_{L_{\pid}})$ such that $T\tenseur_{\Ocal_{L_{\pid}}}L_{\pid}$ is equal to $V$. To such a $T$ is attached a residual representation $(\Tbar,\rhobar,\Ocal_{\pi}/\mgot_{\Ocal_{\pi}})$ adapted to $\lambda_{\pi}\modulo\mgot_{\Ocal_{\pi}}$. We say that $\rho$ is residually absolutely irreducible if $\rhobar$ is absolutely irreducible, in which case the representations $\Tbar$ and $T$ are unique up to isomorphism and characterized by the property that, for all $v\notin\Sigma$, $\tr(\Fr(v))=\lambda_{\pi}(T(v))\modulo\mgot_{\Ocal_{\pi}}$ and $\tr(\Fr(v))=\lambda_{\pi}(T(v))$ respectively.

As in \cite{EmertonCompleted}, for $0\leq i\leq 2d$, we introduce various groups computing the completed cohomology with tame level $U^{p}$ of the tower of varieties $\{Y(U_{p}U^{p})\}_{U_{p}}$.
\begin{align}\nonumber
\Htilde^{i}_{c}(U^{p},\Ocal/\varpi^{s})&=\liminj{U_{p}}\ H^{i}_{c}(Y(U_{p}U^{p})(\C),\Ocal/\varpi^{s})\\\nonumber
\Htilde^{i}_{c}(U^{p},\Ocal)&=\limproj{s}\ \liminj{U_{p}}\ H^{i}_{c}(Y(U_{p}U^{p})(\C),\Ocal/\varpi^{s})\\\nonumber
\Htilde^{i}_{c}(U^{p},E)&=\left(\limproj{s}\ \liminj{U_{p}}\ H^{i}_{c}(Y(U_{p}U^{p})(\C),\Ocal/\varpi^{s})\right)\tenseur_{\Ocal}E\\\nonumber
\Htildetilde^{i}(U^{p},\Ocal/\varpi^{s})&=\limproj{U_{p}}\ H^{i}(Y(U_{p}U^{p})(\C),\Ocal/\varpi^{s})\\\nonumber
\Htildetilde^{i}(U^{p},\Ocal)&=\limproj{s}\ \limproj{U_{p}}\ H^{i}(Y(U_{p}U^{p})(\C),\Ocal/\varpi^{s})\\\nonumber
\Htildetilde^{i}(U^{p},E)&=\left(\limproj{s}\ \limproj{U_{p}}\ H^{i}(Y(U_{p}U^{p})(\C),\Ocal/\varpi^{s})\right)\tenseur_{\Ocal}E
\end{align}
By \cite[Corollary 2.2.27]{EmertonCompleted}, the cohomology group $\Htilde^{d}_{c}(U^{p},\Ocal)$ is also the $\varpi$-adic completion of $\liminj{U_{p}}\ H^{d}_{c}(Y(U_{p}U^{p},\Ocal))$.

Poincaré duality
\begin{equation}\nonumber
H^{i}_{c}(Y(U_{p}U^{p})(\C),\Ocal/\varpi^{s})\times H^{2d-i}(Y(U_{p}U^{p})(\C),\Ocal/\varpi^{s})\fleche\Ocal/\varpi^{s}
\end{equation}
induces a duality between $\Htilde^{d}_{c}(U^{p},-)$ and $\Htildetilde^{d}(U^{p},-)$ for $-=\Ocal/\varpi^{s},\Ocal$ or $E$.

Again for $-=\Ocal/\varpi^{s},\Ocal$ or $E$, the inverse limit
\begin{equation}\nonumber
\Hs(U^{p})\eqdef\limproj{U_{p}}\ \Hs(U_{p}U^{p})
\end{equation}
acts on $\Htilde_{c}^{i}(U^{p},-)$ and $\Htildetilde^{i}(U^{p},-)$. The ring $\Hs(U^{p})$ is a product of finitely many complete, local rings. A morphism $\lambda:\Hs(U^{p})[1/p]\fleche\Qbar_{p}$ is said to be automorphic if it factors through the morphism $\lambda_{\pi}:\Hs(U_{p}U^{p})[1/p]\fleche\Qbar_{p}$ attached to an automorphic representation $\pi$ which contributes to cohomology. If $\lambda_{\pi}$ is an automorphic specialization of $\Hs(U^{p})[1/p]$, we say that $\pi$ contributes to $\Htilde^{d}_{c}(U^{p},\Ocal)$. A maximal ideal $\mgot$ of $\Hs(U^{p})$ is said to be automorphic if there exists an automorphic morphism $\lambda$ which factors through $\Hs(U^{p})_{\mgot}$ or, more concretely, if there exist a local system $\Fcal_{W}$ on $Y(U_{p}U^{p})$ for $U_{p}$ sufficiently small and an automorphic representation $\pi$ contributing to $\liminj{U_{p}}\!\ H^{d}_{c}(Y(U_{p}U^{p})(\C),\Fcal_{W})_{\mgot}$.
\subsubsection{Interpolation of system of Hecke eigenvalues}
We henceforth assume the following properties of $\G$ and $U^{p}$. 
\begin{HypEnglish}\label{HypEisenstein}
There exists an automorphic maximal ideal $\mgot$ of $\Hs(U^{p})$ satisfying the following properties.
\begin{enumerate}
\item If $\pi$ is an algebraic automorphic representation of $\G(\A_{F^{+}}^{(\infty)})$ contributing to $\Htilde^{d}_{c}(U^{p},\Ocal)_{\mgot}$, then there exists a residually absolutely irreducible $G_{F,\Sigma}$-representation $\rho_{\pi}$ adapted to $\pi$. \item For all $i\neq d$, the cohomology group $\Htilde^{i}_{c}(U^{p},\Ocal)$ vanishes after localization at $\mgot$.
\end{enumerate}
\end{HypEnglish}
%
The significance of assumption \ref{HypEisenstein} lies in the following proposition.
\begin{Prop}\label{PropInterpolation}
Let $\Hsm$ be the localization of $\Hs(U^{p})$ at an automorphic maximal ideal satisfying assumption \ref{HypEisenstein}. Then there exists an absolutely irreducible $G_{F,\Sigma}$-representation $(T,\rho,\Hsm)$ characterized by the property
\begin{equation}\nonumber
\tr(\rho(\Fr(w)))=T(w)
\end{equation}
for all $w\in\Ssplit$. There exists a Zariski-dense subset $\Spec^{\aut}\Hsm[1/p]\subset\Spec\Hsm[1/p]$ such that the fibre at $x\in\Spec^{\aut}\Hsm[1/p]$ of $T$ is an absolutely irreducible $G_{F,\Sigma}$-representation adapted to an automorphic representation $\pi$ contributing to $\Htilde^{d}_{c}(U^{p},\Ocal)_{\mgot}$.

More generally, if $\lambda:\Hsm\fleche S$ is map of local $\Ocal$-aglebras, then there exists an absolutely irreducible $G_{F,\Sigma}$-representation $(T_{\lambda},\rho_{\lambda},S)$ characterized by the property
\begin{equation}\nonumber
\tr(\rho_{\lambda}(\Fr(w)))=\lambda(T_{1}(w))
\end{equation} 
for all $w\in\Ssplit$.
\end{Prop}
\begin{proof}
As preliminary for the proof, we recall that if $W$ is a representation of an open subgroup $K_{p}$ of $\G(F^{+}_{p})$, then a vector $v$ of a $K_{p}$-representation $V$ is said to be locally $W$-algebraic if there exists an open subgroup $H$ of $K_{p}$, a natural number $n$, and an $H$-equivariant homomorphism $W^{n}\fleche V$ whose image contains $v$. The $K_{p}$-locally algebraic vectors of $V$ are the vectors which are $W$-locally algebraic for some representation $W$ (see \cite[Definition 4.2.1 and Proposition-Definition 4.2.6]{EmertonAnalytic}). We denote by $\Ccal^{\operatorname{la}}(K_{p},E)$ the vector space of $E$-valued locally analytic function on $K_{p}$ (\cite[page 413]{EmertonLocally}) and recall that there is a natural injection with dense image from $\Ccal^{\operatorname{la}}(K_{p},E)$ in the $E$-vector space $\Ccal(K_{p},E)$ of continuous $E$-valued functions on $K_{p}$.

Fix $s>0$ and $U_{p}$ a sufficiently small compact open subgroup of $\G(F^{+}_{p})$. Thanks to statement 2 of assumption \ref{HypEisenstein}, the cohomology group $H^{d}_{c}(U_{p}U^{p},\Ocal/\varpi^{s})_{\mgot}$ is the only non-zero cohomology group of a complex computed by an injective resolution, and hence an injective object in the category of smooth representation of $U_{p}$. Hence, the dual of $\Htilde^{d}_{c}(U^{p},\Ocal/\varpi^{s})_{\mgot}$, which is $\Htildetilde(U^{p},\Ocal/\varpi^{s})_{\mgot}$, is a projective $\Ocal/\varpi^{s}[[U_{p}]]$-module. As a consequence, $\Htilde^{d}(U^{p},E)_{\mgot}$ is a projective $E\tenseur_{\Ocal}\Ocal[[U_{p}]]$-module and so is a free $E\tenseur_{\Ocal}\Ocal[[U_{p}]]$-module. The functor
\begin{equation}\nonumber
\Hom_{E\tenseur_{\Ocal}\Ocal[[U_{p}]]}(\Htildetilde^{d}(U^{p},E)_{\mgot},-)
\end{equation}
is thus exact. Fix $K_{p}$ a maximal compact open subgroup of $\G(F^{+}_{p})$. Because $K_{p}/U_{p}$ is finite, taking $K_{p}/U_{p}$-invariant is an exact functor and so 
\begin{equation}\nonumber
\Hom_{E\tenseur_{\Ocal}\Ocal[[U_{p}]]}(\Htildetilde^{d}(U^{p},E)_{\mgot},-)^{K_{p}/U_{p}}\simeq\Hom_{E\tenseur_{\Ocal}\Ocal[[K_{p}]]}(\Htildetilde^{d}(U^{p},E)_{\mgot},-)
\end{equation}
is exact. Hence $\Htildetilde^{d}(U^{p},E)_{\mgot}$ is a projective $E\tenseur_{\Ocal}\Ocal[[K_{p}]]$-module and so a direct summand of a module isomorphic to $E\tenseur_{\Ocal}\Ocal[[K_{p}]]^{r}$ for some $r$. In turn, $\Htilde^{d}_{c}(U^{p},\Ocal)_{\mgot}$ is then a direct summand of a module isomorphic to $\Ccal(K_{p},E)^{r}$. As $K_{p}$-locally algebraic vectors are dense in $\Ccal^{\operatorname{la}}(K_{p},E)$ and so in $\Ccal(K_{p},E)$, they are dense in $\Htilde^{d}_{c}(U^{p},\Ocal)_{\mgot}$.

Let as above $W$ be a finite dimensional representation of $\G(F^{+}_{\infty})$ with coefficients in $E$ and denote by $\check{W}$ the contragredient representation of $W$. Because all higher $\Ext$ groups vanish by statement 2 of assumption \ref{HypEisenstein}, \cite[Corollary 2.2.18]{EmertonCompleted} gives an isomorphism 
\begin{equation}\label{EqIsomLocalAlg}
\left(\liminj{U_{p}}\ H^{d}_{c}(Y(U_{p}U^{p})(\C),\Fcal_{W})_{\mgot}\right)\tenseur_{E}\check{W}\simeq\Htilde_{c}^{d}(U^{p},\Ocal)_{\mgot,\check{W}-\textrm{l.alg}}.
\end{equation}
Consider an element $t\in\Hsm(U^{p})$ in the intersection of the kernels of all automorphic $\lambda$. By statement 1 of assumption \ref{HypEisenstein}, $t$ annihilates $H^{d}_{c}(U_{p}U^{p},\Fcal_{W})_{\mgot}$ for all $W$. Thus $t$ annihilates $\Htilde_{c}^{d}(U^{p},\Ocal)_{\mgot,\check{W}-\textrm{l.alg}}$ by \eqref{EqIsomLocalAlg} and so it annihilates $\Htilde^{d}_{c}(U^{p},\Ocal)_{\mgot}$ by density. As $\Hsm(U^{p})$ acts faithfully on $\Htilde^{d}_{c}(U^{p},\Ocal)_{\mgot}$ by definition, $t$ is equal to zero.

Patching the determinant laws\footnote{That is to say the determinants in the sense of \cite{ChenevierDeterminant}. G.Chenevier recommends the terminology we use here.} attached to the $G_{F,\Sigma}$-representations $\rho_{\pi}$ adapted to the automorphic points of $\Spec\Hsm(U^{p})$ thus yields first a $G_{F,\Sigma}$-representation $(T(\aid),\rho(\aid),R(\aid))$ attached to any quotient $R(\aid)=\Hsm(U^{p})/\aid$ of $\Hsm(U^{p})$ by a minimal ideal $\aid$, then a $G_{F,\Sigma}$-representation $(T,\rho,\Hsm)$ as in the statement of the proposition and finally a $G_{F,\Sigma}$-representation $(T_{\lambda},\rho_{\lambda},S)$ for any $\lambda:\Hsm(U^{p})\fleche S$ as in the proposition.

By construction, the representation $T_{\lambda}$ satisfies $\tr(\rho_{\lambda}(\Fr(w)))=\lambda(T_{1}(w))$ for all $w\in\Ssplit$ and is characterized by this property since $\rhobar_{\mgot}$ is absolutely irreducible.
\end{proof}


\subsubsection{Examples}\label{SubExamples}
We list important examples of groups $\G$ such that $\Hs$ admits an automorphic maximal satisfying assumption \ref{HypEisenstein}.
\paragraph{Quaternion algebras}Let $F^{+}$ be a totally real field and let $\G$ be the unit group of a quaternion algebra $B$ such that there is at most one infinite place $\tau$ of $F^{+}$ such that $B\tenseur_{F^{+},\tau}\R\simeq\GL_{2}(\R)$. Assumption \ref{HypEisenstein} is then satisfied if (and only if) $\mgot$ is not Eisenstein in the sense of \cite[Proposition 2]{DiamondTaylor}. If $B$ is totally definite, then its Shimura variety is zero-dimensional so its cohomology vanishes outside degree 0 and the assertion about $\mgot$ is easy. If there is exactly one real place splitting $B$, then the cohomology in degree 0 and 2 of a quaternionic Shimura curve is of residual type in the sense of \cite[Section 5.1]{FujiwaraDeformation} and hence vanishes after localization at $\mgot$ if $\mgot$ is not Eisenstein. In both cases, one can choose $F=F^{+}$.
\paragraph{Totally definite unitary group}Let $F^{+}$ be a totally real field and $F$ a CM extension of $F^{+}$. We assume that all places $v|p$ of $F^{+}$ split completely in $F$ and that the extension $F/F^{+}$ is unramified at all finite places. Let $n$ be a strictly positive integer which we assume to be odd if $[F^{+}:\Q]$ is even. Let $\G$ be a unitary group in $n$ variables over $F^{+}$ which is split by $F$, quasi-split at all finite places and such that $\G(F^{+}\tenseur_{\Q}\R)$ is compact. Let $U^{p}=\produit{v\nmid p\infty}{}U_{v}$ be a compact, open subgroup of $\G(\A_{F^{+}}^{(p\infty)})$ and let $\Sigma_{0}$ be a finite set of finite places of $F^{+}$ containing all places $v|p$ and all places $v$ such that $U_{v}$ is not a hyperspecial maximal subgroup of $\G(F^{+}_{v})$. Assume that all places $v\in\Sigma_{0}$ split completely in $K$. As the Shimura variety attached to such a $\G$ is zero-dimensional, an automorphic ideal $\mgot$ satisfies assumption \ref{HypEisenstein} if $\rhobar_{\mgot}$ is absolutely irreducible. 
\paragraph{General symplectic group of genus 2}Let $F$ be $\Q$ and $\G$ be $\GSp_{4}$. Then \cite{GenestierTilouine} shows that assumption \ref{HypEisenstein} is satisfied for a restrictive class of maximal ideals $\mgot$.

\section{Two results on $p$-adic fundamental lines}\label{SubTwo}
Fix $F$ a number field and finite set of finie places $\{v|p\}\subset\Sigma\subset\Sigma_{0}$. In this section, we consider $p$-adic families of $G_{F,\Sigma_{0}}$-representations $(T,\rho,S)$ satisfying the following properties.
\begin{HypEnglish}\label{HypMain}
\begin{enumerate}
\item The ring $S$ is of the form $\Lambda[G]$ where $\Lambda$ is a complete, local noetherian ring with residual field $k$ and where $G$ is the finite Galois group of an abelian extension $K$ of $F$ unramified outside $\Sigma$.
\item\label{ItIrr}The residual representation $\rhobar:G_{F,\Sigma_{0}}\fleche\GL_{n}(k)$ is absolutely irreducible.
\item\label{ItLanglands} For all $v\nmid p$, the characteristic polynomial of $\Fr(v)$ acting on $T^{I_{v}}$ has coefficients in $S$.
\item The set $\Spec^{\mot}S[1/p]$ of motivic specializations of $S$ pure of weight non-zero is non-empty. 
\end{enumerate}
\end{HypEnglish}
\subsection{The algebraic part}
\subsubsection{The main theorem}
Assume that $S=\Lambda[G]$ and $(T,\rho,S)$ satisfy assumption \ref{HypMain}. Let $\{v|p\}\subset\Sigma\subset\Sigma_{0}$ be a subset of finite places. 
\begin{DefEnglish}\label{DefAlgebraicPart}
The algebraic part of the $p$-adic fundamental line $\Delta_{\Sigma}(T/F)$ is the projective $S$-module 
\begin{equation}\nonumber
\Xcali_{\Sigma}(T/F)\eqdef\Det^{-1}_{S}\RGamma_{\et}(\Ocal_{F}[1/\Sigma_{0}],T)\tenseur\produittenseur{v\in\Sigma_{0}\backslash\Sigma}{}\left(\Det_{S}\RGamma_{\et}(F_{v},T)\tenseur_{S}\Det^{-1}_{S}\frac{S}{\Eul_{v}(T,1)S}\right).
\end{equation}
\end{DefEnglish}
Note that the last assertion of assumption \ref{HypMain} guarantees that $\Eul(T,1)$ is non-zero.

The following theorem shows that algebraic parts commute with motivic (and more generally monodromic) specializations, and hence establishes \cite[Conjecture 3.2.2 (i)]{KatoViaBdR} at motivic points. 
\begin{TheoEnglish}\label{TheoPositif}
Assume first that $\Sigma=\Sigma_{0}$. Then the natural map 
\begin{equation}\nonumber
\Xcali_{\Sigma}(T/F)\tenseur_{\Lambda,\psi}\Lambda'\fleche\Xcali_{\Sigma}(T_{\psi}/F)
\end{equation}
is an isomorphism.

Assume now that $\Lambda$ is an integral domain and that $\Sigma$ is arbitrary. Let $\psi:\Lambda\fleche\Lambda'$ be a specialization with values in a characteristic zero integral domain such that there exists a motivic specialization $\pi:\Lambda\fleche\Ocal_{\pi}$ pure of non-zero weight factoring through $\psi$. Then the natural map 
\begin{equation}\nonumber
\Xcali_{\Sigma}(T/F)\tenseur_{\Lambda,\psi}\Lambda'\fleche\Xcali_{\Sigma}(T_{\psi}/F)
\end{equation}
is an isomorphism.
\end{TheoEnglish}
\begin{proof}
As all the places of ramification of $T$ belong to $\Sigma$, the first assertion follows from the general base-change property of $\RGamma_{\et}(\Ocal_{F}[1/\Sigma],-)$. Now we assume that $\Sigma\subsetneq\Sigma_{0}$.

Without the hypotheses of the theorem on $\Lambda$ and $\psi$, the natural change of coefficients map induces canonical isomorphisms 
\begin{equation}\nonumber
\RGamma_{\et}(\Ocal_{F}[1/\Sigma_{0}],T)\Ltenseur_{\Lambda,\psi}\Lambda'\isocan\RGamma_{\et}(\Ocal_{F}[1/\Sigma_{0}],T_{\psi})
\end{equation}
and 
\begin{equation}\nonumber
\RGamma_{\et}(F_{v},T)\Ltenseur_{\Lambda}\Lambda'\simeq\RGamma_{\et}(F_{v},T_{\psi})
\end{equation}
for all $v\in\Sigma_{0}\backslash\Sigma$. 

Moreover, every eigenvalue of $\Fr(v)$ acting on $T^{I_{v}}$ yields by specialization an eigenvalue of $\Fr(v)$ acting on $T_{\pi}^{I_{v}}$ (note that the converse is not true \textit{a priori}). As $T_{\pi}$ is pure of non-zero weight, all the eigenvalue of $\Fr(v)$ acting on $T^{I_{v}}$ are different from 1 and so the same is true for $T$.

Hence, the claim that there is an isomorphism
\begin{equation}\nonumber
\Xcali_{\Sigma}(T/F)\tenseur_{\Lambda,\psi}\Lambda'\isocan\Xcali_{\Sigma}(T_{\psi}/F)
\end{equation}
naturally induced by $\psi$ reduces to the claim that the Euler factor $\Eul_{v}(T,1)$ satisfies the change of coefficients property
\begin{equation}\nonumber
\psi(\Eul_{v}(T,1))=\Eul_{v}(T_{\psi},1)
\end{equation}
for all $v\in\Sigma_{0}\backslash\Sigma$. This in turn is implied by the statement that the rank of $T^{I_{v}}$ over $\Lambda$ is equal to the rank of $T_{\psi}^{I_{v}}$ over $\Lambda'$ for all $v\in\Sigma_{0}\backslash\Sigma$. Fix such a $v$. 

The representation $V_{\pi}=T_{\pi}\tenseur_{\Ocal}\Frac(\Ocal)$ admits a weight filtration $W_{\bullet}$ and a monodromy filtration $M_{\bullet}$ which coincide by the hypothesis that $\pi$ is motivic. Let $V$ be $T\tenseur_{\Lambda}\Lambda[1/p]$. By an easy extension of the arguments of \cite[Appendix]{SerreTate}, there exist a finite extension $K_{w}$ of $F_{v}$ and a nilpotent element $N$ of $\End(V)$ such that for all $\tau\in I_{w}$, $\rho(\tau)$ is equal to $\exp(t_{\zeta,p}(\tau)N)$ where $t_{\zeta,p}$ is the usual map from $I_{w}$ to $\zp$ given by the product on all primes $\ell$ such that $w\nmid\ell$ of the cyclotomic characters and projection onto $\zp$ (this map depends on the choice of a compatible system of roots of unity $\zeta=(\zeta_{n})_{\ell\nmid n}$). There exists a unique finite increasing filtration $M_{\bullet}$ on $V$ such that $NM_{i}\subset NM_{i-2}$ and such that $N^{s}$ induces an isomorphism between $\Gr^{M}_{s}V$ and $\Gr_{s}^{M}V$. By \cite[Theorem 4.1]{SahaPreprint}, the natural map $V\fleche V_{\pi}$ induced by $\pi$ defines an isomorphism $\Gr_{s}^{M}V\simeq\Gr_{s}^{W}V_{\pi}$ from the graded piece of the monodromy filtration on $V$ to the corresponding graded piece of the weight filtration on $V_{\pi}$. As $\Gr_{s}^{W}V_{\pi}$ is equal to $\Gr_{s}^{M}V_{\pi}$ by the assumption that $V_{\pi}$ is pure, $\pi$ induces an isomorphism $\Gr_{s}^{M}V\simeq\Gr_{s}^{W}V_{\pi}$. In particular, $\rank_{\Ocal_{\pi}}T_{\pi}^{I_{v}}=\rank_{\Lambda}T^{I_{v}}$.

By hypothesis, the diagram
\begin{equation}\nonumber
\xymatrix{
\Lambda\ar[rr]^{\psi}\ar[rd]_{\pi}&&\Lambda'\ar[ld]\\
&\Ocal_{\pi}&
}
\end{equation}
commutes so $\rank_{\Lambda}T^{I_{v}}\leq\rank_{\Lambda'}T_{\psi}^{I_{v}}\leq\rank_{\Ocal_{\pi}}T_{\pi}^{I_{v}}$. Hence, the rank of $T^{I_{v}}$ over $\Lambda$ is equal to the rank of $T_{\psi}^{I_{v}}$ over $\Lambda'$.
\end{proof}
\paragraph{Remarks:}(i) Let $\Lambda$ be an integral domain and let $\psi:\Lambda\fleche S$ be a specialization with values in a characteristic zero domain. The proof of theorem \ref{TheoPositif} makes clear that the compatibility of $\Xcali_{\Sigma}(T/F)$ with $-\tenseur_{\Lambda,\psi}S$ actually holds as long as, for all $v\in\Sigma_{0}\backslash\Sigma$, the rank of the nilpotent element $N$ defining the monodromy filtration on $V$ is equal to the rank of the nilpotent element $N_{\psi}$ defining the monodromy filtration on $S$. Indeed, in that case there exists a specialization $\phi$ of $S$ with values in $\Qbar_{p}$ such that the weight filtration and monodromy filtration on $T_{\phi\circ\psi}$ coincide and the proof carries over unchanged. The set of kernels of specializations $\psi$ satisfying these properties is Zariski-dense in $\Spec\Lambda[1/p]$.

(ii) Theorem \ref{TheoPositif} can be thought of a general version of Mazur's Control Theorem on Selmer groups (\cite{MazurRational,OchiaiControl,GreenbergControl} and much subsequent works). Another approach towards this kind of theorem has consisted in proving the surjectivity of the localization map (see \cite{GreenbergVatsal} for instance). This latter approach quickly meets quite formidable challenge beyond the case of $\GL_{2}$ whereas ours is completely general (but relies on the properties of the monodromy filtration on $p$-adic families established in \cite{SahaPreprint}).
\begin{CorEnglish}
Assume that $\Lambda$ is an integral domain and that $\psi$ is a motivic specialization pure of non-zero weight. Then \cite[Conjecture 3.2.2 (i)]{KatoViaBdR} holds for $\psi$, \textit{i.e} the natural map
\begin{equation}\nonumber
\Delta_{\Sigma}(T/F)\tenseur_{\Lambda,\psi}S\fleche\Delta_{\Sigma}(T_{\psi})
\end{equation}is an isomorphism.
\end{CorEnglish}
\begin{proof}
In the formulation of \cite{KatoViaBdR}, the non-algebraic part of the fundamental line of $T$ is simply $\Det^{-1}_{\Lambda} T(-1)^{+}$ and so commutes with arbitrary specializations as $p$ is odd. The assertion thus follows directly from theorem \ref{TheoPositif}.
\end{proof}
\subsubsection{Proof of corollary \ref{CorIntro}}
We adopt the notations and conventions of section \ref{SubAutomorphic}. Let $w$ be a finite place of $K$ over a $v\in\Sigma$ which splits completely in $K$. Let $\aid$ be a minimal prime ideal of $\Hsm(U^{p})$ and let $R(\aid)$ be the quotient $\Hsm(U^{p})/\aid$. In the following proposition, we investigate the properties of the $G_{F_{w}}$-representation $\rho$ of proposition \ref{PropInterpolation} at a place $w|v$ of $F$ such that $v$ splits completely in $F$ and such that $\G$ is split by $F_{w}$ (and hence by $F_{v}^{+}$).
\begin{Prop}\label{PropLocal}
Assume that there exists a Zariski-dense subset $Z$ of $\Spec \Raid[1/p]$ such that the fiber $(V_{x},\rho_{x},\Qbar_{p})$ of $T(\aid)$ at $x$ is a $G_{F,\Sigma}$-representation adapted to an automorphic representation $\pi$ which is pure at $w$ and satisfies local-global compatibility at $v$ in the sense that $\Eul_{w}(V_{x},p^{-s})=L(\pi_{v},s)$. Then $\Eul_{w}(T(\aid),X)$ belongs to $R(\aid)[X]$ and satisfies
\begin{equation}\nonumber
\psi(\Eul_{w}(T(\aid),X))=\Eul_{w}(T_{\psi},X)
\end{equation}
for all $\psi:R(\aid)\fleche S$ with values in a characteristic zero domain $S$ and such that the weight and monodromy filtration on $T(\aid)\tenseur_{R(\aid),\psi}S[1/p]$ coincide. In particular, $\Eul_{w}(T_{\psi},X)$ belongs to $S[X]$ for all such $\psi$
\end{Prop}
\paragraph{Remark:}The statement that $\Eul_{w}(T_{\psi},X)$ belongs to $S[X]$ (rather than to a partial normalization of $S$) is not \textit{a priori} obvious even for $S=R(\aid)$: for a general $G_{F,\Sigma_{0}}$-representation $T$, $T_{\psi}^{I_{w}}$ might not be a direct summand of $T_{\psi}$ over $S$.
\begin{proof}
Let $\psi$ be a specialization of $R(\aid)$ whose kernel is in $Z$. By definition, $\psi$ is then the system of Hecke eigenvalues attached to an automorphic representation $\pi$ of $\G(\A_{F^{+}}^{(\infty)})$. Define $D_{\psi}^{I_{w}}$ to be the determinant law attached to the $G_{F_{w}}/I_{w}$-representation $V_{x}^{I_{w}}$. Likewise, consider $(V(\aid),\rho(\aid),\Frac(R(\aid)))$ the $G_{F,\Sigma}$-representation $T(\aid)\tenseur_{R(\aid)}\Frac(R(\aid))$ and define $D^{I_{w}}:G_{F_{w}}/I_{w}\fleche \Rcal(\aid)$ to be the determinant law (with values \textit{a priori} in the normalization $\Rcal(\aid)$ of $R(\aid)$) equal to $\det(\rho(\aid)|V(\aid)^{I_{w}})$. As $V_{x}$ is a pure $G_{F_{w}}$-representation by assumption, the composition $\psi\circ D^{I_{w}}$ is equal to $D_{\psi}^{I_{w}}$ by \cite[Theorem 4.1]{SahaPreprint}.

The local factor $\pi_{v}$ of $\pi$ is an automorphic representation of $\GL_{n}(F_{v}^{+})$ and hence its $L$-factor is defined and is a polynomial in $p^{-s}$. We claim that \cite[Theorem 3.3]{GodementJacquet} entails that this polynomial has coefficients in $\psi(\Hsm(U^{p}))$. If $\pi_{v}$ is not absolutely cuspidal, its $L$-factor is a product of $L$-factors of automorphic representations of $\GL_{m}(F_{v}^{+})$ so it is enough to prove the claim for absolutely cuspidal representation and for characters of $\GL_{1}(F_{v}^{+})$. In the first case, the $L$-factor is equal to 1 while in the second case, the representation itself is equal to $\psi$.  As $\pi$ is assumed to satisfy local-global compatibility at $w$, the Euler factor $\Eul_{w}(T_{\psi},X)=D_{\psi}^{I_{w}}(1-\Fr(w)X)$ is equal to the $L$-factor of $\pi_{v}$ and so belongs to $\psi(\Hsm(U^{p}))[X]$. 

The ring morphism 
\begin{equation}\nonumber
\phi:\Hsm(U^{p})\fleche\produit{\ker\psi\in Z}{}\Qbar_{p}
\end{equation}
sending $x$ to $(\psi(x))_{\psi}$ is injective with closed image and induces an homeomorphism between $\Hsm(U^{p})$ and $\phi(\Hsm(U^{p}))$. By the first part of the proof, the map 
\begin{equation}\nonumber
D_{\phi}:G_{F_{w}}\fleche\produit{\ker\psi\in Z}{}\Qbar_{p}
\end{equation}
defined by
\begin{equation}\nonumber
g\mapsto\produit{\ker\psi\in Z}{}D^{I_{w}}_{\psi}(g)
\end{equation}which sends $g\in G_{F_{w}}$
is a determinant law with values in $\phi(\Hsm(U^{p}))$. The determinant law
\begin{equation}\nonumber
\phi^{-1}\circ D_{\phi}:\Hsm(U^{p})[G_{F_{w}}]\fleche\Hsm(U^{p})
\end{equation}
is then the unique continuous map coinciding with $D^{I_{w}}_{\psi}$ for all $\psi$ with $\ker\psi\in Z$ and so is equal to $D^{I_{w}}$. The latter thus has values in $\Hsm(U^{p})$.

Consider now $\psi:R(\aid)\fleche S$ with values in a characteristic zero domain $S$ and such that the weight and monodromy filtration on $T(\aid)\tenseur_{R(\aid),\psi}S[1/p]$ coincide. Then the proof of theorem \ref{TheoPositif} and the argument above shows that $\psi\circ D^{I_{w}}$ is equal to $D_{\psi}^{I_{w}}$. This implies the remaining statements of the lemma.
\end{proof}
Note that the proof of the previous proposition also establishes a local-global compatibility principle for the $G_{F,\Sigma}$-representations $T(\aid)$. 

We now establish a precise formulation of corollary \ref{CorIntro}. Statement \ref{Item3comp} below is a precise statement of the compatibility: if the ETNC with coefficients in $R(\aid)$ is true, then the classical Iwasawa main conjecture is true at automorphic point and the TNC is true for automorphic motives of large enough weight.
\begin{TheoEnglish}\label{TheoAutomorphe}
Let $F$ and $\G$ be as in the first two cases of section \ref{SubExamples} (\textit{i.e} either the unit group of a quaternion algebra over a totally real field $F$ or a totally definite unitary group over $F^{+}$ split by $F$). Let $T(\aid)$ be the $G_{F,\Sigma}$-representation with coefficients in $R(\aid)=\Hsm(U^{p})/\aid$ of proposition \ref{PropInterpolation}. Let $\{v|p\}\subset\Sigma'\subset\Sigma$ be a finite set of primes and let $K/F$ be a finite abelian extension unramified outside $\Sigma'$. Then the natural map
\begin{equation}\label{EqIsomCor}
\Delta_{\Sigma'}(T(\aid)/K)\tenseur_{R(\aid),\psi}S\fleche\Delta_{\Sigma'}(T_{\psi}/K)
\end{equation} 
induced by $\psi:R(\aid)\fleche S$ is an isomorphism provided that $S$ is a characteristic zero domain, that the weight and monodromy filtrations on $T_{\psi}\tenseur_{S}S[1/p]$ coincide and that $T_{\psi}$ has no eigenvalues equal to 1. In particular, the following holds.
\begin{enumerate}
\item\label{Item1aut} If $\psi$ is an automorphic point of large enough weight, then \eqref{EqIsomCor} is an isomorphism.
\item\label{Item2iw} If $F_{\infty}/F$ is a $\zp^{d}$-extension, $\psi_{0}:R(\aid)\fleche S$ is an automorphic point, $\Lambda_{\Iw}=S[[\Gal(F_{\infty}/F)]]$ and $T_{\psi}=T_{\psi_{0}}\tenseur_{S}\Lambda_{\Iw}$, then \eqref{EqIsomCor} is an isomorphism.
\item\label{Item3comp} Suppose $\z_{\Sigma}(T/K)$ is a basis of $\Delta_{\Sigma}(T/K)$ defined by the property that
\begin{equation}\nonumber
\per_{\C}\circ\per_{p}^{-1}(x(\z_{\Sigma}(T/K)))=L_{\Sigma}(M(x)_{K}^{*}(1),0)
\end{equation}
for all $x$ in Zariski-dense subset of automorphic points with values in $\Qbar_{p}$. Then $\z_{\Sigma}(T_{\psi}/K)$ is a basis of $\Delta_{\Sigma}(T_{\psi}/K)$ for all $\psi$ as in \ref{Item1aut} and \ref{Item2iw} above.
\end{enumerate}
\end{TheoEnglish}
\begin{proof}
It suffices to check the hypotheses of theorem \ref{TheoPositif} and proposition \ref{PropLocal}.

The non-emptiness of the automorphic locus of $\Spec\Hsm(U^{p})$ is part of the proof of proposition \ref{PropInterpolation} which in fact proves its Zariski-density.

When $\G$ us the unit group of a quaternion algebra, the fact that automorphic points are motivic follows in increasing generality from \cite{LanglandsModular,OhtaAutomorphic,CarayolHilbert} and finally \cite[Theorem 1]{BlasiusRamanujan}. When $\G$ is a unitary group, it follows in increasing generality from \cite{ClozelUnitary,HarrisTaylor,TaylorYoshida,ShinAnnals,ChenevierHarris} and finally \cite[Theorem 1.2]{CaraianiWMC}.
\end{proof}
\subsection{Kato's general conjecture is (probably) false for (not so) trivial reasons}
As soon as we relax the hypotheses of theorem \ref{TheoPositif}, \cite[Conjecture 3.2.2 (i)]{KatoViaBdR} and statement \ref{ItemSpec} of conjecture \ref{ConjETNC} admit counter-examples, as we now explain.

\paragraph{Hypothesis on the target}Suppose first we relax the hypothesis that $\Lambda'$ be a characteristic zero domain. Suppose $(T,\rho,\zp)$ is a $G_{\Q}$-representation which is ramified at $\ell\neq p$, that the Euler factor of $T$ at $\ell$ is a $p$-adic unit, that $\rhobar$ is unramified at $\ell$ and that the characteristic polynomial of $\rhobar(\Fr(\ell))$ vanishes at 1 modulo $p$. Then the natural map 
\begin{equation}\nonumber
\Delta_{\{p\}}(T/\Q)\tenseur_{\zp}\Fp_{p}\fleche\Delta_{\{p\}}(\Tbar/\Q)
\end{equation}
is not an isomorphism. For a concrete example, take $p=3$ and $T=T(f)(1)$ the Tate twist of the $p$-adic $G_{\Q,\{2,3,5,41\}}$-representation attached to the eigencuspform $f_{1}\in S_{2}(\Gamma_{0}(1640))$ attached to the elliptic curve $y^2=x^{3}-31307x-1717706$ and $\ell=41$. Then $T$ is residually absolutely irreducible and $\rhobar$ is unramified at $\ell$ with Euler factor $1+\ell\equiv0\modulo p$.

It seems clear, however, that \cite{KatoViaBdR} at least implicitly disallows this kind of specialization when $\Sigma$ does not contain all primes of ramification of $T$ (this can be inferred from \cite[Section 3.2.4]{KatoViaBdR} which characterizes zeta elements only for the fundamental lines with $\Sigma$ containing all the primes of ramification of $T$). A genuine counter-example to \cite[Conjecture 3.2.2 (i)]{KatoViaBdR} should thus be a specialization with values in a characteristic zero domain.

\paragraph{Hypothesis on $\psi$}Now we relax the hypothesis that the monodromy and weight filtrations on $T_{\psi}\tenseur_{S}S[1/p]$ coincide. As this cannot happen in a punctual family, the simplest potential setting to consider is a non-punctual family of rank 2 $G_{\Q,\Sigma}$-representations. We show that already in this setting, counter-examples are (probably) pervasive.

Assume that $\G$ is the group $\GL_{2}$ over $F=\Q$. Fix $U^{p}$ a compact open subgroup of $\G(\A_{\Q}^{(p\infty)})$ and $\Sigma$ a finite set of finite primes containing $p$ and the primes at which $U^{p}$ is not a maximal compact open subgroup. Let $\mgot$ be an automorphic maximal ideal of $\Hs(U^{p})$ satisfying the hypotheses of assumption \ref{HypEisenstein} (in that case, this simply amounts to requiring that the residual $G_{\Q,\Sigma}$-representation attached to $\mgot$ be irreducible and odd by the proof of Serre's conjecture). Assume in addition that there exists an automorphic specialization $\lambda_{f}:\Hsm(U^{p})\fleche\Qbar_{p}$, or more concretely an eigencuspform $f\in S_{k}(U_{p}U^{p},\epsi)$, such that the $G_{\Q,\Sigma}$-representation $\rho_{f}=\rho_{\lambda_{f}}$ restricted to $G_{\qp}$ is ordinary $p$-distinguished in the sense that it is an extension
\begin{equation}\nonumber
\suiteexacte{}{}{\lambda_{f,p}}{\rho_{{f}}|G_{\qp}}{\lambda_{f,p}^{-1}\epsi\chi_{\cyc}^{k-1}}
\end{equation}
the Jordan-H\"older constituents of which are non-isomorphic modulo $\mgot$. Let $\Hsm(U_{p}U^{p})_{k}$ be the image of $\Hsm$ in $\End(S_{k}(U_{p}U^{p})_{\mgot})$. Under the hypothesis above on $\lambda_{f}$, it is well-known that $\Hsm(U_{p}U^{p})_{k}$ is generated by the traces of $\rho_{f}(\Fr(\ell))$ for $\ell\notin\Sigma$, that $\Hsm(U_{p}U^{p})_{k}$ contains $T(p)$, the image inside $\End(S_{k}(U_{p}U^{p})_{\mgot})$ of the double coset $[U_{p}\matrice{1}{0}{0}{p}U_{p}]$, and that $\lambda_{f,p}$ is the unramified character of $G_{\qp}$ sending $\Fr(p)$ to $T(p)$ (see for instance \cite[II.2]{WilesFermat}). By \cite{HidaAnnalesENS,HidaInventionesOrdinary}, $\Hsm(U^{p})$ then admits a $p$-torsion free, reduced quotient $\Hsmord(U^{p})$ (or $\Hsmord$ for simplicity of notation) of Krull dimension 2 such that the set $\Spec^{\cl}\Hsmord[1/p]$ of the kernels of automorphic specialization $\lambda_{f}$ such that $\rho_{f}$ is ordinary $p$-distinguished form a Zariski-dense subset of $\Spec\Hsmord[1/p]$. As in the proof of the proposition \ref{PropInterpolation}, there thus exists a $G_{\Q,\Sigma}$-representation $(T^{\ord},\rho^{\ord},\Hsmord)$ interpolating the ordinary $p$-distinguished $G_{\Q,\Sigma}$-representation $\rho_{f}$ and, for all quotient $R(\aid)^{\ord}$ of $\Hsmord$ by one of its minimal prime ideal $\aid$, a $G_{\Q,\Sigma}$-representation $(T(\aid)^{\ord},\rho(\aid)^{\ord},R(\aid)^{\ord})$. By \cite[Corollary 3.7]{HidaAmericanCongruence}, if $\lambda$ is an automorphic specialization of $\Hsmord(U^{p})$, there is a unique minimal prime ideal $\aid$ such that $\lambda$ factors through $\Hsmord(U^{p})/\aid$.

We assume that there exists a prime $\ell\in\Sigma$ such that the rank of $H^{0}(I_{\ell},T^{\ord})$ over $\Hsmord(U^{p})$ is exactly 1 whereas $\rhobar_{\mgot}$ is unramified at $\ell$ (this is easily achieved by enlarging $U^{p}$ thanks to \cite[Theorem 1]{RibetRaising}). By construction and theorem \ref{TheoPositif}, there then exists a minimal prime ideal $\aid_{1}\in\Spec\Hsmord(U^{p})$ such that $H^{0}(I_{\ell},T(\aid_{1})^{\ord})$ is of rank 1. Moreover, proposition \ref{PropLocal} then guarantees that $\Eul_{\ell}(T(\aid_{1})^{\ord},1)$ belongs to $R(\aid_{1})^{\ord}$ and commutes with motivic specializations. We assume that there exists a motivic specialization $\psi$ such that $\Eul_{v}(T_{\psi},1)$ is a unit in $\bar{\Z}_{p}\croix$, in which case $\Eul_{\ell}(T(\aid_{1})^{\ord},1)$ is not a unit in $R(\aid_{1})^{\ord}$. As $\rhobar_{\mgot}$ is unramified at $\ell$, there exists by \cite[Theorem 1.1]{RibetEpsilon} a motivic specialization $\psi$ of $\Hsmord(U^{p})$ which is unramified at $\ell$. Hence, there exists a minimal prime $\aid_{2}\in\Spec\Hsmord(U^{p})$ such that $T(\aid_{2})^{\ord}$ is unramified at $\ell$. We assume that that there exists a motivic specialization $\phi$ such that $\Eul_{v}(T_{\phi},1)$ is not a unit in $\bar{\Z}_{p}\croix$, in which case $\Eul_{\ell}(T(\aid_{2})^{\ord},1)$ is not a unit in $R(\aid_{1})^{\ord}$ by proposition \ref{PropLocal}. We note that examples of pair of minimal ideals satisfying the properties above abound.

In order to ensure that our putative counter-example to conjecture \ref{ConjETNC} does not stem from our extension of the definition of the $p$-adic fundamental line to the étale sheaves on $\Spec\Z[1/p]$ attached to the $T(\aid_{i})^{\ord}$ (which are not known to belong to $D_{\ctf}(\Spec\Z[1/p],R(\aid_{i})^{\ord})$ in general), we further assume that $R(\aid_{i})^{\ord}$ is a regular local ring. The action of the inverse system of diagonal correspondences
\begin{equation}\nonumber
\left([U_{p}U^{p}\matrice{a}{0}{0}{a}U_{p}U^{p}]\right)_{U_{p}}
\end{equation}
for $a\in(\A_{\Q}^{(\infty)})\croix$ on $\Htilde^{1}(U^{p},\Ocal)$ endows $\Hsmord(U^{p})$ with a structure of free $\Lambdaf$-algebra where $\Lambdaf$ is isomorphic to $\zp[[1+p\zp]]\simeq\zp[[X]]$. In order to show that $R(\aid_{i})^{\ord}$ is a regular local ring, it is thus enough to show that it is isomorphic to $\Lambdaf$. This can be checked by exhibiting an automorphic point of $\Spec R(\aid_{i})^{\ord}[1/p]$ attached to an eigencuspform which is not congruent modulo $p$ to any eigencuspform of the same weight and level (again, examples of pair of minimal ideals satisfying this property abound).

By \cite[Proposition 5.2.2]{EmertonPollackWeston}, the two irreducible components $\Spec R(\aid_{1})^{\ord}$ and $\Spec R(\aid_{2})^{\ord}$ of $\Spec\Hsmord$ meet at some height 1 prime $x$ of $\Spec\Hsmord$. Assume that $x$ has values in a characteristic zero domain $\Ocal_{x}$ and let $T_{x}$ be the attached $G_{\Q,\Sigma}$-representation. Then $T_{x}$ is a specialization of $T(\aid_{2})^{\ord}$ so is unramified at $\ell$ and $\Eul_{\ell}(T_{x},1)=x(\Eul_{\ell}(T(\aid_{2})^{\ord},1))\notin\Ocal_{x}\croix$. As $x(\Eul_{\ell}(T(\aid_{1})^{\ord},1))$ belongs to $\Ocal_{x}\croix$ while $\Eul_{\ell}(T_{x},1)$ does not, the natural map 
\begin{equation}\nonumber
\Xcali_{\Sigma\backslash\{\ell\}}(T(\aid_{1})^{\ord}/\Q)\tenseur_{R(\aid_{1})^{\ord},x}\Ocal_{x}\fleche\Xcali_{\Sigma\backslash\{\ell\}}(T_{x}/\Q)
\end{equation}
is not an isomorphism and so neither can be 
\begin{equation}\nonumber
\Delta_{\Sigma\backslash\{\ell\}}(T(\aid_{1})^{\ord}/\Q)\tenseur_{R(\aid_{1})^{\ord},x}\Ocal_{x}\fleche\Delta_{\Sigma\backslash\{\ell\}}(T_{x}/\Q)
\end{equation}
under the choice of étale cohomology as non-algebraic part of the fundamental line. Experts seem to expect that for any maximal ideal $\mgot$ as above, there exists a finite set $\Sigma$ and a prime $\ell\in\Sigma$ such that there exists irreducible components of $\Spec\Hsm$ and a characteristic zero specialization $x$ giving a counter-example to statement \ref{ItemSpec} of conjecture \ref{ConjETNC} because of the failure of the local Euler factor at $\ell$ to commute with $x$.
\paragraph{An example:}Let $p$ be the prime number 5. Let $U^{p}$ be the subgroup of $\G(\A_{\Q}^{(p\infty)})$ which is maximal outside $\Sigma=\{2,5,7,13\}$ and equal to 
\begin{equation}\nonumber
U^{p}_{\ell}=\left\{g\in\GL_{2}(\Z_{\ell})|g\equiv\matrice{*}{*}{0}{*}\modulo\ell^{v_{\ell}(364)}\right\}
\end{equation}
for $\ell\in\Sigma\backslash\{5\}$. Let $\mgot$ the maximal ideal attached to the unique $G_{\Q,\Sigma}$-representation
\begin{equation}\nonumber
\rhobar_{\mgot}:G_{\Q,\Sigma}\fleche\GL_{2}(\Fp_{5})
\end{equation}
of Serre weight 2, tame level 52 and such that 
\begin{equation}\nonumber
\tr(\Fr(\ell))=\begin{cases}
0&\textrm{if $\ell=3,$}\\
-2&\textrm{if $\ell=11,$}\\
1&\textrm{if $\ell=17,$}\\
-1&\textrm{if $\ell=19.$}\\
\end{cases}
\end{equation}
Denote respectively by $f_{1}$ and $f_{2}$ the eigencuspforms respectively attached to the elliptic curves
\begin{equation}\nonumber
E_{1}:y^{2}=x^{3}-584x+5444,\ E_{2}:y^{2}=x^{3}+x-10.
\end{equation}
Then $f_{1}$ is new of level $364$ while $f_{2}$ is new of level 52. For $i\in\{1,2\}$, we let $M(f_{i})(1)$ be the pure motivic data $(E_{i},\Id,1,0)$ and $\lambda_{i}$ be the specialization of $\Hsmord(U^{p})$ attached to $M(f_{i})(1)$. Then $\lambda_{i}$ is by construction motivic. The Euler factor $\Eul_{7}(T_{\lambda_{1}},1)$ is equal to $-2\in\Z_{5}\croix$ whereas $\Eul_{7}(T_{\lambda_{2}},1)$ is equal to $10/7\in 5\Z_{5}$. Hence there exists a (unique) pair $(\aid_{1},\aid_{2})$ of distinct minimal prime ideals $\aid_{i}\in\Spec\Hsmord(U^{p})$ such that $\lambda_{i}$ factors through $R(\aid_{i})^{\ord}=\Hsmord(U^{p})/\aid_{i}$. Hida's Control Theorem for Hecke algebra (\cite[Theorem 1.2]{HidaInventionesOrdinary}), Nakayama's lemma and and the fact that both $S_{2}(\Gamma_{0}(52),\zp)_{\mgot}$ and $S_{2}(\Gamma_{0}(364),\zp)_{\mgot}$ are free $\zp$-module of rank 1 show that the structure morphism $\Lambdaf\plonge R(\aid_{i})^{\ord}$ is a surjection and thus that $R(\aid_{i})^{\ord}$ is a regular local ring (and is even isomorphic to $\zp[[X]]$). The $G_{\Q,\Sigma}$-representation $T(\aid_{1})^{\ord}$ is ramified at $\ell=7$ and its Euler factor is a unit in $R(\aid_{1})^{\ord}$ while $T(\aid_{2})^{\ord}$ is unramified at 7 and its Euler factor is not a unit in $R(\aid_{2})^{\ord}$.

Due to the difficulty of computing non-classical modular forms, it is not known (or rather not published) whether the last sufficient condition for the specialization of $R(\aid_{1})^{\ord}$ attached to $x$ to be an actual (rather than probable) counter-example to statement \ref{ItemSpec} of conjecture \ref{ConjETNC} holds, that is to say whether the point $x\in\Spec\Hsmord(U^{p})$ at which $\Spec R(\aid_{1})^{\ord}$ and $\Spec R(\aid_{2})^{\ord}$ meet is of characteristic zero. However, this is certainly expected.
\section{A revised conjecture}
Let $(T,\rho,\Hsmo(U^{p}))$ and $(T(\aid),\rho(\aid),R(\aid))$ be the $G_{F,\Sigma_{0}}$-representations defined in section \ref{SubAutomorphic} (in particular $\mgot$ satisfies assumption \ref{HypEisenstein} and $\aid\in\Spec\Hsmo$ is minimal). In this section, we take $\Sigma\subset\Sigma_{0}$ to be the set $\{v|p\}$ for notational simplicity (note that in view of the proof of theorem \ref{TheoPositif}, this is the hardest case anyway). Denote by $V(\aid)$ the $G_{F,\Sigma_{0}}$-representation $T(\aid)\tenseur_{R(\aid)}\Frac(R(\aid))$. Assume further that, as in proposition \ref{PropLocal}, for all $w\in\Sigma_{0}\backslash\{v|p\}$, $w$ is above a place $v$ of $F^{+}$ which splits completely in $F$ and that $\G$ is split by $F_{w}=F_{v}^{+}$. 

\subsection{$p$-adic interpolation of local automorphic representations}
The first step, that is to say the $p$-adic interpolation of local automorphic representation alongside irreducible components of Hecke algebras, relies on results and conjectures of M.Emerton and D.Helm (\cite{EmertonHelm}). 
\subsubsection{Review of the generic Local Langlands Correspondence}\label{SubBreuilSchneider}
Consider $(V,\rho,K)$ an $n$-dimensional $G_{F_{w}}$-representation, where $w\nmid p$ is a finite place of $F$ and $K$ is a field containing $\qp$. We recall the Local Langlands Correspondence normalized as in \cite[Section 4]{BreuilSchneider} (see also \cite[Section 4.2]{EmertonHelm}) that attaches to $\rho$-or more precisely to the Frobenius semisimple Weil-Deligne representation $(\rho,N)$-a $\GL_{n}(F_{w})$-representation $(V_{\pi(\rho)},\pi(\rho),K)$. Denote by $\nInd_{P}^{\GL_{n}}$ the normalized parabolic induction functor of \cite[Section 1.8]{BernsteinZelevinsky} and assume first that $K$ is a finite extension of $\qp$.

We consider first the case where $(\rho,N)$ is absolutely irreducible. Then the unitary Local Langlands Correspondence (after a choice of a square root of the residual characteristic of $F_{w}$) attaches to $\rho$ a supercuspidal representation $\pi^{u}(\rho)$ of $\GL_{n}(F_{w})$ with coefficients in $\Qbar_{p}$. In that case, $\pi(\rho)$ is the unique $\GL_{n}(F_{w})$-representation with coefficients in $K$ such that $\pi(\rho)\tenseur_{K}\Qbar_{p}=\pi^{u}(\rho)\tenseur_{\Qbar_{p}}(\cardinal{\ }\circ\det)^{\frac{1-n}{2}}$ (this does not depend of our choice of square root).

The next case is when $n=md$ and there exists an absolutely irreducible $m$-dimensional Weil-Deligne representation $(\rho',N)$ with coefficients in $K$ such that $(\rho,N)$ is the basic indecomposable representation
\begin{equation}\nonumber
\Sp_{n,d}(\rho')=(\rho'_{0}\oplus\cdots\oplus\rho'_{d},N),
\end{equation}
where $\rho'_{i}=\rho'\tenseur(|\ |^{i}\circ\det)$ and $N$ maps isomorphically $\rho'_{i}$ onto $\rho'_{i+1}$ and is zero on $\rho'_{d}$. In that case, consider the representation $\pi^{u}(\rho')\tenseur\pi^{u}(\rho')(1)\tenseur\cdots\tenseur\pi^{u}(\rho')(d-1)$ of the group $\GL_{d}\times\cdots\times\GL_{d}$, which we view as the Levi subgroup of the upper triangular parabolic subgroup $P$ of $\GL_{n}$, and the Zelevinsky segment $\s_{n,d}(\pi^{u}(\rho'))$ attached to it, that is to say the representation
\begin{equation}\nonumber
\s_{n,d}(\pi^{u}(\rho'))=\produittenseur{i=0}{d}\left(\pi^{u}(\rho')\tenseur(|\ |^{i}\circ\det)\right)
\end{equation}
viewed as representation of $P$. Define as usual the generalized Steinberg representation $\St_{n,d}(\pi^{u}(\rho'))$ of $\GL_{n}(F_{w})$ to be the unique irreducible quotient of the normalized parabolic induction
\begin{equation}\nonumber
\nInd_{P}^{\GL_{n}}(\s_{n,d}(\pi^{u}(\rho')))=\nInd_{P}^{\GL_{n}}\produittenseur{i=0}{d}\pi^{u}(\rho')\tenseur(|\ |^{i}\circ\det).
\end{equation}
Then $\pi(\rho)$, which we also denote by $\St_{n,d}(\pi(\rho'))$, is the unique $\GL_{n}(F_{w})$-representation with coefficients in $K$ such that $\pi(\rho)\tenseur_{K}\Qbar_{p}=\St_{n,d}(\pi^{u}(\rho'))\tenseur_{\Qbar_{p}}(\cardinal{\ }\circ\det)^{\frac{1-n}{2}}$.

Finally, if $n=n_{1}+\cdots+n_{r}$ and if there exist basic indecomposable representations $\Sp_{n_{i},d_{i}}(\rho'_{i})$ such that $(\rho,N)=\sommedirecte{i=1}{r}\Sp_{n_{i},d_{i}}(\rho'_{i})$, then there exists an ordering of the $\Sp_{n_{i},d_{i}}(\rho'_{i})$ such that the list of Zelevinsky segments $(\s_{n_{i},d_{i}}(\pi(\rho'_{i})))_{1\leq i\leq r}$ satisfies the property that $\s_{n_{s},d_{s}}(\pi(\rho'_{s}))$ does not precede $\s_{n_{t},d_{t}}(\pi(\rho'_{t}))$ if $s<t$ (in the sense of \cite[Section 4]{ZelevinskyIrreducible}). We view the representation $\produittenseur{i=1}{r}\St_{n_{i},d_{i}}(\pi(\rho'_{i}))$ as a representation of the parabolic subgroup $P$ of $\GL_{n}$ with Levi component equal to $\GL_{n_{1}}\times\cdots\times\GL_{n_{r}}$. Then $\pi(\rho)$ is the unique $\GL_{n}(F_{w})$-representation with coefficients in $K$ such that $\pi(\rho)\tenseur_{K}\Qbar_{p}=\nInd_{P}^{G}\left(\produittenseur{i=1}{r}\St_{n_{i},d_{i}}(\pi^{u}(\rho'_{i}))\right)\tenseur_{\Qbar_{p}}(\cardinal{\ }\circ\det)^{\frac{1-n}{2}}$. In all of the above, the claimed unicity of $\pi(\rho)$ follows from \cite{ClozelMotifs,HenniartLanglandsBordeaux} and \cite[Lemma 4.2]{BreuilSchneider}.

Suppose now that $(\rho,N)$ has coefficients in $K$ a general field extension of $\qp$. Assume first that $(\rho,N)=\Sp_{n,d}(\rho')$, possibly with $d=1$. Then there exists a character of the Weil group $\psi:W_{F_{w}}\fleche\Kbar\croix$ such that $\rho'\tenseur\psi$ is defined over $\Qbar$ and so such that $\St_{n,d}(\pi^{u}(\rho'\tenseur\psi))$ is defined. Define $\St_{n,d}(\pi(\rho'))$ to be $\St_{n,d}(\pi^{u}(\rho'\tenseur\psi))\tenseur(\psi^{-1}\circ\det)$. Then $\St_{n,d}(\pi(\rho'))\tenseur_{\Qbar_{p}}(\cardinal{\ }\circ\det)^{\frac{1-n}{2}}$ admits a unique model over $K$ which we take as the definition of $\pi(\rho)$. If more generally $(\rho,N)$ is a direct sum $(\rho,N)=\sommedirecte{i=1}{r}\Sp_{n_{i},d_{i}}(\rho'_{i})$, then $\St_{n,d}(\pi(\rho'))$ exists and is defined over $\Qbar_{p}$ for all $i$ and $\nInd_{P}^{G}\left(\produittenseur{i=1}{r}\St_{n_{i},d_{i}}(\pi^{u}(\rho'_{i}))\right)\tenseur_{\Qbar_{p}}(\cardinal{\ }\circ\det)^{\frac{1-n}{2}}$ (with the ordering of the Zelevinsky segment satisfying the does not precede condition as above) admits a unique model over $K$ which we take as the definition of $\pi(\rho)$.
\subsubsection{A conjecture of Emerton and Helm}
Returning to our setting, for any $w\in\Sigma_{0}\backslash\Sigma$, there is a Weil-Deligne representation $(\rho(\aid)_{w},N)$ with coefficients in $\Frac(R(\aid))$ attached to the $G_{F_{w}}$-representation $V(\aid)$. Denote by $(\rho(\aid)_{w},N)^{*}$ its dual.

Let $\pi_{w}(V(\aid))$ be the admissible smooth $\G(F_{v}^{+})$-representation equal to the smooth dual of the representation $\pi((\rho(\aid)_{w},N)^{*})$ or, equivalently by the compatibility of the Local Langlands Correspondence with taking duals, the representation $\pi((\rho(\aid)_{w},N))$ (the reason we are taking double duals is that we will momentarily consider $S[\G(F_{v}^{+})]$-modules for $S$ a ring and, in that case, the dual modules are easier to handle than the most natural objects). According to \cite[Conjecture 1.4.1]{EmertonHelm}, there then exists for all $w\in\Sigma_{0}$  a (unique up to isomorphism) torsion free $\Hsmo$-module $\pi_{w}(\rho)$ with an admissible, smooth action of $\G(F_{v}^{+})$ as well as a (unique) torsion free $\Hsmo$-module $\pi_{\Sigma}(\rho)=\produittenseur{w\in\Sigma_{0}\backslash\Sigma}{}\pi_{w}(\rho)$ with an admissible, smooth action of $\produittenseur{w\in\Sigma_{0}\backslash\Sigma}{}\G(F_{v}^{+})$ satisfying the properties of \cite[Theorem 6.2.1]{EmertonHelm} (we warn the reader that our $\pi_{\Sigma}(\rho)$ would be denoted $\pi_{\Sigma_{0}\backslash\Sigma}(\rho)$ in \cite{EmertonHelm}; their choice of notation is much more logical but ours is made to conform with the conventions of the ETNC). In particular, there should exist a $\produittenseur{w\in\Sigma_{0}\backslash\Sigma}{}\G(F_{v}^{+})$-equivariant isomorphism 
\begin{equation}\label{EqPiMinimal}
\pi_{\Sigma}(\rho)\tenseur_{\Hsmo}\Frac(R(\aid))\simeq\produittenseur{w\in\Sigma_{0}\backslash\Sigma}{}\pi_{w}(V(\aid)).
\end{equation}
More generally, if $\psi:\Hsmo\fleche S$ is a characteristic zero specialization with values in a domain, the generic Local Langlands Correspondence attaches to $V_{\psi}\eqdef T_{\psi}\tenseur_{S}\Frac(S)$ a $\Frac(S)$-module $\pi_{w}(V_{\psi})$ with an admissible smooth action of $\G(F_{v}^{+})$ and there should exist an isomorphism
\begin{equation}\label{EqPiPsi}
\pi_{\Sigma}(\rho)\tenseur_{\Hsmo,\psi}\Frac(S)\simeq\pi_{\Sigma}(V_{\psi})\eqdef\produittenseur{w\in\Sigma_{0}\backslash\Sigma}{}\pi_{w}(V_{\psi})
\end{equation}
provided there exists a minimal ideal $\bid\in\Spec\Hsmo[1/p]$ such that $\rho(\bid)$ is a minimal lift of $\rho_{\psi}$ for all $w\in\Sigma_{0}\backslash\Sigma$. Here, we recall that $\rho(\bid)_{w}=\rho(\bid)|G_{F_{w}}$ admits a monodromy filtration and thus induces by specialization a filtration on $\rho_{\psi}|G_{F_{w}}$ if $\psi$ factors through $R(\bid)$. The $G_{F_{w}}$-representation $\rho(\bid)$ is a minimal lift of $\rho_{\psi}$ if the monodromy filtration on $\rho_{\psi}$ coincides with the filtration coming from the specialization of the monodromy filtration on $\rho(\bid)$. Still more generally, there should exist a torsion-free $S$-module $\pi_{w}(\rho_{\psi})$ with an admissible smooth action of $\G(F_{v}^{+})$ such that $\pi_{v}(\rho_{\psi})\tenseur_{S}\Frac(S)\simeq\pi_{v}(V_{\psi})$ and such that 
\begin{equation}\nonumber
\pi_{\Sigma}(\rho)\tenseur_{\Hsmo,\psi}S\simeq\pi_{\Sigma}(\rho_{\psi})\eqdef\produittenseur{w\in\Sigma_{0}\backslash\Sigma}{}\pi_{w}(\rho_{\psi})
\end{equation}
provided there exists a minimal ideal $\bid\in\Spec\Hsmo[1/p]$ such that $\rho(\bid)$ is a minimal lift of $\rho_{\psi}$ for all $w\in\Sigma_{0}\backslash\Sigma$.

Crucially for our purpose, note that if $\psi$ factors through both $R(\aid)$ and $R(\bid)$ and if $\rho(\bid)$ is a minimal lift of $\rho_{\psi}$ at some $w\in\Sigma_{0}\backslash\Sigma$ whereas $\rho(\aid)$ is not a minimal lift of $\rho_{\psi}$ at $w\in\Sigma_{0}\backslash\Sigma$, then the requirements \eqref{EqPiMinimal} and \eqref{EqPiPsi} entail that $\pi_{w}(\rho(\aid))\tenseur_{R(\aid),\psi}\Frac(S)$ is not isomorphic to $\pi_{w}(V_{\psi})$. Indeed, \cite[Theorem 7.6]{SahaPreprint} shows that if $\pi_{w}(\rho(\aid))\tenseur_{R(\aid),\psi}\Frac(S)\simeq\pi_{w}(V_{\psi})$, then the monodromy filtration on $V_{\psi}|G_{F_{w}}$ coincides with the specialization of the monodromy filtration on $V(\aid)$, and hence $\rho(\aid)$ is a minimal lift of $\rho_{\psi}$ at $w$. 

In addition to the representations $\pi_{w}(\rho_{\psi})$ for $w\in\Sigma_{0}\backslash\Sigma$, there should exist for \emph{any} specialization $\psi:\Hsmo\fleche S$ an admissible smooth $\G(F_{v}^{+})$-representation $\pi_{w}(\rho_{\psi})$ realizing the $p$-adic Local Langlands for $w|p$ and the formation of $\pi_{w}(\rho\tenseur-)$ should commute with arbitrary specialization (see \cite[Section 3.2.2]{EmertonICM}). We denote by $\pi_{p}(\rho)$ the tensor product $\produittenseur{w|p}{}\pi_{w}(\rho)$. Not much is known, even conjecturally, about the $\pi_{w}(\rho)$ except when $n\leq2$. As we see below, this is not (too) detrimental to the statement of our revised conjecture: the compatibility of $\pi_{p}$ with specializations, which is analogous to the apparition of $\Spec\Ocal_{F}[1/p]$ in $\Xcali$, guarantees that $\pi_{p}$ does not play any role in the behavior of fundamental lines under specializations. 
\subsection{Conjectures}
\subsubsection{Local Euler factors}
Let $\psi:\Hsmo\fleche S$ be a characteristic zero specialization with values in a discrete valuation ring factoring through $R(\aid)=\Hsmo/\aid$ for some minimal prime ideal $\aid$. By \cite[Theorem 1.5.1 and 3.3.2]{EmertonHelm}, there exists an admissible smooth $S[\GL_{m}(F_{w})]$-module $\pi(\rho_{\psi})$ and there is a surjection of ${\pi}({\rhobar})$ onto $\pi(\rho_{\psi})/\varpi_{S}\pi(\rho_{\psi})$. By \cite[Proposition 4.13]{HelmBernstein}, all such $\pi(\rho_{\psi})$ have the same mod $p$ inertial supercuspidal support $(M,\s)$ (here $M$ is a Levi subgroup and $\s$ is a supercuspidal automorphic representation of $M$ over $\kbar$).

Denote by $W(\kbar)$ the ring of Witt vectors of the algebraic closure of the residual field $k$ of $\Hsmo$. The full subcategory of smooth $W(\kbar)[\GL_{m}(F_{w})]$-modules every simple subquotient of which has mod $p$ inertial supercuspidal support given by $(M,\s)$ is a block of the category of smooth $W(\kbar)[\GL_{m}(F_{w})]$-modules. We denote this block by $\Bcal_{M,\s}$. It contains by construction all the $\pi(\rho_{\psi})$ for $\psi$ as above. Let $\Zcal(\Bcal_{M,\s})$ be the integral Bernstein center of $\Bcal_{M,\s}$ in the sense of \cite[Corollary 9.19]{HelmBernstein} and, for any field $\Kcal$, let $\Kcal[\Bcal_{M,\s}]$ be the affine coordinate ring of $\Bcal_{M,\s}$. If $\pi$ is a an irreducible, smooth $\GL_{m}(F_{w})$-representation, then there is a character $\zid(\pi)$ of $\Zcal(\Bcal_{M,\s})$ attached to $\psi$ by considering the action of $\Zcal(\Bcal_{M,\s})$ on $\pi$. 

By \cite[Proposition 3.11]{ChenevierApplication} (or more precisely, by its proof), there exists a unique determinant law
\begin{equation}\nonumber
D_{M,\s}:W(\kbar)[G_{F_{w}}]\fleche\Frac(W(\kbar))[\Bcal_{M,\s}]
\end{equation}
such that $D_{M,\s}(-)(\pi(\rho_{\psi}))$ is equal to $D_{\psi}(-)=\psi\circ D(\aid)(-)$ for all $\psi$ with values in discrete valuation rings as above and such that $\pi(\rho_{\psi})$ is irreducible (note that the equality makes sense as $D_{M,\s}(-)$ is a function on $\Bcal_{M,\s}$). According to \cite{BernsteinCentre}, $D_{M,\s}(-)(\pi(\rho_{\psi}))$ depends only on $\zid(\pi(\rho_{\psi}))$.

The following conjecture predicts that local $L$-factors can be computed using elements of $\Zcal[\Bcal_{M,\s}]$.
\begin{Conj}\label{ConjCenter}
Let $m\geq1$ be an integer. Let $\Bcal^{(m)}_{M,\s}$ be the block of the category of smooth $W(\kbar)[\GL_{m}(F_{w})]$-modules corresponding to the mod p inertial supercuspidal support type $(M,\s)$. Then the image of $1-\Fr(w)$ through $D_{M,\s}$ belongs to $\Zcal(\Bcal^{(m)}_{M,\s})$.
\end{Conj}
In \cite[Conjecture 7.6]{HelmWhittaker}, D.Helm formulates a general conjecture linking the integral Bernstein center and framed universal deformation rings of $G_{F_{w}}$-representations which implies both \cite[Conjecture 1.4.1]{EmertonHelm} and conjecture \ref{ConjCenter}.
\begin{Prop}\label{PropCenterInduction}
Assume conjecture \ref{ConjCenter} and denote $D_{M,\s}(1-\Fr(w))$ by $\zid(\Bcal^{(m)}_{M,\s})_{w}$. Then
\begin{equation}\label{EqFactorL}
\zid(\Bcal^{(m)}_{M,\s})_{w}\cdot\pi(\rho_{\psi})=\det(1-\rho_{\psi}(\Fr(w)))\pi(\rho_{\psi})
\end{equation}
for all characteristic zero specialization $\psi:\Hsmo\fleche S$ with values in a domain.
\end{Prop}
\begin{proof}
It is enough to prove the more general result that
\begin{equation}\label{EqLLCpi}
\zid(\Bcal^{(n)}_{M,\s})_{w}\cdot\pi=\det(1-\rho_{\pi}(\Fr(w)))\pi
\end{equation}
for all $n\geq1$, all modulo $p$ inertial supercuspidal support $M,\s$ and all $\pi\in\Bcal_{M,\s}^{(n)}$ which can be obtained by the generic Local Langlands Correspondence (here $\rho_{\pi}$ is the $G_{F_{w}}$-representation corresponding to $\pi$ through the Local Langlands Correspondence).

Suppose first that $\pi$ is not irreducible. By extending scalars to the fraction field of $S$ and then descending to $\Qbar_{p}$ as in the end of section \ref{SubBreuilSchneider}, we may and do assume in addition that $\pi$ has coefficients in $\Qbar_{p}$. As recalled in subsection \ref{SubBreuilSchneider}, there then exist integers $n_{i}$ such that $n_{1}+\cdots+n_{r}=n$, generalized Steinberg representations $\St_{n_{i},d_{i}}(\pi^{u}(\rho'_{i}))$ of $M_{i}=\GL_{n_{i}}(F_{w})$ and a parabolic subgroup $P$ with Levi component $M_{1}\times\cdots\times M_{r}$ such that
\begin{equation}\nonumber
\pi\simeq\nInd_{P}^{G}\left(\produittenseur{i=1}{r}\St_{n_{i},d_{i}}(\pi^{u}(\rho'_{i}))\right)\tenseur_{\Qbar_{p}}(\cardinal{\ }\circ\det)^{\frac{1-n}{2}}.
\end{equation}
The representation $\s$ is then a tensor product $\s=\produittenseur{i=1}{r}\s_{i}$ with $\s_{i}$ a supercuspidal representation of $M_{i}$. By \cite[Theorem 10.3]{HelmBernstein}, the action of $\Zcal(\Bcal^{(n)}_{M,\s})_{w}$ on $\pi$ then factors through $\produittenseur{i=1}{r}\Zcal(\Bcal^{(n_{i})}_{M_{i},\s_{i}})_{w}$.

As $L$-factors are multiplicative with respect to normalized parabolic induction (\cite[Theorem 3.4]{GodementJacquet}), the statement \eqref{EqLLCpi} for $\pi$ follows from its counterpart for each $\St_{n_{i},d_{i}}(\pi^{u}(\rho'_{i}))$. Hence, it is enough to treat the case of irreducible $\pi$.

If $\pi$ is irreducible, then \eqref{EqLLCpi} follows from the construction of $\zid(\Bcal^{(m)}_{M,\s})_{w}$ in \cite[Proposition 3.11]{ChenevierApplication}.
\end{proof}
One of the characteristic properties of $\pi(\rho_{\psi})$ (assuming it exists) is that $\End_{S[\GL_{m}(F_{w})]}(\pi(\rho_{\psi}))=S$. Combining this with conjecture \ref{ConjCenter} defines local automorphic $\Lcali$-factors to the automorphic representations $\pi(\rho_{\psi})$ with $\psi$ a characteristic zero specialization of $\Hsmo$.
\begin{DefEnglish}
Let $\psi:\Hsmo\fleche S$ be a characteristic zero specialization. The automorphic $\Lcali$-factor $\Lcali(\pi(\rho_{\psi}))$
\begin{equation}\nonumber
\Lcali(\pi(\rho_{\psi}))=\zid\left(\Bcal^{(m)}_{M,\s}\right)_{w}(\pi(\rho_{\psi}))\in S
\end{equation}
is the endomorphism $\zid\left(\Bcal^{(m)}_{M,\s}\right)_{w}$ of $\pi(\rho_{\psi})$.
\end{DefEnglish}
\paragraph{Remark:}The automorphic $\Lcali$-factor behaves just like the image of the determinant of the complex of $S[\GL_{m}(F_{w})]$-modules
\begin{equation}\nonumber
\left[\zid\left(\Bcal^{(m)}_{M,\s}\right)_{w}\pi{\fleche}\pi\right]
\end{equation}
in degree 0,1 through the canonical trivialization induced by the identity between $\pi$ in degree 0 and 1; if such a determinant were known to be well-defined. One way to make sense of this determinant could be as follows. By construction, $\pi$ is a co-Whittaker $S[\GL_{m}(F_{w})]$-module (in the sense of \cite[Definition 6.1]{HelmWhittaker}) so the $m$-th  Bernstein derivative $\pi^{(m)}$ of $\pi$ is defined and is an $S$-module free of rank 1. Hence, the complex $\left[\zid\left(\Bcal^{(m)}_{M,\s}\right)_{w}\pi^{(m)}{\fleche}\pi^{(m)}\right]$ is a perfect complex of $S$-modules and we could define
\begin{equation}\nonumber
\Det_{S[\GL_{m}F_{w}]}\left[\zid\left(\Bcal^{(m)}_{M,\s}\right)_{w}\pi{\fleche}\pi\right]\eqdef\Det_{S}\left[\zid\left(\Bcal^{(m)}_{M,\s}\right)_{w}\pi^{(m)}{\fleche}\pi^{(m)}\right].
\end{equation}
\subsubsection{Global conjecture}
Let $(T,\rho,\Lambda)$ be the $G_{F,\Sigma_{0}}$-representation attached to a characteristic zero specialization of $\Hsmo$. We assume that $T$ satisfies assumption \ref{HypMain} (this is for instance the case if $T$ os attached to the identity specialization) and recall that, in that case, the algebraic part
\begin{equation}\nonumber
\Xcali_{\Sigma}(T/F)\eqdef\Det^{-1}_{\Lambda}\RGamma_{\et}(\Ocal_{F}[1/\Sigma_{0}],T)\tenseur\produittenseur{v\in\Sigma_{0}\backslash\Sigma}{}\left(\Det_{\Lambda}\RGamma_{\et}(F_{v},T)\tenseur_{\Lambda}\Det^{-1}_{\Lambda}\frac{\Lambda}{\Eul_{v}(T,1)\Lambda}\right)
\end{equation}
of the $p$-adic fundamental line is well-defined. We assume \cite[Conjecture 1.4.1]{EmertonHelm} as well as conjecture \ref{ConjCenter}. We recall that 
\begin{equation}\nonumber
\pi_{\Sigma}(\rho)=\produittenseur{w\in\Sigma_{0}\backslash\Sigma}{}\pi_{w}(\rho)
\end{equation}
and write $\pi_{\Sigma}(\rho_{\psi})^{(m)}$ for the rank 1, free $\Lambda$-module
\begin{equation}\nonumber
\pi_{\Sigma}(\rho_{\psi})^{(m)}=\produittenseur{w\in\Sigma_{0}\backslash\Sigma}{}\pi_{w}(\rho)^{(m)}.
\end{equation}
\begin{DefEnglish}
The automorphic part of the $p$-adic fundamental line is the rank 1, free $\Lambda$-module
\begin{equation}\nonumber
\Ycali_{\Sigma}(T/F)=\Det_{\Lambda}T(-1)^{+}\tenseur_{\Lambda}\pi_{\Sigma}(\rho)^{(m)}.
\end{equation}
The full $p$-adic fundamental line is the rank 1, free $\Lambda$-modules
\begin{equation}\nonumber
\Deltatilde_{\Sigma}(T/F)\eqdef\Xcali_{\Sigma}(T/F)\tenseur_{\Lambda}\Ycali_{\Sigma}(T/F)^{-1}.
\end{equation}
\end{DefEnglish}
The following proposition verifies the compatibility of $\Deltatilde$ with arbitrary characteristic zero specialization of $\Hsmo$ (non only the pure ones).
\begin{Prop}\label{PropDeltaTilde}
If $\psi:\Lambda\fleche S$ is a characteristic zero specialization such that 1 is an eigenvalue of $\Fr(w)$ in its action on $T_{\psi}$ for no $w\in\Sigma_{0}\backslash\Sigma$, then there is a canonical isomorphism
\begin{equation}\nonumber
\Deltatilde_{\Sigma}(T/F)\tenseur_{\Lambda,\psi}S\isocan\Deltatilde_{\Sigma}(T\tenseur_{\Lambda,\psi}S/F).
\end{equation}
\end{Prop}
\begin{proof}
Assume first that $\rho_{\psi}$ is in addition monodromic and pure at all $w\in\Sigma_{0}\backslash\Sigma$. Then there are canonical isomorphisms
\begin{equation}\nonumber
\Xcali_{\Sigma}(T/F)\tenseur_{\Lambda,\psi}S\isocan\Xcali_{\Sigma}(T\tenseur_{\Lambda,\psi}S/F),\ \left(\Det_{\Lambda}T(-1)^{+}\right)\tenseur_{\Lambda,\psi}S\left(\Det_{S}T(-1)\right)
\end{equation}
and
\begin{equation}\nonumber
\pi_{w}(\rho)\tenseur_{\Lambda,\psi}S\isocan\pi_{w}(\rho_{\psi})
\end{equation}
for all $w\in\Sigma_{0}\backslash\Sigma$. Combining them yields $\Deltatilde_{\Lambda}(T/F)\tenseur_{\Lambda,\psi}S\isocan\Deltatilde(T\tenseur_{\Lambda,\psi}S/F)$.

Now we assume that there exists a place $w\in\Sigma_{0}\backslash\Sigma$ such that $\rho_{\psi}$ is not pure. In that case, denote by $ZS(\rho)$ and $ZS(\rho_{\psi})$ the Zelevinsky segments respectively attached to $\pi(\rho)_{w}\tenseur_{\Lambda,\psi}\Frac(S)$ and to $\pi(\rho_{\psi})_{w}\tenseur_{S}\Frac(S)$. By construction, $ZS(\rho)$ can be obtained by repeatedly replacing pair of segments by their unions so that $\pi(\rho)_{w}\tenseur_{\Lambda,\psi}\Frac(S)$ embeds into $\pi(\rho_{\psi})_{w}\tenseur_{S}\Frac(S)$ by \cite[Theorem 7.1]{ZelevinskyIrreducible} and \cite[Proposition 4.3.6]{EmertonHelm}. This embedding realizes $\pi_{w}(\rho)^{(m)}\tenseur_{\Lambda,\psi}S$ and $\pi_{w}(\rho_{\psi})^{(m)}$ as sub-$S$-modules inside $\pi_{w}(\rho_{\psi})^{(m)}\tenseur_{S}\Frac(S)$. As $\zid\left(\Bcal^{(m)}_{M,\s}\right)_{w}$ acts naturally on $\pi_{w}(\rho)^{(m)}\tenseur_{\Lambda,\psi}S$ by multiplication by $\psi\left(\Lcali(\pi(\rho)_{w})\right)$ and on $\pi_{w}(\rho_{\psi})^{(m)}$ by multiplication by $\Lcali(\pi_{w}(\rho_{\psi}))$, the index of $\pi_{w}(\rho)^{(m)}\tenseur_{\Lambda,\psi}S$ inside $\pi_{w}(\rho_{\psi})^{(m)}$ is $\psi\left(\Lcali(\pi_{w}(\rho))\right)\tenseur_{S}\Lcali(\pi_{w}(\rho_{\psi}))^{-1}$. As the contribution of $w$ to $\Xcali_{\Sigma}(T)\tenseur_{\Lambda,\psi}S\tenseur\Xcali_{\Sigma}(T_{\psi})^{-1}$ is
\begin{equation}\nonumber
\frac{\psi(\Eul(T,1))}{\Eul(T_{\psi},1)}S
\end{equation}
and the Local Langlands Correspondence matches $L$-factors with Euler factors, there is a canonical identification of the contribution of $w$ to $\Deltatilde_{\Sigma}(T)\tenseur_{\Lambda,\psi}S\tenseur\Deltatilde(T_{\psi})^{-1}$ with $S$. Putting together the contributions of all the places $w\in\Sigma_{0}\backslash\Sigma$, we obtain the desired canonical isomoprhism $\Deltatilde(T)\tenseur_{\Lambda,\psi}S\isocan\Deltatilde(T_{\psi})$.
\end{proof}
Just like the algebraic part of the $p$-adic fundamental line is the determinant of a natural object and as we hinted in the introduction, it seems likely that the automorphic part of the $p$-adic fundamental line is the suitably defined determinant of a natural object. Interpreting the top Bernstein derivative as a non-commutative determinant functor in the category of $\Lambda[\GL_{m}(F_{w})]$-module, we see that (up to the term $\produittenseur{w|p}{}\pi(\rho)_{w}$) $\Ycali_{\Sigma}$ should be the determinant of the completed cohomology group $\Htildetilde^{d}(U^{p},\Ocal)_{\mgot}$ according to the local-global compatibility conjecture in the $p$-adic Langlands program.
\begin{Conj}\label{ConjRevised}
Let $\Sigma\subset\Sigma_{0}$ be a set of finite places containing $\{w|p\}$ and let $K$ be an abelian extension of $F$ which is unramified outside $\Sigma$. The $p$-adic fundamental line $\Deltatilde_{\Sigma}(T/K)$ of the $G_{F,\Sigma_{0}}$-representation $(T,\rho,\Hsmo)$ admits a basis $\tilde{\z}_{\Sigma}$ satisfying the following property.
Let $\psi:\Hsmo\fleche S$ be a characteristic zero motivic specialization attached to a motive $M$ pure of non-zero weight.
Then the canonical isomorphism
\begin{equation}\label{EqValeur}
\Deltatilde_{\Sigma}(T/K)\tenseur_{\Hsmo,\psi}S\isocan\Delta_{\Sigma}(T_{\psi}/K)
\end{equation}
obtained by composing the isomorphism of of proposition \ref{PropDeltaTilde} with an identification of $\pi_{\Sigma}(\rho_{\psi_{x}})$ with $S$ sends $\tilde{\z}_{\Sigma}(T/K)\tenseur1$ to the $p$-adic zeta element $\z_{\Sigma}(M_{K})\tenseur1$ of $M$.
\end{Conj}
We note that by proposition \ref{PropDeltaTilde}, the compatibility of the conjecture \ref{ConjRevised} with arbitrary characteristic zero specialization is known.
\subsubsection{The case of $\GL_{2}$}
As evidence for conjecture \ref{ConjRevised}, we mention the following theorem, which shows a weaker form of the conjecture often holds for motives of modular forms. We assume that $\G$ is the reductive group $\GL_{2}$ over $\Q$.
\begin{TheoEnglish}\label{TheoExemple}
Assume that $\rhobar$ satisfies the following hypotheses.
\begin{enumerate}
\item Let $p^{*}$ be $(-1)^{(p-1)/2}p$. Then $\rhobar|_{G_{\Q(\sqrt{p^{*}})}}$ is irreducible.
\item The semisimplification of $\rhobar|_{G_{\qp}}$ is not scalar.
\suspend{enumerate}
Assume moreover that $\rhobar$ satisfies at least one of the following two conditions.
\resume{enumerate}
\item There exists $\ell$ dividing exactly once $N(\rhobar)$, the representation $\rhobar|_{G_{\qp}}$ is reducible and $\det\rhobar$ is unramified outside $p$.
\item There exists $\ell$ dividing exactly once $N(\rhobar)$ and there exists an eigencuspform $f\in S_{2}(\Gamma_{0}(N))$ attached to a classical $x\in\Spec\Hs[1/p]$ with rational coefficients, zero eigenvalue at $p$ and such that $N$ is square-free.
\end{enumerate}
Then $\Deltatilde_{\Sigma}(T/\Q)$ admits a basis $\tilde{\z}_{\Sigma}(T/\Q)$ such that the isomorphism \eqref{EqValeur} sends $\tilde{\z}_{\Sigma}(T/\Q)$ to the $p$-adic zeta element $\z_{\Sigma}(M)\tenseur1$ provided there exists an eigencuspform $f\in S_{k}$ such that $L_{\Sigma}(M^{*}(1),0)=L_{\Sigma}(f,\chi,s)\neq0$ with $\chi$ a finite order character of $\Q(\zeta_{p^{\infty}})$ and $1\leq s\leq k-1$ is an integer.
\end{TheoEnglish}
Contrary to the full conjecture, theorem \ref{TheoExemple} predicts the special values of a certain subset of motivic points (namely the critical motives with non-zero $L$-function).
\begin{proof}
In \cite{EmertonHelm,HelmWhittaker}, D.Helm has announced a proof of \cite[Conjecture 1.4.1]{EmertonHelm} and of conjecture \ref{ConjCenter} for $\GL_{2}$. The construction of $\tilde{\z}_{\Sigma}(T/\Q)$, the proof that it is a basis of $\Deltatilde_{\Sigma}(T/\Q)$ and the proof that it is satisfies \eqref{EqValeur} at the $\psi$ like in the theorem is the object of \cite{HeckeETNC} (see especially section 3.3 and 4.1 thereof).
\end{proof}
\bibliographystyle{plain}
\def\Dbar{\leavevmode\lower.6ex\hbox to 0pt{\hskip-.23ex \accent"16\hss}D}
  \def\cfac#1{\ifmmode\setbox7\hbox{$\accent"5E#1$}\else
  \setbox7\hbox{\accent"5E#1}\penalty 10000\relax\fi\raise 1\ht7
  \hbox{\lower1.15ex\hbox to 1\wd7{\hss\accent"13\hss}}\penalty 10000
  \hskip-1\wd7\penalty 10000\box7}
  \def\cftil#1{\ifmmode\setbox7\hbox{$\accent"5E#1$}\else
  \setbox7\hbox{\accent"5E#1}\penalty 10000\relax\fi\raise 1\ht7
  \hbox{\lower1.15ex\hbox to 1\wd7{\hss\accent"7E\hss}}\penalty 10000
  \hskip-1\wd7\penalty 10000\box7} \def\Dbar{\leavevmode\lower.6ex\hbox to
  0pt{\hskip-.23ex \accent"16\hss}D}
  \def\cfac#1{\ifmmode\setbox7\hbox{$\accent"5E#1$}\else
  \setbox7\hbox{\accent"5E#1}\penalty 10000\relax\fi\raise 1\ht7
  \hbox{\lower1.15ex\hbox to 1\wd7{\hss\accent"13\hss}}\penalty 10000
  \hskip-1\wd7\penalty 10000\box7}
  \def\cftil#1{\ifmmode\setbox7\hbox{$\accent"5E#1$}\else
  \setbox7\hbox{\accent"5E#1}\penalty 10000\relax\fi\raise 1\ht7
  \hbox{\lower1.15ex\hbox to 1\wd7{\hss\accent"7E\hss}}\penalty 10000
  \hskip-1\wd7\penalty 10000\box7} \def\Dbar{\leavevmode\lower.6ex\hbox to
  0pt{\hskip-.23ex \accent"16\hss}D}
  \def\cfac#1{\ifmmode\setbox7\hbox{$\accent"5E#1$}\else
  \setbox7\hbox{\accent"5E#1}\penalty 10000\relax\fi\raise 1\ht7
  \hbox{\lower1.15ex\hbox to 1\wd7{\hss\accent"13\hss}}\penalty 10000
  \hskip-1\wd7\penalty 10000\box7}
  \def\cftil#1{\ifmmode\setbox7\hbox{$\accent"5E#1$}\else
  \setbox7\hbox{\accent"5E#1}\penalty 10000\relax\fi\raise 1\ht7
  \hbox{\lower1.15ex\hbox to 1\wd7{\hss\accent"7E\hss}}\penalty 10000
  \hskip-1\wd7\penalty 10000\box7} \def\Dbar{\leavevmode\lower.6ex\hbox to
  0pt{\hskip-.23ex \accent"16\hss}D}
  \def\cfac#1{\ifmmode\setbox7\hbox{$\accent"5E#1$}\else
  \setbox7\hbox{\accent"5E#1}\penalty 10000\relax\fi\raise 1\ht7
  \hbox{\lower1.15ex\hbox to 1\wd7{\hss\accent"13\hss}}\penalty 10000
  \hskip-1\wd7\penalty 10000\box7}
  \def\cftil#1{\ifmmode\setbox7\hbox{$\accent"5E#1$}\else
  \setbox7\hbox{\accent"5E#1}\penalty 10000\relax\fi\raise 1\ht7
  \hbox{\lower1.15ex\hbox to 1\wd7{\hss\accent"7E\hss}}\penalty 10000
  \hskip-1\wd7\penalty 10000\box7} \def\Dbar{\leavevmode\lower.6ex\hbox to
  0pt{\hskip-.23ex \accent"16\hss}D}
  \def\cfac#1{\ifmmode\setbox7\hbox{$\accent"5E#1$}\else
  \setbox7\hbox{\accent"5E#1}\penalty 10000\relax\fi\raise 1\ht7
  \hbox{\lower1.15ex\hbox to 1\wd7{\hss\accent"13\hss}}\penalty 10000
  \hskip-1\wd7\penalty 10000\box7}
  \def\cftil#1{\ifmmode\setbox7\hbox{$\accent"5E#1$}\else
  \setbox7\hbox{\accent"5E#1}\penalty 10000\relax\fi\raise 1\ht7
  \hbox{\lower1.15ex\hbox to 1\wd7{\hss\accent"7E\hss}}\penalty 10000
  \hskip-1\wd7\penalty 10000\box7} \def\cprime{$'$}
  \def\cftil#1{\ifmmode\setbox7\hbox{$\accent"5E#1$}\else
  \setbox7\hbox{\accent"5E#1}\penalty 10000\relax\fi\raise 1\ht7
  \hbox{\lower1.15ex\hbox to 1\wd7{\hss\accent"7E\hss}}\penalty 10000
  \hskip-1\wd7\penalty 10000\box7}

\end{document}